\documentclass[11pt]{article}

\usepackage[hmargin=2.2cm,vmargin=2cm,a4paper]{geometry} 
\usepackage{amssymb,amsmath,amsthm}
\usepackage[shortlabels]{enumitem}
\usepackage{mathabx}
\usepackage{dsfont}
\usepackage{upgreek}
\usepackage[title]{appendix}

\usepackage[backend=bibtex,style=numeric,
doi=false,isbn=false,url=false,eprint=false,maxbibnames=5]{biblatex}
\bibliography{library.bib}

\usepackage{newunicodechar}
\usepackage{todonotes}
\numberwithin{equation}{section}

\usepackage{xparse}
\usepackage{color}
\definecolor{Darkgreen}{rgb}{0,0.4,0}
\definecolor{viola}{rgb}{0.7,0,1}
\usepackage[colorlinks,linkcolor=Darkgreen,citecolor=Darkgreen,breaklinks=true,linktocpage]{hyperref}

\NewDocumentCommand{\set}{m o}{\ensuremath{\left\lbrace\IfNoValueTF{#2}{#1}{#1 \: \left| \:   #2\right.}\right\rbrace }}

\numberwithin{equation}{section}
\theoremstyle{plain}
\newtheorem{deff}{Definition}[section]

\newtheorem{thm}[deff]{Theorem}
\newtheorem{prop}[deff]{Proposition}
\newtheorem{lemma}[deff]{Lemma}

\theoremstyle{remark}
\newtheorem{remark}[deff]{Remark}

\DeclareMathOperator{\capac}{cap}
\DeclareMathOperator{\Poi}{Poi}

\newcommand{\N}{\ensuremath{\mathbb{N}}}
\newcommand{\R}{\ensuremath{\mathbb{R}}}
\newcommand{\Z}{\ensuremath{\mathbb{Z}}}
\newcommand{\ind}{\boldsymbol{1}}

\newcommand{\given}{\: \big| \:}											
\newcommand{\Given}{\: \Big| \:}										
\newcommand{\as}{\text{a.s.}}
\newcommand{\barAS}{\text{-a.s.}}
\renewcommand{\tilde}{\widetilde}
\newcommand{\ab}[1]{\left| #1 \right|}				

\RenewDocumentCommand{\P}{s   O{}   m   G{} }				
{	\IfBooleanTF{#1}	{\mathcal{P}^{#4}_{#2} 		\ifblank{#3}{ }{\! \left(	#3	\right)}			}
	{\mathbb{P}^{#4}_{#2}  		\ifblank{#3}{ }{\! \left(	#3	\right)}			}
}
\newcommand{\Pgw}{\mathbb{P}^{{\rm GW}}}
\newcommand{\Egw}{\mathbb{E}^{{\rm GW}}}

\newcommand{\Pgff}[2]{\mathbb{P}^{{\rm G}}_{#1}\ifblank{#2}{}{\left(#2\right)}}
\newcommand{\Egff}[2]{\mathbb{E}^{{\rm G}}_{#1} \left[#2\right]}	

\newcommand{\Pri}{ \mathbb{P}^{{\rm RI}}}

\newlength{\leftstackrelawd}
\newlength{\leftstackrelbwd}
\def\leftstackrel#1#2{\settowidth{\leftstackrelawd}
	{${{}^{#1}}$}\settowidth{\leftstackrelbwd}{$#2$}
	\addtolength{\leftstackrelawd}{-\leftstackrelbwd}
	\leavevmode\ifthenelse{\lengthtest{\leftstackrelawd>0pt}}
	{\kern-.5\leftstackrelawd}{}\mathrel{\mathop{#2}\limits^{#1}}}

\newunicodechar{∂}{\ensuremath{\partial}}
\newunicodechar{∞}{\ensuremath{\infty}}
\newunicodechar{∀}{\ensuremath{\forall}}
\newunicodechar{∃}{\ensuremath{\exists}}
\newunicodechar{∇}{\ensuremath{\nabla}}
\newunicodechar{α}{\ensuremath{\alpha}}
\newunicodechar{σ}{\ensuremath{\sigma}}
\newunicodechar{δ}{\ensuremath{\delta}}
\newunicodechar{φ}{\ensuremath{\varphi}}
\newunicodechar{γ}{\ensuremath{\gamma}}
\newunicodechar{η}{\ensuremath{\eta}}
\newunicodechar{ξ}{\ensuremath{\xi}}
\newunicodechar{κ}{\ensuremath{\kappa}}
\newunicodechar{λ}{\ensuremath{\lambda}}
\newunicodechar{ε}{\ensuremath{\varepsilon}}
\newunicodechar{ρ}{\ensuremath{\rho}}
\newunicodechar{υ}{\ensuremath{\upsilon}}
\newunicodechar{θ}{\ensuremath{\theta}}
\newunicodechar{ι}{\ensuremath{\iota}}
\newunicodechar{π}{\ensuremath{\pi}}
\newunicodechar{ζ}{\ensuremath{\zeta}}
\newunicodechar{χ}{\ensuremath{\chi}}
\newunicodechar{ψ}{\ensuremath{\psi}}
\newunicodechar{ω}{\ensuremath{\omega}}
\newunicodechar{β}{\ensuremath{\beta}}
\newunicodechar{ν}{\ensuremath{\nu}}
\newunicodechar{μ}{\ensuremath{\mu}}
\newunicodechar{Σ}{\ensuremath{\Sigma}}
\newunicodechar{Δ}{\ensuremath{\Delta}}
\newunicodechar{Φ}{\ensuremath{\Phi}}
\newunicodechar{Γ}{\ensuremath{\Gamma}}
\newunicodechar{Ξ}{\ensuremath{\Xi}}
\newunicodechar{Λ}{\ensuremath{\Lambda}}
\newunicodechar{Θ}{\ensuremath{\Theta}}
\newunicodechar{Π}{\ensuremath{\Pi}}
\newunicodechar{Ψ}{\ensuremath{\Psi}}
\newunicodechar{Ω}{\ensuremath{\Omega}}

\begin{document}
	
	\title{Generating Galton--Watson trees using random walks and percolation for the Gaussian free field}
	
	\author{		
		Alexander Drewitz
		\thanks{Universität zu Köln,
			Department of Mathematics and Computer Science,
			Weyertal 86--90,
			50931 Köln, Germany.
			Emails: \url{adrewitz@uni-koeln.de},
			\url{ggallo@uni-koeln.de}}
		\and
		Gioele Gallo $ ^* $
		\and 
		Alexis Prévost
		\thanks{University of Geneva,
			Section of Mathematics,
			24, rue du Général Dufour,
			1211 Geneva, Switzerland.
			Email: \url{alexis.prevost@unige.ch}}
	}
	
	\date{\today}
	\maketitle

\begin{abstract}
The study of Gaussian free field level sets on supercritical Galton--Watson  trees has been initiated by Ab\"acherli and Sznitman in 2018. By means of entirely different tools, we continue this investigation  and generalize their main result on the positivity of the associated percolation critical parameter $h_*$ to the setting of arbitrary supercritical offspring distribution and random conductances. In our setting, this establishes a rigorous proof of the physics literature mantra that positive correlations facilitate percolation when compared to the independent case. Our proof proceeds by constructing the Galton--Watson tree through an exploration via finite random walk trajectories. This exploration of the tree progressively unveils an infinite connected component in the  random interlacements set on the tree, which is stable under small quenched noise. Using a Dynkin-type isomorphism theorem, we then infer the strict positivity of the critical parameter $ h_* .$ As a byproduct, we obtain transience results for the above-mentioned sets.
\end{abstract}

\section{Introduction}
The main subject of this article is the study of level set percolation for the Gaussian free field  on supercritical Galton--Watson trees. Due to the strong correlations inherent to the model, the problem of level set percolation induced by the Gaussian free field is quite intricate and significantly harder to understand than that of Bernoulli percolation. In the setting of fairly general transient graphs, the model has received increased attention in the last decade, as it is an important showcase for percolation problems with long-range correlations. A fundamental question in this context is to show the positivity of the associated critical parameter $h_*$ -- see \eqref{DEF-h_*} below for its definition -- which entails a coexistence phase for $h>0$ close to zero. It has been investigated on $\Z^d,$ $d\geq3,$ in \cite{BricmontLebowitzMaes87,RodriguezSznitman2013,DrewitzPrevostRodriguez2018-1}, and on more general graphs with polynomial growth in \cite{DrewitzPrevostRodriguez2018-2}. Of particular relevance for us is the setting of the Gaussian free field on trees, which has been studied in \cite{Sznitman2016,AbacherliSznitman2018,AbacherliCerny2019}. More precisely, in \cite[Section~5]{AbacherliSznitman2018}, Ab\"acherli and Sznitman consider the particular case of the Gaussian free field on supercritical Galton--Watson trees with mean offspring distribution  $m\in{(1,\infty)},$ and prove that $h_*\in{[0,\infty)}$ for all $m\in{(1,\infty)},$ as well as the strict inequality $h_*>0$ when $m>2.$

The main goal of the current article is to extend this result $h_*>0$ to all supercritical Galton--Watson trees, i.e.\,with offspring mean $m\in{(1,\infty)},$ which along the way solves an open question of \cite[Remark~5.6]{AbacherliSznitman2018}. Moreover, we additionally allow the edges of the tree to be equipped with random conductances with finite mean, and show that the associated critical parameter $ h_*$ is still deterministic and strictly positive.

It is intriguing to compare our main result with Bernoulli site percolation on supercritical Galton--Watson trees $\mathcal T,$ for which -- conditioned on survival -- the associated critical parameter is known to almost surely equal the inverse of the offspring mean, i.e., $p_c(\mathcal T) = 1/m;$ see \cite{Lyons90} or \cite[Proposition~5.9]{LyonsPeres17}. Contrasting this well-known result with the inequality $h_*(\mathcal T) >0$ is particularly interesting in the newly investigated range $m \in (1,2]$ in our article. Indeed, in this range we have that the density of Bernoulli percolation at the critical parameter is given by $p_c(\mathcal T) = 1/m \ge 1/2,$ whereas the density of percolation for the Gaussian free field level sets at the critical parameter is strictly smaller than $1/2,$ since $h_*(\mathcal T) >0.$ Therefore, when $m \in (1,2]$  the positive correlations of the Gaussian free field make percolation easier.
This is a behavior expected for many percolation models, see in particular  \cite{PhysRevA.46.R1724} as well as \cite{Marinov2006PercolationIT} for numerical reasonings concerning the setting of percolation with long-range correlations.
To the best of our knowledge, the only other class of transient graphs where an inequality between densities at criticality of Gaussian free field and independent percolation has been rigorously proven are $d$-regular trees, see \cite[Corollary~4.5]{Sznitman2016}, but it is conjectured to hold for a large class of transient graphs.

A key tool in our proof is based on a construction of the Galton--Watson tree and random walks on it at the same time, see Section~\ref{SECTION-watershed}. Each random walk will explore a portion of the tree below its starting point, and we call such a subset of the tree a ``watershed''.
The specific exploration via watersheds will prevent the random walks from ``predicting the future of the tree''\,during its construction; that is, we construct each watershed on a part of the Galton--Watson tree while preserving the independence of the rest of the tree. The main feature of the explored tree is its stability to perturbation by small quenched noise.  The desired positivity of $h_*$ will then be obtained by means of a Dynkin-type isomorphism theorem between the Gaussian free field and random walks, see \cite{Einsenbaumandall2000}, or more precisely with random interlacements, a random soup of random walks, see \cite{Sznitman2011a,Lupu2016}. Moreover, we expect that our exploration procedure of the Galton--Watson tree via watersheds can also be used to obtain other interesting results. A first manifestation of this is already provided by the results on noise-stability and transience for the interlacements set as well as for the level sets of the Gaussian free field above small positive levels, see Theorems~\ref{THM-quenchednoise} and~\ref{THM-Transience}  below.

\subsection{Main results}

Let us now explain our setting and results in more detail. We consider a
\begin{equation}
	\label{DEF-mathcalT}
	\text{Galton--Watson random tree $ \mathcal{T}$ with mean offspring distribution $m>1,$ conditioned on survival,}
\end{equation}
and denote the underlying probability measure by $\mathbb{P}^{{\rm GW}}.$
We endow the natural graph structure induced by $ \mathcal{T}$ with positive random conductances $λ_{x,y},$ $x\sim y,$ such that, conditionally on $\mathcal{T},$  and denoting by $y^-$ the parent of $ y \in \mathcal T,$ with $y $ different from the root $\emptyset,$
\begin{equation}\label{DEF-Conductances+}
	\begin{gathered}
		\text{the family $ \{λ_{x,y}\, : \, y \in \mathcal T \text{ and } y^-=x \}_{x \in \mathcal T},$ is i.i.d.\ and} \\
		\mathbb{E}^{{\rm GW}}[\lambda_{x,+}] < ∞	\quad ∀x\in\mathcal{T},\text{ where }λ_{x,+}:=\sum_{y:\,y^-=x} λ_{x,y};
	\end{gathered}
\end{equation}
note that this setting is slightly more general than endowing the edges of the Galton--Watson tree with independent conductances.
In particular, when the conductances $\lambda_{x,y},$ $x\sim y,$ are constant equal to $1,$ we recover the usual Galton--Watson tree, and in this case condition \eqref{DEF-Conductances+} simply boils down to the mean offspring distribution $m$ being finite. In a slight abuse of notation, we also denote by $\mathcal{T}$ the weighted graph with the conductances $\lambda,$ and will explicitly mention when we consider the tree $\mathcal{T}$ to be weightless as in \eqref{DEF-mathcalT} to avoid confusion. We refer to Section~\ref{SECTION-GWTree} for precise notation and definitions.

It is known that the random tree $\mathcal{T}$ is almost surely transient, cf.\ Proposition~\ref{prop:GWtransient}, and conditionally on its realization, we denote by $g^{\mathcal{T}}$ the Green function associated to the random walk on $\mathcal{T},$ see below \eqref{DEF-GreenFunction}.

Conditionally on the realization of $\mathcal T,$ we then define the Gaussian free field $ (φ_x)_{x\in\mathcal{T}} $ under some probability measure $\Pgff{\mathcal T}{}$ as the centered Gaussian field with covariance function $g^{\mathcal{T}},$ see Section~\ref{sec:GFF} for further details. Note that this is a Gaussian free field in a random environment, that is we first generate the Galton--Watson tree $\mathcal{T}$ with random conductances and then --  conditionally on the surviving Galton--Watson tree $\mathcal{T}$ -- we generate a Gaussian free field on $\mathcal{T}.$

We will study the percolative properties of the \emph{level sets} or {\em excursion sets} of the Gaussian free field on $\mathcal T,$ i.e., of the random set
\begin{equation}\label{DEF-E ge h}
	E^{\ge h} := E^{\ge h}(\mathcal T)= \set{x\in \mathcal{T} \colon φ_x\ge h}, \quad h\in \R.
\end{equation}
We observe that the level set is clearly decreasing in $ h $, and we define the critical parameter
\begin{equation}\label{DEF-h_*}
	h_*:= h_*(\mathcal T):= \inf \set{h\in \R\, \colon  \Pgff{\mathcal T}{}\text{-a.s.\ all connected components of $ E^{\ge h}(\mathcal{T}) $ are bounded}}
\end{equation}
for the corresponding percolation problem.

A priori, it is not known if $h_*$ is deterministic, nor whether the phase transition is nontrivial, i.e., whether $h_* \in \R$. For unitary conductances, the former is proved in \cite[Lemma~5.1]{AbacherliSznitman2018}, and the latter -- more precisely the inequality $0\leq h_*<\infty$ -- is proved in \cite[Proposition~5.2]{AbacherliSznitman2018}, taking advantage of  \cite{Tassy2010}. The result $ h_*>0$ is shown to hold in \cite{AbacherliSznitman2018} for constant conductances under the additional assumption $ m\in (2,∞);$ however,  it seems that the assumption of finite mean is not essential to their proof.  Let us also note in passing that even for Galton--Watson trees with random i.i.d.\ conductances, $h_*(\mathcal T)$ is still deterministic, see Appendix \ref{SECTION-uhDeterministic}. We now state our main result.
\begin{thm}\label{THM-h*>0}
	Under \eqref{DEF-mathcalT} and \eqref{DEF-Conductances+}, there exists $h>0$ such that $E^{\geq h}$ contains $\mathbb{E}^{{\rm GW}}[\Pgff{\mathcal{T}}{}(\cdot)]$-almost surely an unbounded connected component, and hence $ h_*(\mathcal T)>0.$
\end{thm}

Note that Theorem~\ref{THM-h*>0} does not yet imply that the phase transition is non-trivial, that is $h_*(\mathcal{T})<\infty.$ Indeed, this finiteness property does hold true for i.i.d.\ weights, but it may fail without this condition -- we refer to the discussion below \eqref{eq:h*small2u*} for details.

In the case $m>2,$ the assumption $\Egw[λ_{x,+}]<∞$ from \eqref{DEF-Conductances+} is not necessary to prove the inequality $h_*>0$ as explained at the end of Section~\ref{Section-warmup} (for unitary conductances this also follows from \cite[Theorem~5.5]{AbacherliSznitman2018}). In view of Theorem~\ref{THM-h*>0}, a natural question then is whether $h_*>0$ under the broader assumptions $\Egw[λ_{x,+}]=∞$ and $m\in{(1,2]}.$

We will now put our result into the context of previous literature on percolation for the Gaussian free field. The study of this percolation problem for unitary conductances had been initiated by Bricmont, Lebowitz and Maes in \cite{BricmontLebowitzMaes87} on the Euclidean lattice $ \Z^d $ in transient dimensions $ d\ge 3.$ Using a soft but quite robust contour approach, they proved  that $h_*(\Z^d) \ge 0$ for all $d \ge 3,$ as well as $h_*(\Z^3) < \infty.$ More recently,  on $ \mathbb{Z}^d $, it has been established in \cite{RodriguezSznitman2013} that $ h_*(\Z^d)<\infty$ for all $ d\ge 3,$ as well as $h_*(\Z^d)>0$ for all sufficiently large $d;$ in \cite{DrewitzPrevostRodriguez2018-1} it has then subsequently been shown that $ h_*(\Z^d)>0 $  for all $ d\ge 3 $. For trees with unitary conductances, the parameter $h_*\in{(0,\infty)}$ was first characterized in \cite{Sznitman2016} on $ d$-regular trees, $ d\ge 3 $,  and subsequently in \cite{AbacherliSznitman2018} for a larger class of transient trees, including supercritical Galton--Watson trees with mean $ m>2 $.  In the regular tree case, finer percolative properties have been obtained in the recent preprint \cite{Cerny2023}, which appeared after the preprint version of this article.

In \cite{AbacherliCerny2019}, further percolative properties for $ d$-regular trees have then been studied in the super- and sub-critical regime. In \cite{DrewitzPrevostRodriguez2018-2}, $ h_*>0 ,$ and in fact local uniqueness of the infinite cluster at a positive level, has been shown for a larger class of graphs with polynomial growth. This class of graphs actually include $\Z^d,$ $d\geq3,$ with bounded conductances as a special case, which was further studied in \cite{ChiariniNitzschner2021}. We also refer to \cite{Sznitman2015,AbacherliCerny2020,DuminilCopin2020,GoswamiRodriguezSevero2021,Conchon2021,Cerny2021} for further recent progress in this area.

Our proof crucially relies on another important object: the random interlacements set $ \mathcal{I}^u,$ $u>0,$ which has been introduced in $\Z^d,$ $d \ge 3,$ by \cite{Sznitman2010}. Later on, it has been generalized to transient weighted graphs in \cite{Teixeira2009}.  It is related to the Gaussian free field via Ray-Knight type isomorphism theorems, first obtained in \cite{Sznitman2011a}, and later on extended in a series of works \cite{Lupu2016,Sznitman2016,DrewitzPrevostRodriguez2021}. From a heuristic point of view, random interlacements is a random soup of doubly infinite transient random walks, and the union $\mathcal{I}^u$ of their traces thus trivially has an unbounded connected component (and hence percolates). On $\Z^d,$ $d\geq3,$ it was proved in \cite{MR3024098} that $\mathcal{I}^u$ still percolates when perturbed by a small quenched noise, and this property was essential in the proof of $h_*>0$ from \cite{DrewitzPrevostRodriguez2018-1}. Although our approach to proving $h_*>0$ on Galton--Watson trees is quite different from that of \cite{DrewitzPrevostRodriguez2018-1}, the stability of $\mathcal{I}^u$ to perturbation via small quenched noise will still play an essential role in our proof of Theorem~\ref{THM-h*>0}. Note that in the context of random Galton--Watson trees, we will see $\mathcal{I}^u$ as a quenched random interlacements on the realization of the tree $\mathcal{T}; $ see Section \ref{sec:RI} for details.

We now describe this stability property -- which is of independent interest, see its implications in Theorem~\ref{THM-Transience} below -- in more detail. Again conditionally on the realization of the tree $\mathcal{T},$ for some $p\in{(0,1)},$ denote by $\mathcal{B}_x,$ $x\in{\mathcal{T}},$ an independent family of i.i.d.\ Bernoulli random variables with parameter $p$ and let
\begin{equation}
	\label{eq:defBp}
	B_p:=\{x\in{\mathcal{T}}:\mathcal{B}_x=1\}.
\end{equation}
\begin{thm}
	\label{THM-quenchednoise}
	Under \eqref{DEF-mathcalT} and \eqref{DEF-Conductances+}, for all $u>0,$ there exists $p\in{(0,1)}$ such that $\mathcal{I}^u\cap B_p$ contains almost surely an infinite connected component. Moreover, there exist $h>0$ and $p\in{(0,1)}$ such that $E^{\geq h}\cap B_p$ contains almost surely an infinite connected component.
\end{thm}

In \cite{MR3024098}, the question of stability of the vacant set $\mathcal{V}^u:=(\mathcal{I}^u)^c$ to perturbation by small quenched noise on $\Z^d$ has also been studied. In a similar vein, on Galton--Watson trees one can also easily prove that $\mathcal{V}^u\cap B_p$ percolates for $p$ large enough, see Remark~\ref{REMARK-u*Proof}. In \cite{MR3024098}, the proof of stability of $\mathcal{I}^u$ to perturbation by small quenched noise involves some local connectivity result for random interlacements, which can also be used to prove transience of the interlacements set \cite{MR2819660}, or of $\mathcal{I}^u\cap B_p,$ see \cite{MR3024098}. It turns out that, although our proof of Theorem~\ref{THM-quenchednoise} is entirely different from that of \cite{MR3024098}, it can also be employed to show transience of $\mathcal{I}^u\cap B_p,$ or of $E^{\geq h}\cap B_p$ at small, but positive, levels, under some additional assumptions on the conductances.

\begin{thm}
	\label{THM-Transience}
	Assume \eqref{DEF-mathcalT}, \eqref{DEF-Conductances+} and that, conditionally on the non-weighted graph $\mathcal{T},$  $(\lambda_{x,y})_{x\sim y\in{\mathcal{T}}}$ are i.i.d.\ conductances with compact support in $(0,\infty).$ Then for all $u>0,$ there exists $p\in{(0,1)}$ such that $\mathcal{I}^u\cap B_p$ contains almost surely a transient connected component. Moreover, there exist $h>0$ and $p\in{(0,1)}$ such that $E^{\geq h}\cap B_p$ contains almost surely a transient connected component.
\end{thm}
For the reader's convenience we refer to the discussion above \eqref{eq-BoundedConductancesIn6} for the precise definition of what it means in our context that, conditionally on the non-weighted graph $\mathcal{T},$  $(\lambda_{x,y})_{x\sim y\in{\mathcal{T}}}$ are i.i.d.\ conductances with compact support in $(0,\infty)$ -- which, in fact, is arguably the  ``natural'' way of endowing a tree with i.i.d.\ random conductances, but less general when compared to \eqref{DEF-Conductances+}.

Let us finish this subsection with some comments on percolation for the vacant set of random interlacements, and the finiteness of $h_*.$ The random interlacements set $\mathcal{I}^u$ always percolates since the trace of a transient random walk is an unbounded connected set; one may, however, wonder if the same holds true for its complement the \emph{vacant set} $ \mathcal{V}^u $ when the intensity parameter varies.

Denoting by $u_*$ the critical parameter associated to the percolation of $\mathcal{V}^u,$ $u>0,$ the isomorphism between random interlacements and the Gaussian free field, see Proposition~\ref{THM-Isomorphism} below (which can be used in our context in view of Proposition~\ref{prop:capRWinfty}), implies similarly as in \cite[Theorem~3]{Lupu2016} that
\begin{equation}
	\label{eq:h*small2u*}
	h_*\leq\sqrt{2u_*}.
\end{equation}
The inequality \eqref{eq:h*small2u*} combined with Theorem~\ref{THM-h*>0} implies $u_*>0$, but note that the inequality $u_*>0$ could be proved via easier means, see Remark~\ref{REMARK-u*Proof}. Let us note here that in the special case of unitary conductances, an explicit formula for $u_*$ has been derived in \cite{Tassy2010}. The proof of \cite[Theorem~1]{Tassy2010} can be adapted to random conductances as long as $(\lambda_{x,y})_{x\sim y\in{\mathcal{T}}}$ are i.i.d.\ conductances conditionally on the non-weighted graph $\mathcal{T}.$
In particular, $u_*<\infty$ under the same conditions, and thus $h_*<\infty$ as well by  \eqref{eq:h*small2u*}.
However, if we allow the
weights $(\lambda_{x,y})_{x\sim y\in{\mathcal{T}}}$ to not be i.i.d.\ conditionally on the non-weighted graph $\mathcal{T}$ -- but still satisfying the usual setup of \eqref{DEF-Conductances+} --
one can find Galton--Watson trees where $h_*=\infty,$ see \eqref{EQ-h*infinity}.

The weak inequality \eqref{eq:h*small2u*} can actually  be improved to $ h_* < \sqrt{2u_*} $ on $ d$-regular trees, $ d\ge 3,$ see \cite{Sznitman2016}. In \cite{AbacherliSznitman2018}, the authors provide general enough conditions to obtain $ h_* < \sqrt{2u_*} $ on transient trees, and in particular for Galton--Watson trees with unitary conductances this strict inequality holds under additional hypotheses on exponential moments of the offspring distribution, see \cite[Theorem~5.4]{AbacherliSznitman2018}. They also provide an example, namely the tree where each vertex has an offspring size equal to its distance to the root, where actually $ 0=h_* = \sqrt{2u_*}.$ Note that this entails that Theorem~\ref{THM-h*>0} does not hold when removing the assumption $ \mathbb{E}^{{\rm GW}}[\lambda_{x,+}]<\infty$ from \eqref{DEF-Conductances+}, as well as the assumption that the distribution of the number of children does not depend on the generation.

\subsection{Outline of the proof}
\label{sec:outline}
We now comment on the proofs of Theorems~\ref{THM-h*>0}, \ref{THM-quenchednoise} and \ref{THM-Transience} in more detail. Let us first elaborate on the fact that Theorem~\ref{THM-quenchednoise} is useful to obtain Theorem~\ref{THM-h*>0}.
The isomorphism between random interlacements and the Gaussian free field, see Proposition~\ref{THM-Isomorphism}, implies that for each $u>0,$ random interlacements and the Gaussian free field on $\mathcal T$ can be coupled in such a way that
\begin{equation}
	\label{eq:isoforsets}
	\text{ almost  surely,} \quad \mathcal{I}^u\subset E^{\geq -\sqrt{2u}}.
\end{equation}
This implies in particular that $E^{\geq-\sqrt{2u}}$ percolates for all $u>0,$ and taking $u\downarrow 0$ we infer that $h_*\geq0.$ Note that the validity of the inclusion \eqref{eq:isoforsets} requires some condition on the tree to be fulfilled -- see \eqref{eq:capRWinfty} -- but we will actually show in Proposition~\ref{prop:capRWinfty} that this condition is always satisfied in our context.  In \cite{DrewitzPrevostRodriguez2018-1,DrewitzPrevostRodriguez2018-2}, an extension of the inclusion \eqref{eq:isoforsets} to a continuous metric structure associated with the discrete graph, the so-called cable system, was used to lift the inclusion \eqref{eq:isoforsets} -- when the field was taking not too high values -- to level sets of the Gaussian free field at positive levels, which then yielded the desired strict inequality $h_*>0.$ Here, we follow a simpler approach, that is we use an extension of the inclusion \eqref{eq:isoforsets}, see Proposition~\ref{THM-Isomorphism} below,  which includes information about the exact values of the free field, as well as the local times of random interlacements. Proposition~\ref{THM-Isomorphism} is proven using the cable system, cf.\ \cite{Lupu2016} for further details. The proposition readily implies that there exists a coupling such that for each $u>0,$
\begin{equation}
	\label{eq:Iu+Auincluded}
	\text{ almost surely,} \quad \mathcal{I}^u\cap A_u\subset \widehat{E}^{\geq \sqrt{2u}},
\end{equation}
where $\widehat{E}^{\geq \sqrt{2u}}$  has the same law as $E^{\geq\sqrt{2u}},$ see \eqref{DEF-E ge h},  and
\begin{equation}\label{PROP1-DEF-Au}
	A_{u}:= \left\lbrace	x\in \mathcal{T}\colon
	\mathcal{E}_x>4u\lambda_x\text{ or }\ab{φ_x}>2\sqrt{2u}	\right\rbrace,
\end{equation}
for some i.i.d.\ exponential random variables $(\mathcal{E}_x)_{x\in{\mathcal{T}}}$ with parameter one, independent of the Gaussian free field $φ$ and the interlacements set $\mathcal{I}^u.$ Note that $A_u$ increases a.s.\ to $\mathcal{T}$ as $u\rightarrow0,$ and one can thus interpret the intersection with $A_u$ as applying a small quenched noise. Theorem~\ref{THM-quenchednoise} then suggests that $\mathcal{I}^u \cap A_u$ might percolate for $u$ small enough, which again would imply Theorem~\ref{THM-h*>0} by \eqref{eq:Iu+Auincluded}.

However, one cannot directly use Theorem~\ref{THM-quenchednoise} for proving Theorem~\ref{THM-h*>0} for two reasons: first, the variables $\{x\in{A_u}\},$ $x\in{\mathcal{T}},$ are not independent, and second, the probability that $x\in{A_u}$ depends on the parameter $u$ of the interlacements set, and thus, contrary to $p$ in Theorem~\ref{THM-quenchednoise}, it cannot be taken arbitrarily close to one for a fixed $u.$ The first problem will be essentially solved by lower bounding the probability that $x\in{A_u}$ conditionally on $\{y\in{A_u}\},$ $y\neq x,$ using the Markov property of the free field, see \eqref{eq:boundprobaAu}. To solve the second problem, we will make the dependency of $p$ on $u$ in Theorem~\ref{THM-quenchednoise} explicit, that is, we find a function $p(u),$ with $p(u)\uparrow1$ as $u\rightarrow 0,$ such that $\mathcal{I}^u\cap B_{p(u)}$ percolates for all $u>0,$ and we show that the probability that $x\in{A_u}$ is larger than $p(u)$ for $u$ small enough, see the proof of Proposition~\ref{PROP-Percolation-Au}.

Therefore, in order to obtain Theorem~\ref{THM-h*>0}, it is essentially enough to show that $\mathcal{I}^u\cap B_{p(u)}$ percolates, where $p(u)$ is smaller than the probability that $x\in{A_u}$  for $u$ small enough. The main difficulty is that, when $u$ is small, there are two competing effects at play in this percolation problem. On the one hand, in the $u>0$ small regime, the interlacements set $\mathcal{I}^u$ consists of few trajectories, and hence is less well-connected; i.e., intersecting $\mathcal I^u$ with $B_p$ might break its infinite connected components into finite pieces. This is particularly problematic when $m$ is close to one, since the tree tends to contain long stretches which locally look like $\Z,$ and hence the connectivity of such components turns out to be sensitive to an independent noise. On the other hand, as $u\rightarrow0,$ for each $x\in{\mathcal{T}},$ the probability that $x$ is in $A_u$ tends to one, and it thus becomes less likely to break a fixed connected component of $\mathcal{I}^u$ into finite pieces when intersecting  with $B_{p(u)}.$ The proof of Theorem~\ref{THM-h*>0} therefore requires a subtle comparison of the influences of these two opposite effects as $u\rightarrow0.$ We now provide a short explanation of how this is done.

The probability that a vertex $x$ is contained in ${A_u^c}$ can be easily upper bounded by $u^{3/2}\lambda_x^{3/2},$ see \eqref{eq:boundprobaAu} below, and we can thus take $p(u)=1-u^{3/2}\lambda_x^{3/2}$ for $u$ small enough.
To prove percolation of $\mathcal{I}^u\cap B_{p(u)}$, we use a description of the trajectories in $\mathcal{I}^u$ via their highest (i.e., minimal distance to the root) visited vertex, Theorem~\ref{Thm-Interlacement and RW}, which can be seen as a generalization of \cite[Theorem~5.1]{Teixeira2009}. This description entails that $\mathcal{I}^u$ can be generated by starting, for each vertex $x\in \mathcal T,$ an independent Poissonian number $Γ_x$  of random walks starting at $x$ going down the tree. Here, the Poisson distribution underlying $Γ_x$ has parameter $u\widecheck{e}_{\mathcal{T}}(x),$ where
$\widecheck{e}_{\mathcal{T}}(x) $ -- see \eqref{DEF-cy} -- is a parameter depending on the subtree rooted at $x,$ which bears some similarity with the square of the conductance from $x$ to infinity.

Now in the simpler case where each vertex in the tree $\mathcal{T}$ always had at least two children and the conductances were bounded, one could finish the proof by first conditioning on $\mathcal T$ and by then
proceeding as follows.  One can under these conditions easily show that $\widecheck{e}_{\mathcal{T}}(x)$ is of constant order, uniformly in $x\in{\mathcal{T}}.$ Thus, when $Γ_x\geq1,$ with high probability, starting a random walk at $x$ going down the tree up to the first time it has visited $C/u$ vertices, for a large constant $C,$ there are at least two vertices $y$ with $Γ_y\geq1$ which are not visited by the walk, but children of vertices visited by the walk (the existence of such vertices is guaranteed by the fact that each vertex visited by the walk has at least two children). We will say that such a point $y$ corresponds to a free point, see \eqref{DEF-AllFreePoints}. Moreover, again with high probability as $u\rightarrow0,$ all the vertices visited by this walk are contained in $B_{p(u)},$ with $p(u)=1-u^{3/2}\lambda_x^{3/2},$ and in particular there is a path between $x$ and $y$ in $\mathcal{I}^u\cap B_{p(u)}.$ One can now iterate this procedure starting a new trajectory at each $y$ corresponding to a new free point, and show that the tree of free points contains a $d$-ary tree, see Proposition~\ref{PROP-Prop1}. In particular it percolates, which directly implies the percolation of $\mathcal{I}^u\cap B_{p(u)}$ also.

In this approach, we thus first generate $\mathcal{T},$ and then construct an infinite cluster in $\mathcal{I}^u\cap B_{p(u)}$ on the now fixed tree $\mathcal{T}.$ However, when the mean offspring number $m$ is close to one, or the conductances are not bounded, then the tree $\mathcal{T}$ will contain some connected components of vertices, each with exactly one child, with size more than $C/u,$ on which the above approach is bound to fail. Note, however, that as $u\rightarrow0,$ condition \eqref{DEF-Conductances+} in combination with the Marcinkiewicz-Zygmund law of large numbers
implies that these bad sequences in $\mathcal{T}$ become rarer when the tree is generated, see \eqref{eq:proofofv)}. In order to benefit from this information, we are going to generate the interlacements set $\mathcal{I}^u$ and the Galton--Watson tree $\mathcal{T}$ simultaneously. Generating the two processes at the same time is of considerable importance as it allows us to operate with the interlacements process without being forced to generate the whole tree beforehand.

To generate these two processes at the same time, we will explore the Galton--Watson tree using random walks, in the form of an object that we will call \emph{watershed}, as is explained in Section~\ref{SECTION-watershed} in more detail. The previously mentioned description of random interlacement trajectories via their highest visited vertex then implies that for each vertex $x,$ if a Poisson random variable with parameter $u$ takes the value at least one, one can start a watershed at $x,$ that is a walk starting at $x$ and exploring the tree below $x,$ which is included in random interlacements at level $u/{e}_{\{x\},\mathcal{T}_{x}}(x),$ see Proposition~\ref{PROP-W_Gw_WInterl}; here, ${e}_{\{x\},\mathcal{T}_{x}}$ is the equilibrium measure of the set $\{x\}$ for the subtree $\mathcal{T}_{x}$ of $\mathcal T$ rooted in $x,$ see \eqref{DEF-EquilibriumMeasure}. Now, for each vertex $x$, we will first generate a portion of the tree to make sure that ${e}_{\{x\},\mathcal{T}_{x}}(x)\geq c_e$ for some constant $c_e,$ see \eqref{PROP1-DEF-Uprime}, and then start a watershed  at $x$ if a Poisson random variable with parameter $u$ is at least one, which will thus be included in random interlacements at level $u/c_e,$  see Proposition~\ref{PROP-GoodwatershedsInInterlacements}. We can now use the additional randomness of the tree --  which in particular entails that with high probability there are no large components of  vertices each with exactly one child -- to show that, for $u>0$ small enough, the intersection of all the watersheds and $B_{p(u/c_e)}$ percolates for each $m>1,$ and thus $E^{\geq h}$ percolates for $h$ small enough as well; see Section~\ref{SECTION-Percolation} for details.

Finally, in order to prove Theorem~\ref{THM-Transience}, we note that, for uniformly bounded weights, the trace of a random walk on the watersheds is essentially a coarse-grained random walk on the tree of free points with a drift, see \eqref{eq:drift}. Using an argument from \cite{Collevecchio}, we deduce that such a random walk is transient, which finishes the proof using the isomorphism \eqref{eq:Iu+Auincluded} again.

The structure of the article is as follows: in Section~\ref{SECTION-AllDefinitions} we will define the main objects and set up  notation. In Section~\ref{Section-warmup} we provide a short and simple proof of Theorem~\ref{THM-h*>0} under the additional assumption $m>2$ -- this will turn out instructive for the proof of the general result also. Furthermore, we provide examples of Galton--Watson trees with $h_*=\infty.$ In Section~\ref{SECTION-watershed} we will introduce the exploration of the Galton--Watson tree through random walks, which is used in Section~\ref{SECTION-Percolation} to prove Theorems~\ref{THM-h*>0} and \ref{THM-quenchednoise}. In Section~\ref{SEC-transience}, we use similar methods to prove Theorem~\ref{THM-Transience}. Finally, we prove in Appendix \ref{SECTION-uhDeterministic} that $ h_* $ is deterministic in our setting.

	\bigskip
{\bf Acknowledgment:}
	The authors would like to thank Alain-Sol Sznitman for suggesting this problem, as well as Guillaume Conchon-Kerjan for useful remarks on an earlier version of the paper. The final version of the paper benefited from a careful reviewer's report. AD and GG have been supported by  Deutsche Forschungsgemeinschaft (DFG) grant DR 1096/1-1. 
	AP has been supported by the Engineering and Physical Sciences Research Council (EPSRC) grant EP/R022615/1, Isaac Newton Trust (INT) grant G101121 and European Research Council (ERC) starting grant 804166 (SPRS).

\section{Notation and definitions}\label{SECTION-AllDefinitions}
In Sections~\ref{SECTION-GWTree} and \ref{SECTION-Pruning} we introduce the Galton--Watson trees which we will be considering.  Subsequently, Sections~\ref{sec:GFF} and \ref{sec:RI} are then devoted to random walks, the Gaussian free field, as well as random interlacements on trees. In  Section~\ref{sec:iso} we introduce the isomorphism theorem between random interlacements and the Gaussian free field.

\subsection{Galton--Watson trees}\label{SECTION-GWTree}

We will investigate trees using the Ulam-Harris labeling. For this purpose, consider the space
\begin{equation}\label{DEF-UlamHarris}
	\mathcal{X}:= \bigcup_{i=0}^{∞}\N^i,
\end{equation}
where $\N$ is the set of positive integers, $\N_0$ the set of non-negative integers and $ \N^0 $ is defined as $ \{\emptyset\} $.
For $i,j \in \N $ as well as $ x,y \in \mathcal{X} $ such that $ x=(x_1,\dots,x_i)\in \N^i$ and $ y=(y_1,\dots,y_j)\in \N^j$, we define the  concatenation of $ x $ and $ y $ as $ xy=(x_1,\dots,x_i,y_1,\dots,y_j)\in \N^{i+j}\subseteq\mathcal{X}.$ Moreover, for $A\subseteq\mathcal{X}$ and $x\in{\mathcal{X}}$ we introduce $x\cdot A:=\{xy \, : \, y\in{A}\};$
note that in contrast to pointwise concatenation we put an additional dot for aesthetic reasons.
For all $x=(x_1,\dots,x_i)\in{\mathcal{X}},$ $i\in{\N},$ we define $x^-:=(x_1,\dots,x_{i-1}),$ the parent of $x,$ with the convention $()=\emptyset.$
For a set ${A}\subseteq \mathcal{X}$ we define its (interior) boundary as $∂{A}:=\{x\in {A}\, \colon\, \nexists \, y\in {A}, \,y^-=x \} $. Note that this is not exactly the natural topological boundary, but this slightly modified definition will turn out useful for our purposes. We moreover introduce, for  $A\subseteq\mathcal{X}$ and $ x \in A,$ the set of children of $x$ in $A$ as
\begin{equation}\label{DEF-GenerationOfTree}
	G_x^A := \set { y\in A\given\,y^-=x}.
\end{equation}

We call $T\subset\mathcal{X}$ a tree if for each $x\in{T\setminus\{\emptyset\}},$ we have $x^-\in{T}$ and $|G_x^T|<\infty.$ We then say that $x\in{T}\setminus \{\emptyset\}$ is a child of $y\in{T}$ if $x^-=y.$  If the tree $T$ under consideration is clear from the context, for all $x,y\in{T},$ we write $x\sim y$ if either $x=y^-$ or $y=x^-.$  One can also view a tree $T$ as a graph with edges between $x$ and $y$ if and only if $x\sim y.$  On this graph, we denote by $ d_T(x,y) $ the usual graph distance.
We say that $T$ is a weighted tree if each edge between $x$ and $y$ is endowed with a symmetric conductance $\lambda_{x,y}=\lambda_{y,x}\in{(0,\infty)}.$ For $x\in{T}$ we also define $\lambda_{x,+}$  as in \eqref{DEF-Conductances+}. Since weights are not encoded in $\mathcal X,$  a weighted tree is not a subset of $\mathcal{X}.$ However, to simplify notation, we will often implicitly identify a weighted tree with its set of vertices, a subset of $\mathcal{X}.$  Note that most of the previous notation depends on the choice of the tree $T,$ which will always be clear from the context. For $ x\in{T} $, we write ${T}_x$ for the subtree of $T$ consisting of $x$ and all descendants of $x,$ endowed with the same conductances as in the underlying tree $T.$ In this article, we think of trees as growing from top to bottom, so we sometimes refer to the points in the subtree $T_x$ as the points below $x.$ A priori, ${T}_x$ may consist of finitely many nodes only, but with a standard pruning procedure, we will actually soon reduce ourselves to the case of infinite Galton--Watson trees, see Section~\ref{SECTION-Pruning}.

We now explain how to define a Galton--Watson tree with random weights as a random weighted tree $ \mathcal{T}.$
We consider a probability measure $ ν $ on $[0,\infty)^{\N},$ which will form a canonical probability space, in order to describe the offspring distribution as well as the associated conductances. More precisely, we consider $\nu$ such that if the sequence $(\lambda_i)_{i\in{\N}}$ on $[0,\infty)^{\N}$ has law $\nu,$ then there exists  $d\in{\N}$ such that $\nu$-a.s., $\lambda_i>0$ for all $i< d$ and $\lambda_i=0$ for all $i\geq d.$ We will soon use $\nu$ to assign weights to the edges of the tree by means of a vector $(λ_{x,xi})_{i=1}^∞,$ distributed according to $ν$ for each vertex $x.$
Throughout this article, except in Section~\ref{Section-warmup}, we moreover assume that the law of the conductances satisfies
\begin{equation}
	\label{eq:assfinitemoment}
	\mathbb{E}^{ν}\big[\lambda_+\big] < ∞,\text{ where }\lambda_+=\sum_{i} λ_{i};
\end{equation}
essentially, this is just a reformulation of the second condition in \eqref{DEF-Conductances+}.
Note that we do not assume the conductances to be bounded away from zero or infinity, nor that the conductances $\lambda_i,$ $i\in{\N},$ are independent under $\nu.$
Defining the function
$ π\colon [0,\infty)^{\N}\to \N_0$ via
$ (λ_{i})_{i\in\N} \mapsto  |\{i\in\N:\,\lambda_i>0\}|,$ we introduce the pushforward probability measure
\begin{equation}\label{DEF-OffspringDistribution}
	μ := ν \circ π^{-1}
\end{equation}
on $ \N_0.$ As it corresponds to the law of the number of edges with conductances different from 0,
it will play the role of the offspring distribution. We will assume from now on that the mean of the offspring distribution satisfies
\begin{equation} \label{eq:superCrit}
	m:= \sum_{i=0}^{∞} i μ(i) >1,
\end{equation}
which will correspond to  the case of  supercritical Galton--Watson trees.

On some rich enough probability space we define the Galton--Watson tree $ \mathcal{T}$  by constructing $\mathcal{T}\cap \N^k(\subset\mathcal{X}),$ endowed with conductances on the (undirected) edges with the vertices in $\mathcal{T}\cap\N^{k-1},$ recursively in $k.$ For $k=0,$ we simply start with the vertex $ \emptyset \in\N^0\subseteq \mathcal{X}$ called the root. For $k\geq0,$ once the tree $\mathcal T$ has been generated up to generation $k,$ for each vertex $ x\in\N^k\cap\mathcal{T}$ we generate independently a random vector $  (λ_{x,xi})_{i\in{\N}}$ with law $\nu.$ The vertex $x$ has $ π( (λ_{x,xi})_{i\in{\N}}) $ children, and we endow the edge from $x$ to its child $xi,$ $1\leq i\leq π( (λ_{x,xi})_{i\in{\N}}),$ with the conductance $\lambda_{x,xi}\in{(0,\infty)}.$ This defines $\mathcal{T}\cap\N^{k+1}$ and its conductances with vertices in $\mathcal{T}\cap\N^k.$
The union over $k \in \N_0$ of these sets, endowed with the respective conductances, is denoted by $\mathcal{T},$ the weighted Galton--Watson tree.
Note that the structure of the tree is completely determined by the weights $λ,$ and that an edge between two vertices is present if and only if the conductance between them is non-zero. Under our standing assumption \eqref{eq:superCrit}, the tree becomes extinct with probability $ q <1$ (cf.\ for instance the discussion below \cite[Proposition~5.4]{LyonsPeres17}). Hence, it has a positive probability to survive indefinitely, and in order to avoid trivial situations, we will always condition the Galton--Watson tree on this event of survival in what follows. We denote by $ \Pgw$ the probability measure underlying the Galton--Watson tree constructed above, conditioned on survival.

Let us also define here already the canonical $\sigma$-algebras that we consider throughout the article, and which only become relevant at later points in this article. The set $\mathcal{X}$ is endowed with the $\sigma$-algebra $ σ(\set{x},\,{x\in\mathcal{X}}),$  and the space of subsets of $\mathcal{X}$ is endowed with the  $\sigma$-algebra generated by the coordinate functions $A\mapsto\ind_{\{x\in{A}\}},$ $x\in{\mathcal{X}}.$ If $T\subset\mathcal{X},$ we will often regard $(\lambda_{x,y})_{x\sim y\in{T}}\in{(0,\infty)^{\{x,y\in{T}:\,x\sim y\}}}$ as an element of $[0,\infty)^{\mathcal{X}\times\mathcal{X}},$ endowed with the product of the Borel-$\sigma$-algebras, by taking $\lambda_{x,y}=0$ if either $x\notin{T}$ or $y\notin{T},$ or else if $x$ and $y$ are not neighbors in $T.$

\subsection{Pruning of the tree}\label{SECTION-Pruning}

In this subsection we describe a useful pruning procedure for the tree conditioned on survival, which corresponds to chopping all finite branches of the tree -- the remaining subtree is known as the reduced subtree in the literature, see e.g.\ \cite{LyonsPeres17}. In order to simplify our investigations, we will then observe that the conditioned chopped Galton--Watson tree can also be constructed as a Galton--Watson tree with modified offspring distribution and which then survives almost surely, see \eqref{eq:deff*}. For this purpose, we define the reduced subtree $\mathcal{T}^{\infty}$ of $\mathcal{T}$ as consisting of those vertices of $\mathcal{T}$ which have an infinite line of descendants:
\begin{equation*}
	\mathcal{T}^{∞}:=\set{ x\in\mathcal{T} :\, \mathcal{T}_x\text{ is infinite}},
\end{equation*}
where we recall that the notation $\mathcal T_x$ has been introduced in the paragraph below \eqref{DEF-GenerationOfTree}.

Then \cite[Proposition~5.28 (i)]{LyonsPeres17} entails that  $ \mathcal{T}^∞ ,$ which can be seen as a tree in $\mathcal{X}$, has -- possibly after relabeling and conditionally on survival -- the same law  as a Galton--Watson tree $ \mathcal{T}^* $ with offspring distribution $ μ^*.$ The latter is characterized by its probability generating function
\begin{align}
	\label{eq:deff*}
	\begin{split}
		f^*(s)= \frac{f(q+ s(1-q)) - q}{1-q},		\qquad &\text{where } q\text{ is the probability that }\mathcal{T}\text{ is finite, and }\\
		&\text{$ f $ is the  probability generating function of }\mu.
	\end{split}
\end{align}
Note that $ f^*(0)=0 $, hence $ μ^*(0)=0, $ i.e.\,points in $ \mathcal{T}^* $ have zero probability of generating no children, and that $ μ^* $ has the same mean $ m $ as the law $ μ $ associated to $ \mathcal{T}.$

The behavior of the law of the conductances under pruning is slightly more involved. Indeed, conditionally on $\mathcal T$ and for each $x\in{\mathcal{T}},$  conditionally on its number of children $|G_x^{\mathcal{T}}|,$ the weights $(\lambda_{x,y})_{y\sim x}$ are independent of the event $\{x\in{\mathcal{T}^{\infty}}\}.$ Therefore, one can find a probability measure $\nu^*$ on $[0,\infty)^{\N}$ with $\nu^*\circ\pi^{-1}=\mu^*$ such that the weighted tree $\mathcal{T}^{\infty}$ has -- after relabeling -- the same law conditionally on survival as a weighted Galton--Watson tree $\mathcal{T}$ obtained from the probability $\nu^*.$ The law of $\nu^*$ is the same as the law of $\nu$ restricted to $P$ positive coordinates chosen uniformly at random among the $K+P$ positive coordinates of $\nu,$ where $P$ has law $\mu^*$ and $K$ has the law of the number of children of the root which do not survive, given that the root has $P$ surviving children (its probability generating function is described in \cite[Proposition~5.28 (iv)]{LyonsPeres17}).

Note that even under $ν^*$ it holds true that $\mathbb{E}^{ν^*}[\sum_{i \in \N} λ_i]<∞ $. Indeed, we first condition on survival which is an event of positive probability, and then we delete those points not belonging to $\mathcal{T}^∞$, which can only decrease the respective expected conductance.

We already remark at this point that the above pruning procedure does not change the critical parameter $h_*$ we are interested in, as the Gaussian free field restricted to $\mathcal{T}^∞$ has the same law on the pruned tree, and similarly for random interlacements. In particular, Theorems~\ref{THM-h*>0}, \ref{THM-quenchednoise} and \ref{THM-Transience} can be proven equivalently on the initial tree or on the pruned tree, and we refer to Remark~\ref{RMK-PruningGFFRI} for further details.

Therefore, without loss of generality, from now on we always work under the standing assumption that
\begin{equation}\label{REMARK-TInfiniteDescendants}\tag{SA}
	\begin{gathered}
		\text{$\nu$ is a probability measure such that $\pi\big((\lambda_i)_{i\in{\N}}\big)\geq1$ $\nu$-a.s.; }
		\\\text{i.e., under $ \Pgw $ all $ x\in \mathcal{T} $ have a.s.\ an infinite line of descendants.}
	\end{gathered}
\end{equation}
In particular, under \eqref{REMARK-TInfiniteDescendants}, $\Pgw$ is the law of a Galton--Watson tree without conditioning on survival, since survival occurs with probability one.

\subsection{Gaussian free field}
\label{sec:GFF}

Let us now define one of our main objects of interest, the Gaussian free field. We start with some general definitions related to random walks. Let $T$ be a weighted tree with positive weights $(\lambda_{x,y})_{x\sim y\in{T}}.$ For $x_0\in{T}$ we define a random walk $(X_n)_{n\in\N_0}$ on $T$ under  $P_{x_0}^{T}$ as the Markov chain on its canonical space $\N_0$ starting in $ x_0 $ with transition probabilities
\begin{equation}\label{DEF-RandomWalk}
	P_{x_0}^{T}(X_{n+1}=y\given X_n=x)= \frac{λ_{x,y}}{λ_x}\text{ for all } x\sim y\in T,
\end{equation}
where the total weight $\lambda_x$ at $x$ is defined as
\begin{equation}\label{DEF-TotalConductances}
	λ_x=\sum_{y\sim x}λ_{xy};
\end{equation}
note that the total weight, unlike $λ_{x,+}$ in \eqref{DEF-Conductances+}, sums over the conductance $λ_{x,x^-}$ also. For a set $U\subseteq T$, the hitting and return times of $X$, respectively, are denoted by
\begin{align}
	\begin{split}\label{DEF-stoppingTimes}
		H_{U}(X):=H_U			:= \inf\set{n \ge 0 \, : \,  X_n\in U } \text{ and }
		\tilde{H}_{U}(X):=\tilde{H}_U	:= \inf\set{n \ge 1 \, : \,  X_n \in U
		}	,
	\end{split}
\end{align}
respectively,
with the convention $ \inf\emptyset=∞ $. In the case of a single point $ U:=\set{x} $, we will write $H_x$ and  $\tilde{H}_x $ in place of $ H_{\{x\}}$ and  $\tilde{H}_{\{x\}} $.

In this section, we assume that the random walk $X$ on $T$ is transient, an assumption which will in particular be satisfied for supercritical  Galton--Watson trees conditioned on survival, see Proposition~\ref{prop:GWtransient}. For $U\subset T,$ the Green function associated to $X,$ killed upon exiting $ U $ under $P_{\cdot}^T,$ is given by

\begin{equation}\label{DEF-GreenFunction}
	g^{T}_U(x,y):= \frac{1}{\lambda_y}E_x^T \Big[\sum_{k=0}^{H_{T\setminus U}-1}\ind_{\{X_k=y\}}	\Big]
	\text{ for all }x,y\in{T}.
\end{equation}
In particular, we note that $g^T_U(x,y)=0$ if either $x\notin{U}$ or $y\notin{U}.$
In addition, we write $g^{T}(x,y):=\frac{1}{\lambda_y}E_x^T [\sum_{k=0}^\infty \ind_{\{X_k=y\}}]$, where $x, y \in T,$ for the Green function associated to $X$ on $T.$

Then $g^T$ is symmetric positive definite, and we can hence consider a probability measure $ \Pgff{T}{} $ on $ \R^{{T}} $ endowed with the canonical $ σ $-algebra generated by the coordinate maps $(φ_x)_{x\in{T}}$ such that
\begin{equation*}
	\begin{split}
		&(φ_x)_{x\in{T}} \text{ is a centered Gaussian field}
		\text{ with covariance given by }\Egff{T}{φ_x φ_y}= g^{{T}}(x,y), \,x,y \in T.
	\end{split}
\end{equation*}

We call $φ$ the Gaussian free field on the tree $T.$ Let us now recall the Markov property for $ φ,$ see for instance \cite[Proposition~2.3]{MR2932978}. For a finite set $ K\subseteq {T} $ and $ U:= {T}\setminus K, $ define for all $z\in{T},$
\begin{align}\label{LEMMA-MarkovProperty}
	\begin{split}
		β_z^{U}   := E_z^T \big[φ_{X_{H_{K}}} \ind_{\set{H_K<∞}} \big]\quad\text{ and }\quad
		\psi_z^{U}   := φ_z - β_z^U.
	\end{split}
\end{align}
Then
\begin{equation*}
	\begin{split}
		(\psi_z^{U})_{z\in{T}} \text{ is a centered Gaussian field with covariance function }\Egff{T}{\psi^{U}_z \psi^{U}_w}= g^{{T}}_{U}(z,w),
	\end{split}
\end{equation*}
which vanishes in $K$ and is independent of $ σ(φ_z, z\in K) $. Note moreover that $\beta^U$ is $ σ(φ_z, z\in K) $-measurable, and thus independent of $\psi^U.$

Putting the previous general considerations in our context of interest, we note that for almost all realizations of a weighted Galton--Watson tree $\mathcal{T},$ under $\Pgw$ the Green function $ g^{\mathcal{T}} $ is finite since the random walk is transient: the proof in \cite[Proposition~2.1]{Gantert2012} can be straightforwardly adapted to our case, i.e.\ the case where for each $x\in\mathcal{X}$, the family $(λ_{x,y})_{y\sim x},$ is not necessarily independent. This yields the following result.
\begin{prop}[\cite{Gantert2012}] \label{prop:GWtransient}
	$ \Pgw$-almost surely, the random walk on the tree $\mathcal T$ with conductances  $(\lambda_{x,y})_{x,y\in{\mathcal{T}},x\sim y }$ is transient.
\end{prop}
Hence, for almost all realizations of the Galton--Watson tree $\mathcal{T},$ we can define the Gaussian free field on $\mathcal T$  as the field $φ$ under $\Pgff{\mathcal{T}}{}.$

\subsection{Random interlacements}
\label{sec:RI}

The random interlacements process has been introduced by Sznitman \cite{Sznitman2010} for $ \mathbb{Z}^d$ (see \cite{DrewitzRathSapozhnikov2013} and \cite{MR3014964} for introductory texts) and it has subsequently been generalized to transient weighted graphs in \cite{Teixeira2009}. For  a transient weighted tree $T$ with conductances $(\lambda_{x,y})_{x\sim y \in T},$ we define the equilibrium measure and capacity of a finite set $K\subseteq T$ as
\begin{equation}\label{DEF-EquilibriumMeasure}
	e_{K,T}(x):= \ind_{\{x\in K\}}\lambda_xP_x^T(\tilde{H}_K=∞) \text{ and }    \capac_T (K):= \sum_{x\in K}e_{K,T}(x).
\end{equation}
We also define the capacity of an infinite set $F\subseteq T$ as the limit of the capacity of $F_n$ as $n\rightarrow\infty,$ where $(F_n)_{n\in{\N}}$ is a sequence of finite sets increasing to $F;$ we refer for instance to the end of \cite[Section~2.2]{DrewitzPrevostRodriguez2021} for as to why this limit exists and does not depend on the choice of the exhausting sequence $(F_n)_{n\in{\N}}.$  We further introduce the set
\begin{equation*}
	\overrightarrow{Z}_T:=\set{\overrightarrow{w}\colon \N_0 \to {T}      \given
		\overrightarrow{w}_n \sim \overrightarrow{w}_{n+1} \text{ for all }n \ge 0
		\text{ and } d_T(\emptyset,\overrightarrow {w}_n)\to∞ \text{ as }n\to∞}
\end{equation*}
of transient nearest neighbor trajectories on $T$ as well as the set
\begin{equation} \label{eq:biInfTr}
	\overleftrightarrow{Z}_T:=\big\{\overleftrightarrow{w}\colon \Z \to {T}       \given
	\overleftrightarrow{w}_n\sim \overleftrightarrow{w}_{n+1}  \text{ for all }n \in  \Z
	\text{ and } d_T(\emptyset,\overleftrightarrow{w}_n)\to∞ \text{ as }n\to\pm ∞\big\}
\end{equation}
of doubly infinite transient nearest neighbor trajectories.  In the literature, the set $\overleftrightarrow{Z}_T$ in \eqref{eq:biInfTr} is usually denoted by $W$; in this article, however, in a self-suggestive manner, we reserve $W$ for the notion of watersheds, a key object which will be defined in Section~\ref{SECTION-watershed}. Denote by $ \overleftrightarrow{X}$ the identity map on $\overleftrightarrow{Z}_T,$ and we indicate with $ \overrightarrow{X} $ and $ \overleftarrow{X} $ the forward and backward trajectories
\begin{align*}
	(\overrightarrow{X}_n)_{n\in\N_0} := (\overleftrightarrow{X}_n)_{n\in \N_0}\quad\text{ and }\quad
	(\overleftarrow{X}_n)_{n\in\N_0}  := (\overleftrightarrow{X}_{-n})_{n\in \N_0}.
\end{align*}
Let $ \overrightarrow{\mathcal{Z}}_T$ and $\overleftrightarrow{\mathcal{Z}}_T$ be the associated $ σ$-algebras on $\overrightarrow{Z}_T$ and $\,\overleftrightarrow{Z}_T$ generated by the coordinate functions.
On $ ( \overleftrightarrow{Z}_T, \overleftrightarrow{\mathcal{Z}}_T) $ we consider the family of measures \(Q_K^T\), $ K\subseteq {T} $ finite, which is characterized by the identities
\begin{equation}\label{DEF-Measure Q_K}
	\begin{split}
		Q_K^T \big((\overleftarrow{X}_n)_{n\in\N}\in A,\,X_0=x,\,(\overrightarrow{X}_n)_{n\in\N}\in B \big)
		=&  P_x^T\big(A, \tilde{H}_K =∞\big) λ_x P_x^T(B)\ind_{\set{x\in{K}}}\\
		\leftstackrel{\eqref{DEF-EquilibriumMeasure}}{=}  &P_x^T\big(A\given \tilde{H}_K =∞\big) e_{K,T}(x) P_x^T(B)
	\end{split}
\end{equation}
for all $ A,B\in \overrightarrow{\mathcal{Z}}_T, x \in T;$
here, $ \tilde{H}_{K} $ is the return time to $ K $ defined in \eqref{DEF-stoppingTimes}.

Following \cite{Teixeira2009}, one can then show that there exists a unique measure $ μ_T$ on the quotient space $ Z^*_T $ of trajectories in $ \overleftrightarrow{Z}_T$ modulo time shift, whose restriction to the trajectories hitting $K$ is the pushforward of the measures $ Q_K^T$ by projection onto $ Z^*_T $. Under some probability measure $ \Pri_{{T}} ,$ the random interlacements process on $T$ is then defined as the Poisson point process
\begin{equation}\label{DEF-RI}
	\sum_{i\in{\N}} δ_{(w^*_i,u_i)} \text{ on } Z_T^*\times [0,\infty) \text{ with intensity measure } μ_T\otimes λ,
\end{equation}
where $ λ $ is the one-dimensional Lebesgue measure restricted to $[0,\infty).$ For $u\in{(0,\infty)}$ we define the random interlacements process $\omega_u$ at level $u$   as the sum of $\delta_{w_i^*}$ over all $i\in{\N}$ with $u_i\in [0, u],$ and the random interlacements set $ \mathcal{I}^u $ at level $u$ as the subset of ${T} $ visited by the (equivalence classes of) random walks $w^*_i$ in the support of $\omega_u.$

We now present an alternative construction of the random interlacements process on trees, which will turn out useful for our purposes. It consists of partitioning the space $\overleftrightarrow{Z}_T$ into subsets according to the highest visited vertex of the contained trajectories.  For this purpose, for  $ x\in {T} $ define the quantity
\begin{equation}\label{DEF-cy}
	\widecheck{e}_T(x):= P_x^T \big( \tilde{H}_x=∞, H_{x^-}=∞ \big)  λ_x P_x^T (H_{x^-}=∞ ),
\end{equation}
where we recall that $ H_x $ and $ \tilde{H}_x $ are the hitting and return times, respectively, of $ x $, defined in \eqref{DEF-stoppingTimes}. If $x=\emptyset,$ we take the convention that $H_{x^-}=\infty$ occurs almost surely. We also define the law of a doubly infinite random walk with the point $x$ at smallest distance from the root $\emptyset,$ and which is reached for the first time at time $0,$ by
\begin{equation}
	\label{DEF-Measure barQ_x}
	\overline{Q}_x^T
	\big((\overleftarrow{X}_n)_{n\in\N}\in A,\,(\overrightarrow{X}_n)_{n\in\N}\in B\big)
	:=P_x^T({A}\,|\,\tilde{H}_x=\infty,H_{x^-}=\infty)P_x^T({B}\,|\,{H}_{x^-}=\infty),
\end{equation}
for all $ A,B\in \overrightarrow{\mathcal{Z}}_T.$  Here, we use the convention $H_{x^-}=\infty$ a.s.\ if $x=\emptyset.$
Note that $\widecheck{e}_{T}(\emptyset)\overline{Q}_{\emptyset}^T=Q_{\emptyset}^T.$
We now show that this alternative construction provides us with a random interlacements process as desired.
\begin{thm}\label{Thm-Interlacement and RW}
	Denote by $T$ a transient weighted tree with conductances $(\lambda_{x,y})_{x\sim y \in T}.$
	Let  $ u>0 ,$ and independently for each $ x \in {T},$ let $Γ_x$ be a ${\rm Poi}(u\widecheck{e}_{T}(x))$-distributed random variable. Furthermore, let $X_{x,i},$ ${i\in{\N}},$ be an independent i.i.d.\ family of doubly infinite random walks on $T$ with common law $\overline{Q}_x^T.$ Denote by $X_{x,i}^*$ the trajectory $X_{x,i}$ modulo time-shift. Then
	\begin{equation*}
		\sum_{x\in{T}}\sum_{i=1}^{Γ_x}\delta_{X_{x,i}^*}\text{ has the same law as }\omega_u\text{ under }\Pri_T.
	\end{equation*}

\end{thm}
\begin{proof}
	For $x \in  T$ we denote by $  \overleftrightarrow{Z}_{x,T} $ the subset of $ \overleftrightarrow{Z}_T,$ see \eqref{eq:biInfTr}, which contains only those doubly infinite trajectories with highest point equal to $ x ,$ reached for the first time at time $0,$ i.e.,
	\begin{equation*}
		\overleftrightarrow{Z}_{x,T} := \Big\lbrace X\in \overleftrightarrow{Z}_{T}	:\,X_0=x,H_{x^-}(\overrightarrow{X})=H_{x^-}(\overleftarrow{X})=\tilde{H}_{x}(\overleftarrow{X})=\infty \Big\rbrace.
	\end{equation*}
	Write $ Z_{x,T}^* $ for  the quotient space of $ \overleftrightarrow{Z}_{x,T}$ modulo time shift. Since trajectories on a tree have a unique highest point, the family of sets $Z_{x,T}^*,$ ${x\in{T}},$ forms a partition of  $Z_T^*.$

	For any measure $M$ and measurable set $A,$ write $M\big|_A$ for the restriction $M(A\cap \, \cdot \,)$ to $A.$ Recalling the definitions of $Q^T_K$, $\widecheck{e}_{T}$ and $\overline{Q}^T_x$ in \eqref{DEF-Measure Q_K}, \eqref{DEF-cy} and \eqref{DEF-Measure barQ_x}, we have for all events $A,B\in{\overrightarrow{\mathcal{Z}}}$ that
	\begin{align*}
		Q_{\{x\}}^T\Big|_{\overleftrightarrow{Z}_{x,T}}\big((\overleftarrow{X}_n)_{n\in\N}\in A,(\overrightarrow{X}_n)_{n\in\N}\in B \big)
		&=P_x^T\big(A,H_{x^-}=\infty,\tilde{H}_x=\infty\big)\lambda_xP_x^T(B,H_{x^-}=\infty)\\
		&=\widecheck{e}_{T}(x)P_x^T\big (A \given H_{x^-}=\infty,\tilde{H}_x=\infty\big )    P_x^T(B\given H_{x^-}=\infty)\\
		&=\widecheck{e}_{T}(x)\overline{Q}_x^T\big((\overleftarrow{X}_n)_{n\in\N}\in A,\,(\overrightarrow{X}_n)_{n\in\N}\in B\big).
	\end{align*}
	Next, write $(\overline{Q}_x^T)^*$ for the pushforward of $ \overline{Q}_{x}^T$ into the quotient space.  If a trajectory $X_x\in{ \overleftrightarrow{Z}_T}$ is such that
	$X_{x}^*\in{Z_{x,T}^*},$  then $ {Q}_{\{x\}}^T\barAS$ we have $X_{x}\in{\overleftrightarrow{Z}_{x,T}},$ so
	we see that
	$
	\frac{1}{\widecheck{e}_{T}(x)}μ_T\big|_{Z^*_{x,T}}= (\overline{Q}_x^T)^*.
	$
	Hence, since  $Γ_x$ is a Poisson random variable with parameter $u\widecheck{e}_{T}(x)$ we deduce that
	\begin{equation}
		\label{eq:interPPPwithhighestvertexx}
		\sum_{i=1}^{Γ_x}\delta_{X_{x,i}^*}\text{ is a Poisson point process on $Z^*_T $ with intensity measure }u\mu_T\big|_{Z_{x,T}^*}.
	\end{equation}

	Using the restriction property and the mapping theorem for Poisson point processes in order to first remove the trajectories with label bigger than $u$ and then the labels themselves, we see that the interlacements process $ω_u$ as defined below \eqref{DEF-RI} has the law of a Poisson point process with intensity measure $uμ_T $.

	Furthermore, since the subsets $Z_{x,T}^*,$  $x\in T,$ form a partition of $Z^*_T,$ due to the superposition theorem for Poisson point processes, taking the sum of \eqref{eq:interPPPwithhighestvertexx} over $x\in T$ yields  the law of a Poisson point process with intensity $uμ_T$, i.e. of $ω_u$, and the proof is complete.
\end{proof}

The representation of random interlacements via the highest vertex visited by its trajectories, Theorem~\ref{Thm-Interlacement and RW}, will be the base of our construction of the Galton--Watson tree via random interlacements, cf.\ Proposition~\ref{PROP-W_Gw_WInterl}.

\begin{remark}\label{REMARK-u*Proof}
	Theorem~\ref{Thm-Interlacement and RW} can be seen as a generalization of \cite[Theorem~5.1]{Teixeira2009}.  Indeed, if $x\in{T}$ is such that either $x^-\in{\mathcal{V}^u}:=(\mathcal{I}^u)^c$ or $x=\emptyset,$ then $x\in{\mathcal{V}^u}$ if and only if there are no trajectories in $\overleftrightarrow{Z}_{x,T}$ in the support of $\omega_u.$ By Theorem~\ref{Thm-Interlacement and RW}, this happens independently for each $x\in{T}$ with probability $\P{Γ_x=0}=\exp(-u\widecheck{e}_{T}(x)).$ In other words, the cluster of $\emptyset$ in $\mathcal{V}^u$ has the same law as the cluster of $\emptyset$ when opening each vertex $x$ of $T$ independently with probability $\exp(-u\widecheck{e}_{T}(x)).$ Moreover, $\widecheck{e}_{T}(x)$ is equal to the function $f_{\emptyset}(x)$ from \cite[(5.1)]{Teixeira2009}, and \cite[Theorem~5.1]{Teixeira2009} follows readily after rerooting.

	Similarly to \cite{Teixeira2009}, this can be used to prove the $\Pgw{}$-a.s.\ inequality $u_*(\mathcal{T})>0,$   where $u_*(\mathcal{T})$ is the critical parameter associated to the percolation of $\mathcal{V}^u$ under $\Pri_{\mathcal{T}}.$ Indeed, this follows from the following facts:
	\begin{itemize}
		\item
		the inequality $\widecheck{e}_{T}(x)\leq \lambda_x\leq\lambda_{x,+} + \lambda_{x^-,+}\ind_{\set{x\neq\emptyset}},$ and
		\item
		the fact that the cluster of $\emptyset$ for Bernoulli percolation on $\mathcal{T}$ with parameter $e^{-2uC}\ind_{\set{\lambda_{x,+}\leq C}},$ $x\in{\mathcal{T}},$ is a Galton--Watson tree since $\lambda_{x,+},$ $x\in{\mathcal{T}},$ are i.i.d.\ random variables, which is supercritical for first choosing $C$ large enough and then $u>0$ small enough.
	\end{itemize}
	Note that the inequality $u_*(\mathcal{T})>0$ can also be seen as a consequence of Theorem~\ref{THM-h*>0} as noted below \eqref{eq:h*small2u*}. One can furthermore also similarly prove that $\mathcal{V}^u\cap B_p$ -- see \eqref{eq:defBp} for notation -- percolates for $u>0$ small enough and $p\in{(0,1)}$ large enough, since it is minorized by Bernoulli percolation on $\mathcal{T}$ with parameter $pe^{-2uC}\ind_{\set{\lambda_{x,+}\leq C}},$ $x\in{\mathcal{T}}.$
\end{remark}
\begin{remark}\label{RMK-PruningGFFRI}
	Note that the trace random walk on $\mathcal{T}^{\infty}$ of the random walk on $\mathcal{T}$  is a random walk on $\mathcal{T}^{\infty},$ as follows from instance from \cite[Proposition~1.11]{MR2932978}. Therefore,  as in \cite[(1.30), (1.31)]{AbacherliSznitman2018}, the restriction of $φ$ to $\mathcal{T}^{\infty}$ has the same law as the Gaussian free field on $\mathcal{T}^{\infty},$ and so the critical parameters for level set percolation of the Gaussian free field on $\mathcal{T}$ and  $\mathcal{T}^{\infty}$ coincide -- note that this remains true in the case of weighted trees. In particular, one can substitute $\nu$ by $\nu^*$ when proving Theorem~\ref{THM-h*>0}. Moreover, one can easily prove that $\mathcal{I}^u\cap\mathcal{T}^{\infty}$ -- where $\mathcal{I}^u$ is the random interlacements set on $\mathcal T$ --  has the same law as the random interlacements set on the graph $\mathcal{T}^{\infty}$ (note to this effect that $\lambda_xP_x^T(A,\tilde{H}_K=\infty)$ is equal to $\sum_{y\in{\mathcal{T}^{\infty}}}\lambda_{x,y}P_y^T(A,H_K=\infty)$ for each $x\in{K}$ in \eqref{DEF-Measure Q_K}), and thus one can also substitute $\nu$ by $\nu^*$ when proving Theorems~\ref{THM-quenchednoise} and \ref{THM-Transience}.

\end{remark}

\subsection{An isomorphism theorem}
\label{sec:iso}

A key tool in our investigations is provided by certain Ray-Knight isomorphism theorems relating the Gaussian free field to random interlacements. Such results have a long history, dating back to Dynkin's isomorphism theorem and, less explicitly, even earlier work by Symanzik \cite{Sy68} as well as Brydges, Fr\"ohlich and Spencer \cite{BrFrSp-82}. The exact isomorphism that we are going to use here has been developed in \cite{Sznitman2011a}, \cite{Lupu2016}, \cite{Sznitman2016}, and then \cite{DrewitzPrevostRodriguez2021}.

As before, we still assume some transient weighted tree $T$ to be given. Recalling the definition below \eqref{DEF-RI} of the random interlacements process $\omega_u$ at level $u$, for $x\in{T}$ and $u>0$ let us denote by
\begin{equation*}
	\begin{gathered}
		\text{$N_x(u)$  the sum over all equivalence classes of trajectories $w^*$}\\
		\text{ in $\omega_u$ of the total number of times $w^*$ visits $x.$}
	\end{gathered}
\end{equation*}
On some possibly extended probability space, let  $\mathcal{E}_x^{(k)},$ $x\in{T}$ and $k\in{\N},$  be an i.i.d.\ family of exponential random variables with parameter one, independent of the random interlacements. The local time \(({\ell}_{x,u})_{x\in{T}},\) of random interlacements at level $u$ can then be defined as
\begin{equation}
	\label{eq:deflocaltimesRI}
	\ell_{x,u}:= \frac{1}{λ_x} \sum_{k=1}^{N_x(u)} \mathcal{E}_x^{(k)}\quad \text{ for all }x\in{T}.
\end{equation}
We can now state the isomorphism theorem; note that here and below, we use the convention that $H_{\emptyset^-}=\infty$ holds $P_x^{{T}}$-almost surely for any tree $T$ and $ x\in T. $

\begin{prop}\label{THM-Isomorphism}
	Assume that $T$ is a transient tree verifying that for all $x\in{T},$
	\begin{equation}
		\label{eq:capRWinfty}
		\capac_{T}( \{X_i,i \in\N \})=∞     \quad P_x^{{T}}(\, \cdot\,|\,{H}_{x^-}=\infty)\barAS
	\end{equation}
	Then for each $ u>0 $, there exists a coupling $\mathbb{Q}^u_T$ of two Gaussian free fields $ φ $ and $γ $ on $T,$ a random interlacements process $\omega_u$ on $T$ at level $u,$ and i.i.d.\ exponential random variables $\mathcal{E}_x^{(k)}$, $x\in{T}$ and $k\in{\N},$ with parameter one such that $φ,$ $\mathcal{E}_{\cdot}^{(\cdot)}$ and $\omega_u$ are independent, and $\mathbb{Q}^u_T$-a.s.,
	\begin{equation}\label{EQ-IsomorphismTheorem}
		\gamma_x= -\sqrt{2u} + \sqrt{2\ell_{x,u}+φ^2_x}				\qquad				\text{ for all } x \in  \mathcal{I}^u,
	\end{equation}
	where $\ell_{x,u}$ is defined as in \eqref{eq:deflocaltimesRI} and $\mathcal{I}^u$ as below \eqref{DEF-RI}.
\end{prop}

\begin{proof}
	The isomorphism theorem on the so-called cable system, see \cite[Proposition~6.3]{Lupu2016} or \cite[(0.4)]{Sznitman2011a} on general graphs, states that
	\begin{equation}
		\label{EQ-Isomorphismcable}
		|\tilde{\gamma}_x+\sqrt{2u}|= \sqrt{2\tilde{\ell}_{x,u}+\tilde{φ}^2_x} \quad	\text{ for all }x\in{\tilde{T}}.
	\end{equation}
	Here, $\tilde{T}$ denotes the cable system associated to $T,$ and $\tilde{γ},\tilde{φ}$ and $\tilde{\ell}_{\cdot,u}$ correspond to Gaussian free fields and local times of random interlacements on $\tilde{T}.$  We restrain from introducing the cable system $\tilde{T}$ in this article, as this metric structure will be only used in this proof; see \cite{Lupu2016} for references. We only note that $T\subset\tilde{T},$ and that the restrictions $\gamma,$ $φ$ and $\ell_{\cdot,u}$ of $\tilde{γ},\tilde{φ}$ and $\tilde{\ell}_{\cdot,u}$ to $T$ have the same laws as the corresponding fields from Proposition~\ref{THM-Isomorphism}. In order to deduce \eqref{EQ-IsomorphismTheorem} from \eqref{EQ-Isomorphismcable}, we note that
	\begin{equation}
		\label{eq:possibleinclusion}
		\begin{gathered}
			\text{each trajectory $w^*$ of ${\omega}_u$ is either included in a connected component of}
			\\\{x\in{\tilde{T}}:\,\tilde{γ}_x > -\sqrt{2u}\}
			\text{ or of }
			\{x\in{\tilde{T}}:\tilde{γ}_x < -\sqrt{2u}\},
		\end{gathered}
	\end{equation}
	which is a simple consequence of \cite[(3.19)]{DrewitzPrevostRodriguez2021}. Moreover, by \cite[Theorem~1.1,(1)]{DrewitzPrevostRodriguez2021} and symmetry it holds that
	\begin{equation}
		\label{eq:finitecapclusters}
		\text{all the connected components of }\big \{x\in{\tilde{T}}:\tilde{γ}_x < -\sqrt{2u} \big \} \text{  have finite capacity}.
	\end{equation}
	Under hypothesis \eqref{eq:capRWinfty}, for each trajectory $w^*$ of $\omega_u,$ it follows from Theorem~\ref{Thm-Interlacement and RW} that the capacity of $w^*$ is $\Pri{}$-a.s.\ infinite, and thus by \eqref{eq:possibleinclusion} and \eqref{eq:finitecapclusters}, $w^*$ must be included in $\{x\in{{T}}:\,γ_x > -\sqrt{2u}\}.$ The identity \eqref{EQ-IsomorphismTheorem} then follows readily from \eqref{EQ-Isomorphismcable}.
\end{proof}

Actually Proposition~\ref{THM-Isomorphism} remains true on any locally finite graph, but we will only need it on trees in this paper. We will prove that the hypothesis \eqref{eq:capRWinfty} holds when $T=\mathcal{T}$ is the Galton--Watson tree introduced in Section~\ref{SECTION-GWTree}, see Proposition~\ref{prop:capRWinfty}. Therefore, in our context, Proposition~\ref{THM-Isomorphism} will readily imply the inclusion \eqref{eq:Iu+Auincluded} (defining $\widehat{E}^{\geq\sqrt{2u}}$ therein as the level sets of the field $\gamma$), which is the first step in the proof of Theorem~\ref{THM-h*>0} as explained in Section~\ref{sec:outline}.

\begin{remark}
	\label{REMARK-conditionh*<0irrelevant}

	Following the proof of \cite[Proposition~5.2]{AbacherliSznitman2018}, one can easily show that a version of the isomorphism~\eqref{EQ-IsomorphismTheorem} holds on Galton--Watson trees with unitary conductances and finite mean offspring distribution $m.$ They prove this isomorphism using conditions different from \eqref{eq:capRWinfty}, namely that the sign clusters of the Gaussian free field on the cable system are bounded and a certain boundedness condition of the Green function; in view of \cite[Theorem~1.1,(2)]{DrewitzPrevostRodriguez2021}, the boundedness of the sign clusters is actually sufficient. It turns out that in the context of random conductances (and in particular, if the mean offspring distribution $m$ is infinite or if $(\lambda_{x,y})_{x\sim y\in{\mathcal{T}}}$ are not i.i.d.\ conductances conditionally on the non-weighted graph $\mathcal{T}$), it will be easier to deduce the isomorphism~\eqref{EQ-IsomorphismTheorem} from  condition \eqref{eq:capRWinfty} instead. Indeed, we will prove that condition~\eqref{eq:capRWinfty} holds in Proposition~\ref{prop:capRWinfty} using tools very similar to the proof Theorem~\ref{THM-quenchednoise}.
\end{remark}

\section{Warm up: a first proof in an easier setting}
\label{Section-warmup}

In this section we give a simple proof of the inequality $h_*(\mathcal{T})>0$ under the stronger assumption that $m>2.$  Note that this is also proved via different means in the setting of Galton--Watson trees with unit weights in \cite{AbacherliSznitman2018}. The proof in \cite{AbacherliSznitman2018} could be adapted to the setting of random weights, but it is currently not clear to us how to adapt it to the setting $m\in{(1,2]}.$ Moreover, we believe that our proof in this section for $m>2$ is simpler, and at the same time it exhibits the difficulties that are showing up when proving Theorem~\ref{THM-h*>0} for the case $m\in{(1,2]}.$ What is more, our proof will also provide us with an example of a weighted Galton--Watson tree  where $h_*=\infty,$ see \eqref{EQ-h*infinity}, showing that the phase transition is not always non-trivial in our context.

In order to introduce our setup, we consider the weighted Galton--Watson tree $\mathcal{T}\subseteq \mathcal{X}$ from Subsection~\ref{SECTION-GWTree}. Recall that the law of the weights below each vertex is a probability measure $ν$ on $[0,∞)^{\N},$ and these weights are chosen independently for different vertices, and that the function $\pi((λ_i)_{i\in{\N}})$ denotes the number of offspring, with mean $m,$ see \eqref{DEF-OffspringDistribution} and  \eqref{eq:superCrit}. Contrary to the rest of this article, in this section we do not make the usual assumption \eqref{eq:assfinitemoment} on the weights $λ,$ but keep the assumption $m>1$.
In the following, by  $F$ we denote the cumulative distribution function of a standard normal variable.

\begin{prop}
	For all $h\geq0$ such that there exists $M>0$ with
	\begin{equation}
		\label{EQ-Conditionforh_*geh}
		\mathbb{E}^{ν}\big[  π((λ_i)_{i\in{\N}})\ind_{\{\sum_{i\in{\N}}λ_i\le M\}} \big]F(-h\sqrt{2M})    >1,
	\end{equation}
	we have $h_*\geq h.$
\end{prop}
\begin{proof}
	In this proof, we use the construction of the Gaussian free field as in \cite[Section~2.1]{AbacherliCerny2019} through independent standard normal variables, extended to our case of non-regular trees. Let $(Z_x)_{x\in \mathcal{X}}$ be a family of independent standard normal variables under $P.$ Then, conditionally on the realization of the tree $\mathcal{T}$, define
	$φ_\emptyset := \sqrt{g^{\mathcal{T}}(\emptyset,\emptyset) } Z_\emptyset $ and, recursively in the distance from the root, we set
	\begin{equation*}
		φ_x := P_x^{\mathcal{T}}(H_{x^-}<∞)φ_{x^-} + \sqrt{g^{\mathcal{T}}_{\mathcal{T}_x}(x,x) }Z_x.
	\end{equation*}
	Using the Markov property \eqref{LEMMA-MarkovProperty} with $U=\mathcal{T}_{x}$, one can check that the field $(φ_x)_{x \in \mathcal T}$ defined this way has the law of a Gaussian free field on $\mathcal{T}$. Moreover, using the bound $g^{\mathcal{T}}_{\mathcal{T}_x}(x,x)\ge \frac{1}{λ_x}$, conditioned on the realization of the weighted tree $\mathcal{T}$, the previous display then entails the implication
	\begin{equation} \label{EQ-RecursivePercolationGFF}
		\{Z_x>h\sqrt{λ_x},  φ_{x^-}>h\}
		\,\Rightarrow\,
		\{φ_x>h\},
	\end{equation}
	with the convention $φ_{x^-}>h$ a.s.\ if $x=\emptyset.$

	We define now the random set $S(h,M)\subseteq \mathcal T$ as
	\begin{equation*}
		S(h,M):=\{\emptyset\}\cup\{ x\in\mathcal{T}\setminus\{\emptyset\}\colon Z_{x^-}>h\sqrt{2M}, λ_{x^-,+}\le M\}.
	\end{equation*}
	Note that on the event $x\in \mathcal{T}$, the mean number of children the vertex $x$ has in $S(h,M)$ satisfies
	\begin{equation}\label{EQ-mTimesZ}
		\begin{split}
			\mathbb{E}^{{\rm GW}}\otimes {E}\big[|G^{S(h,M)}_x|\,\big|\,x\in{\mathcal{T}}\big]&= \Egw\left[π((λ_{x,xi})_{i\in{\N}})\ind_{\set{\lambda_{x,+}\leq M}}
			P\big( Z_{x}>h\sqrt{2M}\big)\,\big|\,x\in{\mathcal{T}}
			\right]
			\\&=\mathbb{E}^ν\big[  π((λ_i)_{i\in{\N}})\ind_{\{\sum_{i\in{\N}}λ_i\le M\}} \big]F(-h\sqrt{2M}) .
		\end{split}
	\end{equation}
	Moreover, for each $x\in{\mathcal{T}},$ the number of children of $x$ in $S(h,M)$ only depends on $(\lambda_{x,xi})_{i\in{\N}}$ and $Z_x,$ which are independent in $x.$ Therefore, the connected component of $\emptyset$ in $S(h,M)$ has the law of a Galton--Watson tree with mean given by \eqref{EQ-mTimesZ}. Due to assumption \eqref{EQ-Conditionforh_*geh}, this mean is strictly larger than one  and thus this Galton--Watson tree has a positive probability to be infinite. Finally, it follows easily from \eqref{EQ-RecursivePercolationGFF} and the inequality $\lambda_x\leq \lambda_{x,+} +\lambda_{x^-,+}$ that $φ_{x^-}\geq h$ for each $x\neq\emptyset$ in the connected component of $\emptyset$ in $S(h,M),$ and we can conclude.
\end{proof}

Let us now present two interesting assumptions on the mean offspring  $m$ and on the distribution of the weights $(\lambda_i)_{i\in{\N}},$ under which \eqref{EQ-mTimesZ} is satisfied.
\begin{itemize}
	\item
	Assume $m>2$. We can find some $M>0$ such that
	$\mathbb{E}_{\nu}\big[  π((λ_i)_{i\in{\N}})\ind_{\{\sum_{i\in{\N}}λ_i\le M\}} \big]>2$ since the left hand side converges to $m$ as $M\to∞$, and then a positive level $h$ such that $F(-h\sqrt{2M})$ is close enough to $\frac 12$, so that \eqref{EQ-mTimesZ} is bigger than 1, providing us with $h_*>0.$
	\item Let $N$ be a random variable taking values in $\mathbb N$ with infinite mean under $\nu.$ Define $(λ_i)_{i\in{\N}}$ via  $\lambda_i=1/N$ for all $i\leq N$ and $\lambda_i=0$ for all $i>N.$ Then $\sum_{i\in{\N}}\lambda_i=1$ and $m=\infty$. Hence for each $h>0$ since $F(-h\sqrt{2})>0$ we have that the left-hand side of \eqref{EQ-Conditionforh_*geh} is infinite for $M=1,$ that is
	\begin{equation}\label{EQ-h*infinity}
		h_*=∞.
	\end{equation}
\end{itemize}

Note that we have not taken advantage of the assumption \eqref{DEF-Conductances+} in this section; as a consequence, the inequality $h_*>0$ from Theorem~\ref{THM-h*>0} holds when $m>2$ even without this assumption. It is not clear whether this assumption is necessary when $m\in{(1,2]}.$

\section{A simultaneous exploration of the tree via random interlacements}\label{SECTION-watershed}

In this section we introduce an explorative construction procedure for supercritical Galton--Watson trees via random interlacements, which is tailor-made for our purposes. To the best of our knowledge, previous approaches to problems related to random interlacements on random graphs generated the random interlacements process only after having complete information on the realization of the graph. In our setting, however -- in order to gain a better control on both, the Gaussian free field and the local times of random interlacements -- we generate the underlying graph $\mathcal T$ and the random interlacements process simultaneously. In some sense, this construction provides us with  independence properties that will turn out useful in creating coarse-grained ``good'' parts of the interlacements set and the level sets of an independent Gaussian free field.

In particular, in Subsection~\ref{Section-Singlewatershed} we will first construct a ``single small piece'' of the tree. This piece will consist of the trace of a finite random walk trajectory exploring the Galton--Watson tree at each vertex visited by the walk. We will call a piece of the tree constructed in this way a \emph{watershed}.
Repeating this procedure iteratively for boundary vertices of previously constructed watersheds, in Subsection~\ref{Section-PatchingTogetherwatersheds} we will then patch together all watersheds constructed in this way, as well as some remaining ends; the resulting object will be denoted by $\mathcal{T}^{\mathbf{W}}.$ It turns out that $\mathcal{T}^{\mathbf{W}}$ will be a tree with the following properties: it is a weighted Galton--Watson tree, and the random walk trajectories used to construct its watersheds can be interpreted as part of a random interlacements process on $\mathcal{T}^{\mathbf{W}}.$ This last property will be shown in Subsection~\ref{sec:RWandRI} with the help of Theorem~\ref{Thm-Interlacement and RW}.

\subsection{Watersheds}\label{Section-Singlewatershed}
We now introduce the notion of a watershed starting at a vertex $x \in \mathcal X\setminus\{\emptyset\},$ with parameters $L \in \N, L\ge 2$, and $κ\in [0,\infty),$ on which all the objects constructed in this subsection will depend implicitly (the case $x=\emptyset$ is excluded for technical reasons).  A watershed will form a finite subtree of a Galton--Watson tree, and it will be constructed as the trace of a random walk that is visiting vertices  starting at the root $x$ of a subtree of $\mathcal X,$ until -- if successful --  at least $L$  vertices of the subtree are explored in a suitable way. The parameter $κ$ will represent the conductance of the edge between $x$ and $x^-,$ which is thus fixed. In order to facilitate readability, we will denote objects pertaining to watersheds by boldface letters throughout.

The watershed will be defined by means of a sequence of  triplets $ (\mathbf{T}_k, (\boldsymbol{λ}_{y,z})_{y \sim z, y,z\in \mathbf{T}_k},\mathbf{X}_k)_{k \in \N_0},$ such that, for each $k\in{\N_0},$ we have that
\begin{itemize}
	\item $\mathbf{T}_k\subset\mathcal{X}$ is connected,
	\item the $\boldsymbol{λ}_{y,z}\in{(0,\infty)}$ are (symmetric) weights on the edges $\{y,z\}$ of $\mathbf{T}_k$,  and
	\item $\mathbf{X}_k$ is a random variable with $\mathbf{X}_k\in{\mathbf{T}_k}.$
\end{itemize}
In order to construct this sequence, we first fix
\begin{equation}
	\label{eq:deflambdaik}
	(\boldsymbol{\lambda}_{i}^{(k)})_{i\in{\N}},k\in{\N_0},\text{ an i.i.d. family of random variables with common law }\nu,
\end{equation}
and proceed by induction. We start with	$\mathbf{T}_0$ as being characterized uniquely by the specification of its vertex set $\{x^-,x\}$ (mind that $x^-$ is well-defined as we assumed $x \ne \emptyset$), as well as the conductance  $\boldsymbol{λ}_{x^-,x} := κ$ and  the almost sure equality $\mathbf{X}_0:=x. $

We first define the the triplet $(\mathbf{T}_k, (\boldsymbol{λ}_{y,z})_{y \sim z, y,z\in \mathbf{T}_k}, \mathbf{X}_k)$ until some stopping time $\tilde{V}_L(\mathbf{X}),$ that we will define in \eqref{DEF-StoppingTimeHL},  and thus assume that this triplet is given for some non-negative integer $k<\tilde{V}_L(\mathbf{X}).$ Recalling the definition below \eqref{DEF-UlamHarris} of the boundary $∂\mathcal{T}$ for a tree $\mathcal{T}$, we then define $(\mathbf{T}_{k+1}, (\boldsymbol{λ}_{y,z})_{y \sim z, y,z\in \mathbf{T}_{k+1}}, \mathbf{X}_{k+1})$ as follows:

\begin{itemize}
	\item if $ \mathbf{X}_k\in ∂\mathbf{T}_k,$ we proceed as follows. Let $\mathbf{N}_k:=|\{\mathbf{X}_0,\dots,\mathbf{X}_k\}|,$ and construct the offspring of $\mathbf{X}_k$ via $\boldsymbol{\lambda}^{(\mathbf{N}_k)}.$ More precisely, in Ulam-Harris notation, define $\mathbf{T}_{k+1}$ as the union of $ \mathbf{T}_k $ with the set of offspring of $\mathbf{X}_k,$ that is with $\{\mathbf{X}_ki,1\leq i\leq \pi((\boldsymbol{λ}_i^{(\mathbf{N}_k)})_{i\in{\N}})\},$ so $\mathbf{T}_{k+1}$ again is a tree.  By definition, the number of offspring of $\mathbf{X}_k$ in $\mathbf{T}_{k+1}$ has distribution $\mu.$ Furthermore, the weights $\boldsymbol{λ}$ on $\mathbf{T}_{k+1}$ are the same as on $\mathbf{T}_k,$ where in addition we now attribute weights $\boldsymbol{λ}_{\mathbf{X}_k,\mathbf{X}_ki}:=\boldsymbol{λ}_i^{(\mathbf{N}_k)}$ for  $1\leq i\leq \pi((\boldsymbol{λ}_i^{(\mathbf{N}_k)})_{i\in{\N}})$ to the edges which are contained in $\mathbf{T}_{k+1}$ but not in $\mathbf{T}_{k}.$

	\item
	if $ \mathbf{X}_k\notin ∂\mathbf{T}_k ,$ then we set $ \mathbf{T}_{k+1}:=\mathbf{T}_k,$ and the weights $\boldsymbol{λ}$ on $\mathbf{T}_{k+1}$ are the same as on $\mathbf{T}_k.$
\end{itemize}

In both of the above cases, in order to construct $\mathbf{X}_{k+1},$ we consider a random walk transition of $\mathbf{X}_k$ on $\mathbf T_{k+1};$ hence, independently of everything else, we define the random variable $\mathbf{X}_{k+1}$ as a neighbor of $\mathbf{X}_k$ in $\mathbf{T}_{k+1},$ which is equal to $y\sim \mathbf{X}_k,$ $y\in{\mathbf{T}_{k+1}},$ with probability $\boldsymbol{λ}_{\mathbf{X}_k,y}/\boldsymbol{λ}_{\mathbf{X}_k},$ where $\boldsymbol{λ}_{\mathbf{X}_k}$ is a normalizing constant defined similarly to \eqref{DEF-TotalConductances}. Note that, as long as $x^-$ is not reached by $\mathbf{X},$ the event $\{\mathbf{X}_k\in{\partial \mathbf{T}_k}\}$ above corresponds to the event $\{\mathbf{X}_k\notin{\{\mathbf{X}_0,\dots,\mathbf{X}_{k-1}}\}\}.$

We iterate the above procedure in $k$ until reaching the stopping time $\tilde{V}_L(\mathbf{X})$ that we are about to define. For this purpose, set $H_{x^-}(\mathbf{X})$ to be the first hitting time of $x^-$ by $\mathbf{X},$ defined similarly as in \eqref{DEF-stoppingTimes}, and
\begin{equation}
	\label{DEF-StoppingTimeVn}
	V_L:=V_L(\mathbf{X}):=\inf\{k\geq0:\,|\{\mathbf{X}_0, \ldots, \mathbf{X}_{k}\}| \ge L\}
	\wedge H_{x^-}(\mathbf{X})
\end{equation}
the first time at which the random walk $\mathbf{X}$ has visited $L$ different vertices, or $x^-$ is hit. Then let
\begin{equation}\label{DEF-StoppingTimeHL}
	\tilde{V}_L:=\tilde{V}_L(\mathbf{X}):=
	\begin{cases}
		\inf \big\{n\ge V_L\colon \mathbf{X}_n = \mathbf{X}_{V_L}^- \big\}
		&   \text{if } V_L(\mathbf{X}) < H_{x^-}(\mathbf{X}),   \\
		H_{x^-}         (\mathbf{X})
		&   \text{if } V_L(\mathbf{X})=H_{x^-}(\mathbf{X}),
	\end{cases}
\end{equation}
where we always use the convention $\inf\emptyset=\infty.$ In words, $\tilde{V}_L(\mathbf{X})$ is the first time the parent of $\mathbf{X}_{V_L}$ is visited if $H_{x^-}>V_L,$ and otherwise it equals $H_{x^-}.$ That is, we stop our recursive construction the first time either $x^-$ is visited by $\mathbf{X},$ or $\mathbf{X}$ has visited $L$ vertices at time $V_L,$ and then $\mathbf{X}_{V_L}^-$ is hit. Note that it is possible that neither $x^-,$ nor $\mathbf{X}_{V_L}^-$ after time $V_L,$ are visited, and in this case $\tilde{V}_L=\infty,$ i.e., we continue our recursive construction indefinitely. Otherwise, we stop the recursion at time $\tilde{V}_L$, and for each $k\geq\tilde{V}_L$ we define $(\mathbf{T}_{k}, (\boldsymbol{λ}_{y,z})_{{y \sim z, y,z\in \mathbf{T}_{k}}}, \mathbf{X}_{k}) := (\mathbf{T}_{\tilde{V}_L}, (\boldsymbol{λ}_{y,z})_{{y \sim z, y,z\in \mathbf{T}_{\tilde{V}_L}}}, \mathbf{X}_{\tilde{V}_L}).$ We also abbreviate $(\mathbf{T},\boldsymbol{λ},\mathbf{X} ):=(\mathbf{T}_{k}, (\boldsymbol{λ}_{y,z})_{{y \sim z, y,z\in \mathbf{T}_{k}}},\mathbf{X}_k)_{k\in{\N_0}}.$ This concludes the recursive construction of this triplet.

The process $(\mathbf{T},\boldsymbol{λ},\mathbf{X})$ is called {\em watershed process}, and we denote by
\begin{equation}\label{DEF-watershedMeasure}
	\mathbf{Q}_x^{κ,L}\text{ the law of the watershed process }(\mathbf{T},\boldsymbol{λ},\mathbf{X})
\end{equation}
starting at $x\in{\mathcal{X}} \setminus \{\emptyset\},$ with parameters $L\in{\N}$ and $κ>0.$ Similarly to the above, if we replace the evolving state space of $\mathbf X$ by a fixed tree $T$, under the law $P_x^T$ of the simple random walk $X$ from \eqref{DEF-RandomWalk}, we define $\tilde{V}_L=\tilde{V}_L(X)$ similarly as in \eqref{DEF-StoppingTimeHL}.
In the following proposition, we explain how the process $(\mathbf{T},\boldsymbol{λ},\mathbf{X})$ can be considered a random walk exploration of the initial Galton--Watson tree $\mathcal{T}$ from Section~\ref{SECTION-GWTree}.

\begin{prop}\label{PROP-SinglewatershedLikeGW}
	For all $ x\in{\mathcal{X}}\setminus \{\emptyset\},$ $κ>0,$ and $L\in{\N},$ the process
	$ (\mathbf{T},
	\boldsymbol{λ},
	\mathbf{X}) $ under $ \mathbf{Q}^{κ,L}_x $
	has the same law as
	$(\mathcal{T}^{X}_{k\wedge\tilde{V}_L},
	(λ_{y,z})_{y,z\in \mathcal{T}^{X}_{k\wedge\tilde{V}_L}},
	X_{k\wedge\tilde{V}_L})_{k\in\N_0}$
	under $\mathbb{E}^{{\rm GW}}[P^{\mathcal{T}}_x(\cdot)\,|\,\lambda_{x,x^-}=κ,x\in{\mathcal{T}}],$
	where:
	\begin{itemize}
		\item  conditionally on $(\mathcal{T},(\lambda_{y,z})_{y,z\in{\mathcal{T}}}),$ the process
		$ (X_n)$ is the random walk on ${\mathcal{T}}$  defined in Subsection~\ref{sec:GFF}.
		\item for $k \in \N,$ the set
		$ \mathcal{T}^{X}_k:=\{z\in{\mathcal{T}}:\,z\sim X_i\text{ for some }i\leq k-1\}$ is the subset of $\mathcal{T}$ adjacent to the trace of $ \{X_1, \ldots, X_{k-1}\}. $
	\end{itemize}
\end{prop}
\begin{proof}
	At time $k,$ for  $1\leq k\leq \tilde{V}_L,$ we sample the offspring of $\mathbf{X}_{k-1}$ independently of everything else via their conductances according to $\nu$ if it is the first time $\mathbf{X}_{k-1}$ was visited by $\mathbf{X};$ therefore, $\mathbf{T}_k$ is a Galton--Watson tree restricted to the offspring of the vertices explored by $\mathbf{X}$ before time $k-1,$ union with the edge $\mathbf{T}_0=\{x^-,x\}$. After time $\tilde{V}_L$ (if it is finite),  $\mathbf{T}_k$ stays constant equal to $\mathbf{T}_{\tilde{V}_L},$ and $\mathbf{X}_k$ constant equal to $\mathbf{X}_{\tilde{V}_L}.$

	Similarly, when $\mathbf{X}$ at time $ 1\le k\leq\tilde{V}_L$ performs a jump, the offspring of the point $ \mathbf{X}_{k-1} $ has already been generated according to $\nu$, either at step $ {k} $ or in a preceding step, and then $\mathbf{X}_{k-1}$ jumps to $\mathbf{X}_k$ with the probability
	$$ \frac{\boldsymbol{λ}_{\mathbf{X}_{k-1},\mathbf{X}_k}}{\boldsymbol{λ}_{\mathbf{X}_{k-1}}},$$
	which is analogous to \eqref{DEF-RandomWalk}. Hence both $ \mathbf{X} $ and $ X $ behave like a random walk on their respective trees until time $\tilde{V}_L,$ and $\tilde{V}_L$ corresponds for both walks to the first time either $x^-$ is hit, or $L$ different vertices have been visited by the walk, and then, denoting by $y$ the last of these $L$ vertices, $y^-$ has been hit. One can easily conclude.
\end{proof}

Let us finish this section with an observation which will be essential in the proof of Lemma~\ref{lem:probatobegood} below. For this purpose, first define under $\mathbf{Q}_x^{κ,L}$ the watershed $\mathbf{W}$ as the path of $\mathbf{X}$ until $V_L-1,$ that is
\begin{equation}
	\label{eq:defW}
	\mathbf{W}:=\{\mathbf{X}_0,\dots,\mathbf{X}_{V_L-1}\}.
\end{equation}
Using the convention $\boldsymbol{\lambda}_{y,yi}=0$ if $yi\notin{\mathbf{T}},$ by \eqref{DEF-StoppingTimeVn}, \eqref{DEF-StoppingTimeHL} and the construction of the weights $\boldsymbol{\lambda}_{y,z},$ $y\sim z\in{\mathbf{T}_k},$  we have under $\mathbf{Q}_x^{κ,L}$ that
\begin{equation}
	\label{eq:lambdaonWareiid}
	\begin{gathered}
		(\boldsymbol{\lambda}_{x,xi})_{i\in{\N}}=(\boldsymbol{\lambda}_{i}^{(1)})_{i\in{\N}},\text{ and if }{V}_L(\mathbf{X})<H_{x^-}(\mathbf{X}),\text{ then }
		\\\big\{(\boldsymbol{\lambda}_{y,yi})_{i\in{\N}}:\,y\in{\mathbf{W}\setminus\{x\}}\big\}=\big\{(\boldsymbol{\lambda}_{i}^{(k)})_{i\in{\N}}:\,k\in{\{2,\dots,L-1\}}\big\},
	\end{gathered}
\end{equation}
which follows simply from the fact that the conductances $(\boldsymbol{\lambda}_{y,yi})_{i\in{\N}}$ are equal to $(\boldsymbol{λ}_i^{(k)})_{i\in{\N}}$ if $y$ is  the $k$-th vertex visited by $\mathbf{X}.$

\subsection{Patching together watersheds}\label{Section-PatchingTogetherwatersheds}

In the previous subsection we explained how to construct a watershed process $(\mathbf{T},\boldsymbol{λ},\mathbf{X})$  starting at an  arbitrary vertex. We will now iteratively patch together  watersheds at the endpoints of previously generated watersheds.  The union $\mathcal{T}^{\mathbf{W}}_{-}$ of such watersheds will already constitute a transient subset of the random interlacements set on the Galton--Watson tree.
Embellishing $\mathcal{T}^{\mathbf{W}}_{-}$ with some further ``ends''  will yield a tree  $\mathcal T^{\mathbf{W}}$ which has the law of the weighted Galton--Watson tree we are interested in.

We will now give an informal description of this procedure and provide mathematical details below.
To patch the watersheds together, we will introduce another tree $ F,$  the \emph{tree of free points}.  This tree encodes the points at which watersheds will be patched together in the construction outlined above, i.e.\ $F$ is a tree in $\mathcal{X}$ and, at the same time, to each free point $ a\in F $ we associate another point $\widehat{a} \in \mathcal{X}$ -- which will turn out to also be an element of the tree $\mathcal{T}^{\mathbf{W}}$ to be constructed -- at which we will start a new watershed.
Patching up the watersheds through their vertices corresponding to free points, we will then be able to construct inductively the tree $ \mathcal{T}^{\mathbf{W}}_-$. We refer to Figure \ref{FIG} for an illustration.
\begin{figure}[ht]
	\centering
	\hspace{32pt}
	The tree $ \mathcal{T}^{\mathbf{W}}_-$ \hspace{135pt}
	The tree $F$ of free points
	\includegraphics[height=6cm]{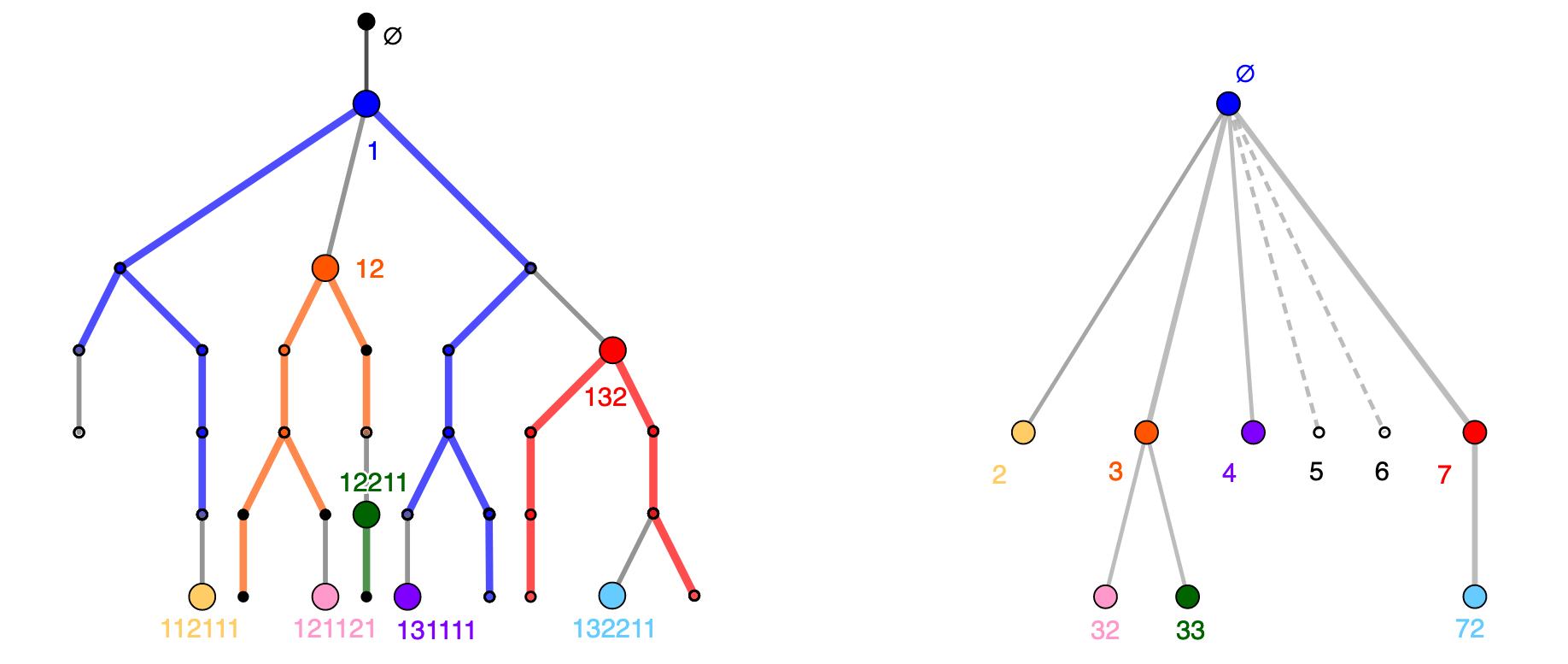}
	\caption{(A finite subset of) the tree $ \mathcal{T}^{\mathbf{W}}_-$, on the left, has some highlighted vertices, denoted by a coding $ \widehat{a},$ at which a new watershed is generated. Those points correspond to points in (a finite subset of) the tree of free points $ F $ on the right, where they have a different coding $a.$  For instance $\widehat{72}=132211$. We highlighted with different colors each  $a\in{F}$ on the right and on the left the corresponding point $\widehat{a}$ and the path on $\mathcal{T}^{\mathbf{W}}_-$ visited by the random walk $\mathbf{X}^a,$ which generates the watershed below $\widehat{a}.$ On the right, the points $5$ and $6$ are part of the tree of free points, but the corresponding vertices $\widehat{5}$ and $\widehat{6}$ do not appear yet on the left since they are below the 6th generation.}\label{FIG}
\end{figure}

We will define the weighted tree $F$ with weights denoted by $\lambda_{a,a'}^F,$ $a\sim a'\in{F},$ through a recursively defined sequence $(F_k)$ of weighted trees, such that to each $a\in{F_{k-1}}$ we associate a watershed $(\mathbf{T}^{a},\boldsymbol{λ}^{a},\mathbf{X}^{a})$ starting in $\widehat{a}$ as defined in the last subsection, and to each vertex $a\in{F_k}$ we associate another vertex $\widehat{a}\in{\mathcal{X}}$.

As explained above, this construction of $F$ as well as the corresponding watersheds, will depend on a parameter $L\in{\N},$ that we fix for the rest of this section. We denote by $\mathbf{P}^{W}_{L}$  the probability measure under which these objects are constructed. For technical reasons, we will start the first watershed in the point $1$ instead of $\emptyset.$

First set $ F_{-1}:=\emptyset $,
\(
F_0:=\set{\emptyset}
\)
take $ \widehat{\emptyset}=1,$  and  generate some weights $(\lambda_{\emptyset,i}^{\mathbf{W}})_{i\in{\N}}$ with law $\nu.$  Now assume $ F_{k-1} $ and $ F_k $ are given for some $k \in \N_0,$ and that each point $ a\in F_k $ is associated to a point $ \widehat{a}\in \mathcal{X} $.  We define $ F_{k+1} $ as follows.
For each $a\in{F_k\setminus F_{k-1}},$  we generate
\begin{equation}
	\label{eq:defwatershedata}
	\text{ an independent watershed $(\mathbf{T}^{{a}}, \boldsymbol{λ}^{{a}},\mathbf{X}^{{a}}) $ with law } \mathbf{Q}^{λ_{a^-,a}^F,L}_{\widehat{a}},
\end{equation}
as defined in \eqref{DEF-watershedMeasure}. Note that $\emptyset^-$ is not well-defined, but for $a = \emptyset$ we will take the convention
\begin{equation} \label{eq:freeRoot}
	λ_{a^-,a}^F:=\lambda_{\emptyset,1}^{\mathbf{W}}.
\end{equation}

The watershed $(\mathbf{T}^a,\boldsymbol{λ}^a,\mathbf{X}^a)$ will be used to encode the set of \emph{free points} via the following set
\begin{equation}\label{DEF-FreePoints}
	\mathfrak{F} _a := (∂\mathbf{T}^{a}_{V_L}) \setminus \{
	\mathbf{X}^a_{V_L}\};
\end{equation}
in other words, apart from $\mathbf{X}_{V_L}^a,$ the set
$\mathfrak{F}_a$ corresponds to the vertices on the boundary of the tree $\mathbf{T}^a$ once the walk has either visited $L$ vertices or hit $\widehat{a}^-.$  The vertex $\mathbf{X}_{V_L}^a$ is excluded from this set since, by definition of $V_L,$ the first generation of the tree below $\mathbf{X}_{V_L}^a$ has already been explored by $\mathbf{T}^a$.  Equivalently, the points in $ \mathfrak{F}_a $ are vertices not visited by the random walk $ \mathbf{X}^{{a}}_k,$ $1 \le k \le V_L,$ but adjacent to its trace, and which have thus already been generated during the construction of the watershed.
We will then generate new watersheds from the vertices in $\mathfrak{F}_a.$ We can now define the next generation of the tree of free points
\begin{equation}
	\label{eq:deffreepointrecursively}
	F_{k+1}:=F_k 		\cup
	\bigcup_{a\in F_k\setminus F_{k-1}} \,
	\bigcup_{i=2}^{\ab{\mathfrak{F} _a}} 		\set{ai}.
\end{equation}
In other words, the sets of points $\mathfrak{F}_a,$ $a\in{F_k\setminus F_{k-1}},$ are used to build the $(k+1)$-st level of the tree of free points, and we define $\widehat{\,ai\,}$ as the $i$-th element (in lexicographic order) of $\mathfrak{F}_a$ for each $1\leq i\le |\mathfrak{F}_a|.$  Note that the union over $i$ starts at $2$ for technical reasons, cf.\ property \ref{item:ii} in Definition~\ref{DEF-Goodwatershed}, and the explanation in the second paragraph thereafter. In particular, $\widehat{a}1$ is well-defined but not part of the tree $F,$ for instance $\widehat{1}=1111$ in Figure~\ref{FIG}.

We moreover define the conductance of the edge above the vertex ${ai}$ for $F_{k+1}$ as
\begin{equation}
	\label{eq:deflambdaF}
	λ_{a,ai}^F:= \boldsymbol{λ}_{(\widehat{\,ai\,})^-,\widehat{\,ai\,}}^a,
\end{equation}
whereas the conductances on $F_k\subset F_{k+1}$ stay the same as before. This concludes the inductive definition of the sequence $(F_k),$ and the tree of free points is simply defined via
\begin{equation}\label{DEF-AllFreePoints}
	F:= \bigcup_{k\in{\N_0}} F_k,
\end{equation}
endowed with the same conductances as the $F_k,$ $k\in{\N_0}.$

Let us now explain how to construct a Galton--Watson tree by gluing together the watersheds $(\mathbf{T}^a,\boldsymbol{λ}^a,\mathbf{X}^a),$ $a\in{F}.$ We first set
\begin{equation}\label{DEF-WUnion}
	\mathcal{T}^{\mathbf{W}}_{-}:=
	\Big\{2,\dots,\pi\big(\big(\lambda^{\mathbf{W}}_{\emptyset,j} \big)_{j\in{\N}} \big)\Big\}
	\cup
	\bigcup_{a\in{F}} \mathbf{T}^{a};
\end{equation}
in other words, $\mathcal{T}_-^{\mathbf{W}}$ consists of a first generation with weights $(\lambda^{\mathbf{W}}_{\emptyset,j})_{j\in{\N}},$ and the union of the watersheds $\mathbf{T}^a,$ $a\in{F};$ note that the root $\emptyset$ belongs to $\mathbf{T}^{\emptyset}$ by \eqref{eq:defwatershedata} and the convention $\widehat{\emptyset}=1,$ cf.\ \eqref{eq:freeRoot} also, and in particular $\emptyset \in \mathcal{T}^{\mathbf{W}}_{-}$.
One can view $\mathcal{T}^{\mathbf{W}}_{-}$ as a tree in $\mathcal{X},$ and we endow each of its edges $\{x,y\}$ such that $x,y\in{\mathbf{T}^a}$ for some $a\in{F}$ with the conductance $\boldsymbol{λ}_{x,y}^a.$ Note that each edge $\{x,y\}$ of $\mathcal{T}^{\mathbf{W}}_{-}$ is also an edge of $\mathbf{T}^a$ for some $a\in{F},$ and in fact, for each $a\in{F},$ $\mathbf{T}^a$ and $\mathbf{T}^{a^-}$ have exactly one edge in common: $\{\widehat{a}^-,\widehat{a}\}.$ Moreover, in view of \eqref{eq:defwatershedata} and \eqref{eq:deflambdaF}, $\boldsymbol{λ}_{\widehat{a}^-,\widehat{a}}^{a}=\lambda_{a^-,a}^F=\boldsymbol{λ}_{\widehat{a}^-,\widehat{a}}^{a^-},$ hence the conductances of the tree $\mathcal{T}^{\mathbf{W}}_-$ are uniquely defined.

Observe that the tree $ \mathcal{T}^{\mathbf{W}}_{-} $ is not yet a Galton--Watson tree with the desired offspring distribution since for some vertices $x\in{\mathcal{T}^{\mathbf{W}}_{-}}$ we did not construct their descendants: this is the case if
$x=\widehat{\,a1\,}$ for some $a\in{F}$
(see \eqref{eq:deffreepointrecursively}), or if $x$ is in the boundary of $\mathbf{T}^a_{\tilde{V}_L}\setminus \mathbf{T}^a_{V_L}$ (since no vertices correspond to free points in this part of the watershed). Therefore, we now add some ends to those points in order to complete the construction of the Galton--Watson tree. More precisely, define independently of everything else
\begin{equation}
	\label{eq:defrestofthetrees}
	\begin{gathered}
		\text{ an independent family of Galton--Watson trees } (\mathcal{T}^x)_{x\in{\mathcal{X}}},
		\\\text{each $\mathcal{T}^x$ with the same law as } x\cdot\mathcal{T}\text{ under }\Pgw.
	\end{gathered}
\end{equation}
In other words, $\mathcal{T}^x$ is a Galton--Watson tree rooted at $x.$
We now define $ \mathcal{T}^{\mathbf{W}} $ as the weighted tree obtained from the union of $ \mathcal{T}^{\mathbf{W}}_{-} $ with the $\mathcal{T}^x,$  $x\in{\partial \mathcal{T}^{\mathbf{W}}_{-}},$ endowed with their respective conductances, and we denote by $\lambda^{\mathbf{W}}$ the conductances on $\mathcal{T}^{\mathbf{W}}.$ We then have that for all $L\in{\N},$
\begin{equation}
	\label{eq:TWisGW}
	\mathcal{T}^{\mathbf{W}}	\text{ has the same law under }\mathbf{P}^{W}_{L}\text{ as the Galton--Watson tree }\mathcal{T}\text{ under } \Pgw;
\end{equation}
indeed, it follows from Proposition~\ref{PROP-SinglewatershedLikeGW} and \eqref{eq:defwatershedata} that, conditionally on $\mathbf{T}^{a'},$ $a'\in{F_{k-1}},$ a single watershed $\mathbf{T}^a,$ $a\in{F_k\setminus F_{k-1}},$ has the same law as a Galton--Watson tree restricted to this watershed, conditionally on $\boldsymbol{λ}_{\widehat{a}^-,\widehat{a}}^a=λ_{a^-,a}^F.$ Since
$λ_{a^-,a}^F
=   \boldsymbol{λ}_{\widehat{a}^-,\widehat{a}}^{a^-}
(=  \boldsymbol{λ}_{\widehat{a}^-,\widehat{a}}^{a}  )$ by \eqref{eq:defwatershedata} and \eqref{eq:deflambdaF} we obtain that the conductances between each vertex $x\in{\mathcal{T}^{\mathbf{W}}_{-}\setminus\partial\mathcal{T}^{\mathbf{W}}_{-}}$ and its offspring are distributed independently according to $\nu.$ Note that, for each  $x\in{\partial\mathcal{T}^{\mathbf{W}}_-}$, the subtree $\mathcal{T}^{\mathbf{W}}_x:= (\mathcal{T}^{\mathbf{W}})_x$ equals $\mathcal{T}^x$ with the desired offspring distribution by definition in \eqref{eq:defrestofthetrees} and below, and we conclude that \eqref{eq:TWisGW} holds true.

\subsection{Watersheds and random interlacements}
\label{sec:RWandRI}

In the previous subsections, we generated simultaneously the Galton--Watson tree and random walks on it through the structure of watersheds. The next goal now is to interpret these random walks as a part of a random interlacements process, which will essentially follow from Theorem~\ref{Thm-Interlacement and RW} and some additional conditions as in \eqref{EQ-ConditionsWatershedsRandInt}. Under some probability measure $\mathbb{P}_{\widetilde{u}}^{\Gamma},$ $\widetilde{u}>0,$ let
\begin{equation}\label{DEF-Poissonian}
	(\Gamma_x)_{x\in{\mathcal{X}}}\text{ be an i.i.d.\ family of } \Poi(\widetilde{u})\text{ random variables}.
\end{equation}
We denote by $\mathbf{P}^{W}_{L,\widetilde{u}}$ the product measure $\mathbf{P}^{W}_{L}\otimes\mathbb{P}_{\widetilde{u}}^{\Gamma},$ under which the tree $\mathcal{T}^{\mathbf{W}}$ and the Poisson random variables $(\Gamma_x)_{x\in{\mathcal{X}}}$ are independent. Furthermore,  for $a\in F$ let

\begin{equation}\label{DEF-watershed}
	\mathbf{W}^{a} := \big \{\mathbf{X}^{{a}}_k \, : \, k\in \{0, \ldots, V_L(\mathbf{X}^{{a}})-1\} \big\}.
\end{equation}

Recall the definition of ${e}_{K,T}$ from \eqref{DEF-EquilibriumMeasure}. 

\begin{prop}\label{PROP-W_Gw_WInterl}
	Let $\widetilde{u},u>0$ and $L\in{\N}.$ On some extension of the probability space corresponding to $\mathbf{P}^{W}_{L,\widetilde{u}},$  one can couple  $\mathcal{T}^{\mathbf{W}}$ defined in \eqref{eq:TWisGW} and a set $\mathcal{I}^u$ in such a way that conditionally on $\mathcal{T}^{\mathbf{W}}$, the set $\mathcal{I}^u$ is an interlacements set at level $u$ on $\mathcal{T}^{\mathbf{W}},$ and for all $ a\in F, $ if
	\begin{equation}\label{EQ-ConditionsWatershedsRandInt}
		Γ_{\widehat{a}}\ge 1,   \quad
		\tilde{V}_L(\mathbf{X}^a)=∞,
		\quad
		\text{ and } \quad  u \ge\frac{\widetilde{u}}{{e}_{\{\widehat{a}\},\mathcal{T}^{\mathbf{W}}_{\widehat{a}}}({\widehat{a}})},
	\end{equation}
	where $\mathcal{T}^{\mathbf{W}}_{\widehat{a}}$ is the subtree of $\mathcal{T}^{\mathbf{W}}$ below $\widehat{a}$, then
	\begin{equation*}
		\mathbf{W}^{a} \subset \mathcal{I}^u.
	\end{equation*}

\end{prop}
\begin{proof}
	Conditionally on $\mathcal{T}^{\mathbf{W}},$ for each $a\in{F},$ define $\overline{\mathbf{X}}^{a}$ as a process on  $\mathcal{T}^{\mathbf{W}}$ such that $\overline{\mathbf{X}}^a_k=\mathbf{X}^a_k$ for $0 \le k\leq\tilde{V}_L(\mathbf{X}^a),$ and such that, if $\tilde{V}_L(\mathbf{X}^a)<\infty,$ the process $\overline{\mathbf{X}}^a_k,$ $k\geq\tilde{V}_L(\mathbf{X}^a),$ is a random walk on $\mathcal{T}^{\mathbf{W}}$ starting in $\mathbf{X}^a_{\tilde{V}_L(\mathbf{X}^a)}.$
	On some extension of the probability space corresponding to $\mathbf{P}^{W}_{L,\widetilde{u}},$ conditionally on $\mathcal{T}^{\mathbf{W}},$ start independently from each $x\in \mathcal{T}^{\mathbf{W}}$ i.i.d.\ random walks $\mathbf{X}^{x,i},$ $i\geq2,$ each with law ${P}_x^{\mathcal{T}^{\mathbf{W}}}(\, \cdot\,|\,H_{x^-}=\infty),$ with the convention $H_{\emptyset^-}=\infty.$ Moreover, take $\mathbf{X}^{x,1}=\overline{\mathbf{X}}^a$ if $x=\widehat{a}$ for some $a\in{F}$ and $H_{\widehat{a}^-}(\overline{\mathbf{X}}^a)=\infty,$ and otherwise let $\mathbf{X}^{x,1}$ be some other independent walk with law ${P}_{x}^{\mathcal{T}^{\mathbf{W}}}(\, \cdot\,|\,H_{x^-}=\infty)$. Taking advantage of the thinning property for Poisson random variables and Proposition~\ref{PROP-SinglewatershedLikeGW}, one can easily prove that, conditionally on $\mathcal{T}^{\mathbf{W}}$ and for each $a\in{F},$ the probability
	$\mathbf{P}^{W}_{L, \widetilde u}(\Gamma_{\widehat{a}}\geq 1,  H_{\widehat{a}^-}(\overline{\mathbf{X}}^a)=\infty)$ is smaller than or equal to the probability that a $\text{Poi}(\widetilde{u} P_{\widehat{a}}^{\mathcal{T}^{\mathbf{W}}}(H_{\widehat{a}^-}=∞))$-distributed random variable is larger or equal to one. Noting that $\tilde{V}_L(\mathbf{X}^a)=\infty$ implies $H_{\widehat{a}^-}(\overline{\mathbf{X}}^a)=\infty,$ and taking advantage of the equality
	\begin{equation*}
		{e}_{\{\widehat{a}\},\mathcal{T}^{\mathbf{W}}_{\widehat{a}}}({\widehat{a}})
		\stackrel{\eqref{DEF-EquilibriumMeasure}}{=}
		\lambda^{\mathbf{W}}_{\widehat{a},+}P_{\widehat{a}}^{\mathcal{T}_{\widehat{a}}^{\mathbf{W}}}(\tilde{H}_{\widehat{a}}=∞)=\lambda^{\mathbf{W}}_{\widehat{a}}P_{\widehat{a}}^{\mathcal{T}^{\mathbf{W}}}(\tilde{H}_{\widehat{a}}=∞,H_{\widehat{a}^-}=\infty)
		\stackrel{\eqref{DEF-cy}}{=}
		\frac{\widecheck{e}_{\mathcal{T}^{\mathbf{W}}}(\widehat{a})}{P_{\widehat{a}}^{\mathcal{T}^{\mathbf{W}}}(H_{\widehat{a}^-}=∞)},
	\end{equation*}
	one can construct conditionally on $\mathcal{T}^{\mathbf{W}}$ for each $x\in{\mathcal{T}^{\mathbf{W}}}$ a Poisson random variable $Γ'_{x}$ with parameter $u\widecheck{e}_{\mathcal{T}^{\mathbf{W}}}(x)$ such that for each $a\in{F},$ the properties in \eqref{EQ-ConditionsWatershedsRandInt} already entail that $\Gamma'_{\widehat{a}}\geq1.$

	Moreover, conditionally on $\mathcal{T}^{\mathbf{W}},$ introduce $ \overleftrightarrow{\mathbf{X}}^{x,i},$ $i\geq1,$ as doubly infinite random walk trajectories on $\mathcal{T}^{\mathbf{W}},$
	whose forward part is defined to be $\mathbf{X}^{x,i},$ and whose backward part is an independent random walk with law $ P_{x}^{\mathcal{T}^{\mathbf{W}}}(\, \cdot  \,|\, H_{x^-}=\infty,\tilde{H}_{x}=∞)$ for each $x\in{\mathcal{T}^{\mathbf{W}}}.$ By Proposition~\ref{PROP-SinglewatershedLikeGW}, conditionally on ${\mathcal{T}^{\mathbf{W}}},$ the process $ \overleftrightarrow{\mathbf{X}}^{x,i}$ has law $\overline{Q}_{x}^{\mathcal{T}^{\mathbf{W}}}$ for each $i\geq1,$ see \eqref{DEF-Measure barQ_x}. We can now define $\mathcal{I}^u$ as the set of vertices visited by any of the trajectories $\overleftrightarrow{\mathbf{X}}^{x,i},$ $i\in{\{1,\dots,Γ'_{x}\}}$ and $x\in{{\mathcal{T}^{\mathbf{W}}}},$ which has the same law conditionally on $\mathcal{T}^{\mathbf{W}}$ as under $\Pri_{\mathcal{T}^{\mathbf{W}}}$ by Theorem~\ref{Thm-Interlacement and RW}. Since \eqref{EQ-ConditionsWatershedsRandInt} implies $\Gamma'_{\widehat{a}}\geq1$ and $\mathbf{X}^{\widehat{a},1}_k={\mathbf{X}}_k^a$ for each $k\in{\N_0},$  we can easily conclude by the definition \eqref{DEF-watershed} of $\mathbf{W}^a.$

\end{proof}

\section{Percolation of the level set}\label{SECTION-Percolation}

In this section we prove Theorems~\ref{THM-h*>0} and \ref{THM-quenchednoise}. We first define a set of ``good'' properties, see Definition~\ref{DEF-Goodwatershed} below, which can be satisfied by a vertex $a$ in the tree of free points $F,$ as defined in Section~\ref{Section-PatchingTogetherwatersheds}. We will show in Lemma~\ref{lem:probatobegood} that $a$ is good with not too small probability.
Our notion of goodness is chosen so that on the one hand, the watershed associated to each good free point is included in the interlacements set $\mathcal{I}^u$ from Proposition~\ref{PROP-W_Gw_WInterl}, see  Proposition~\ref{PROP-Prop1}, and also included in the set $A_u$ from \eqref{PROP1-DEF-Au}  with high probability, see Proposition~\ref{PROP-Percolation-Au}; on the other hand, it also ensures that the tree of good free points survives, see Proposition~\ref{PROP-Prop1}. We refer to the discussion below Definition~\ref{DEF-Goodwatershed} for more details. This readily yields the percolation of the set $A_u\cap \mathcal{I}^u,$ and an application of the inclusion \eqref{eq:Iu+Auincluded}, which follows from Proposition~\ref{THM-Isomorphism} and Proposition~\ref{prop:capRWinfty} below, completes the proof of Theorems~\ref{THM-h*>0} and \ref{THM-quenchednoise}.

Let us now define the properties which make a free point good. For this purpose, recall the watershed $(\mathbf{T}^a,\boldsymbol{λ}^a,\mathbf{X}^a)$ from \eqref{eq:defwatershedata}, where $a\in F,$ with $F$ the tree of free points defined in \eqref{DEF-AllFreePoints}.
We recall that in this watershed, $\mathbf{X}^a$ is a random walk stopped at time $\tilde{V}_L(\mathbf{X}^a),$ see \eqref{DEF-StoppingTimeHL}, and for $K\subset \mathbf{T}^{a}$ we denote by $H_K(\mathbf{X}^a)$  the hitting time of $K$ for this stopped random walk similarly to \eqref{DEF-stoppingTimes}. Recall also the definition of the set $\mathbf{W}^a$ from \eqref{DEF-watershed} and of the Poisson random variable $Γ_{\widehat{a}}$ from \eqref{DEF-Poissonian}. Also recall that when $x\in{\partial\mathcal{T}_-^{\mathbf{W}}},$ the tree $\mathcal{T}^x,$ see \eqref{eq:defrestofthetrees}, is equal to the Galton--Watson tree below $x$ in $\mathcal{T}^{\mathbf{W}}.$ Finally, recall that for a set $A\subset\mathcal{X},$ by $G_x^A$ we denote the set of children of $x$ in $A,$ see \eqref{DEF-GenerationOfTree}, and for a transient tree $T,$ by $g^T$ we denote the Green function on $T,$ see below \eqref{DEF-GreenFunction}.

\begin{deff}\label{DEF-Goodwatershed}
	Let $\widetilde{u},$ $B,$ $c_λ,$ $C_Λ,$ $C_g$ be positive real numbers, $L\in{\N}$ and $c_f\in{(0,1]}.$ Under $\mathbf{P}^{W}_{L,\widetilde{u}},$ we say that $a\in F$  is   $(L, B, c_λ,C_Λ,C_g, c_f)$-\emph{good} if the corresponding watershed  $(\mathbf{T}^{a},\boldsymbol{λ}^a,\mathbf{X}^a),$ the weighted tree $\mathcal{T}^{\widehat{a}1}$ and the Poisson random variable $ Γ_{\widehat{a}}$ satisfy the following properties:
	\begin{enumerate}[i)]
		\item \label{item:i}
		The Poisson variable $ Γ_{\widehat{a}}$ satisfies
		\(
		Γ_{\widehat{a}} \ge 1.
		\)
		\item \label{item:ii}
		The watershed satisfies
		\begin{equation}
			\label{LEMMA Goodwatershed-EQ-ConductanceBounded}
			\big| G_{\widehat{a}}^{\mathbf{T}^{a}} \big| \ge 2,\,
			\boldsymbol{λ}^a_{\widehat{a},\widehat{a}1}> c_λ\text{ and }(\boldsymbol{\lambda}^a)_{\widehat{a},+}\leq C_{\Lambda},
		\end{equation}
		and the weighted tree $\mathcal{T}^{\widehat{a}1}$ satisfies
		\begin{equation}
			\label{Goodwatershed-EQ-GreenFunctionBounded}
			g^{\mathcal{T}^{\widehat{a}1}}(\widehat{a}1,\widehat{a}1) \le C_g.
		\end{equation}

		\item   \label{item:iii}
		The trajectory $\mathbf{X}^a$ satisfies
		\begin{equation*}
			H_{\set{\widehat{a}^-, \widehat{a}1 }}(\mathbf{X}^{{a}})=\tilde{V}_L(\mathbf{X}^{{a}})=\infty.
		\end{equation*}

		\item \label{item:iv}
		The set of children of the vertex $a$ in the tree of free points $F$
		satisfies
		\begin{equation*}
			\big |\big\{a'\in{G_{a}^F}:\,λ_{a,a'}^F\le C_Λ\big \}\big|\geq c_fL.
		\end{equation*}
		\item   \label{item:v}
		The conductances $\boldsymbol{λ}^a$ on $\mathbf{W}^a$ satisfy
		\begin{equation}
			\label{eq:boundonconductances}
			\frac{1}{L^{\frac32}}  \sum_{y\in \mathbf{W}^{a}} 	(\boldsymbol{λ}^a_y)^{\frac32}    <
			B.
		\end{equation}
	\end{enumerate}
\end{deff}

We now explain how the good properties defined above can be combined in order to deduce the percolation of $A_u\cap\mathcal{I}^u,$ see \eqref{PROP1-DEF-Au}. The first three properties imply that the conditions in \eqref{EQ-ConditionsWatershedsRandInt} are verified, see the proof of Proposition~\ref{PROP-GoodwatershedsInInterlacements}, and so, in view of Proposition~\ref{PROP-W_Gw_WInterl}, the set $\mathbf{W}^a$ of the watershed associated to a good free point $a\in{F}$ is included in the coupled interlacements set $\mathcal{I}^u.$  More precisely, property \ref{item:i} implies the first condition in \eqref{EQ-ConditionsWatershedsRandInt};
property \ref{item:ii} will imply a lower bound on ${e}_{\{\widehat{a}\},\mathcal{T}^{\mathbf{W}}_{\widehat{a}}}({\widehat{a}}),$ and thus that the third assumption in \eqref{EQ-ConditionsWatershedsRandInt} is satisfied for $u$ of the same order as $\widetilde{u},$ see \eqref{PROP1-DEF-Uprime}; and property \ref{item:iii} implies that the second condition in \eqref{EQ-ConditionsWatershedsRandInt} is satisfied.
Property \ref{item:iv} ensures the creation of many new free points with bounded conductances to their parent, which will imply --  using Lemma~\ref{LEMMA-DependBernPercolation} below -- that the tree of good free points contains a $d$-ary tree for arbitrarily large $d,$ see Proposition~\ref{PROP-Prop1}. Finally, using \eqref{eq:boundprobaAu}, property \ref{item:v} will provide us with a good bound on the probability that $\mathbf{W}^a\subset A_u.$ Combining these five properties we will thus obtain percolation of the free points $a\in{F}$ such that $\mathbf{W}^a\subset A_u\cap\mathcal{I}^u,$ and thus percolation of $A_u\cap\mathcal{I}^u,$ see Proposition~\ref{PROP-Percolation-Au}.

One of the main difficulties in the previous steps is to understand how property  \ref{item:ii} in our notion of goodness is used to bound the equilibrium measure ${e}_{\{\widehat{a}\},\mathcal{T}^{\mathbf{W}}_{\widehat{a}}}({\widehat{a}})$ from below, which implies that we can find $\tilde{u}$ and $u$ of the same order verifying the third assumption \eqref{EQ-ConditionsWatershedsRandInt}, and, consequently,   that there is a random interlacements trajectory starting in $\widehat{a}$ when $a$ is good. When $\widehat{a}1$ is not visited by $\mathbf{X}^a$, which is the case when $a$ is good by property \ref{item:iii}, then $\widehat{(a1)}=\widehat{a}1$, so no new watershed is generated starting from $\widehat{a}1$ in view of \eqref{eq:deffreepointrecursively}, and thus $\widehat{a}1\in{\partial\mathcal{T}^{\mathbf{W}}_-}.$ Therefore, by the construction of the tree $\mathcal{T}^{\mathbf{W}}$ above \eqref{eq:TWisGW}, we obtain that if $a$ is good, then $\mathcal{T}^{\widehat{a}1}$ is the tree below $\widehat{a}1$ in $\mathcal{T}^{\mathbf{W}}.$ The bound on the Green function on $\mathcal{T}^{\widehat{a}1}$ combined with \eqref{LEMMA Goodwatershed-EQ-ConductanceBounded} in property \ref{item:ii} will then imply the desired lower bound on ${e}_{\{\widehat{a}\},\mathcal{T}^{\mathbf{W}}_{\widehat{a}}}({\widehat{a}}),$ see \eqref{PROP1-EQ-ConstantCE} for details. In other words, the reason we excluded ${a}1$ from the tree of free points in \eqref{eq:deffreepointrecursively} is to make sure that $\mathcal{T}^{\widehat{a}1}$ is the tree below $\widehat{a}1$ in $\mathcal{T}^{\mathbf{W}},$ and thus that we can use the independent tree $\mathcal{T}^{\widehat{a}1}$ to bound ${e}_{\{\widehat{a}\},\mathcal{T}^{\mathbf{W}}_{\widehat{a}}}({\widehat{a}})$ without using any information on the other watersheds in $\mathcal{T}^{\mathbf{W}}.$

We now provide lower bounds on the probabilities of the previous properties in the following lemma. Note that in items \ref{item:lemma ii} to \ref{item:lemma v} below we do not consider exactly the same kind of events as in Definition~\ref{DEF-Goodwatershed}; they do, however, present the advantage of having more independence and we will show in Lemma~\ref{lem:probatobegood}  (see for instance \eqref{eq:propertyiiiissame}) that the probabilities of the events from Definition~\ref{DEF-Goodwatershed} are larger than those of the events from Lemma~\ref{LEMMA-Goodwatershed}.  Recall that $(\Gamma_x)_{x\in{\mathcal{X}}}$ are Poisson random variables with parameter $\widetilde{u}$ under of $\mathbb{P}_{\widetilde{u}}^{\Gamma},$ see \eqref{DEF-Poissonian}, that $(\lambda_i)_{i\geq0}$ under $\nu$ represents the law of the weights below any vertex, and that $\mathbf{Q}_x^{κ,L}$ denotes the law of the watershed introduced in Section~\ref{Section-Singlewatershed}, see \eqref{DEF-watershedMeasure}. Recall also the definition of the (interior) boundary $\partial A$ of a set $A\subset\mathcal{X}$ from the paragraph below \eqref{DEF-UlamHarris}, and to simplify notation for $B\subset A$ we will write $\partial A\setminus B$ for $(\partial A)\cap B^c$.

\begin{lemma}\label{LEMMA-Goodwatershed}
	There exist positive constants $c_λ,C_Λ,C_g,c_V,c_f\in{(0,\infty)}$ such that for each $ε\in{(0,1)}$ and $B>0,$ there exists $L_0=L_0(B,ε)\in{\N}$ such that  for all $x\in{\mathcal{X}\setminus\{\emptyset\}},$ $L\geq L_0,$ $κ\leq C_{\Lambda}$ and $\widetilde{u}>0,$ the following properties hold true:

	\begin{enumerate}[i)]
		\item \label{item:lemma i}$ \mathbb{P}_{\widetilde{u}}^{Γ}(Γ_x\geq 1)=1-\exp(-\widetilde{u}), $

		\item \label{item:lemma ii} $ ν\Big(\pi((\lambda_i)_{i\in{\N}}) \ge 2,\lambda_1> c_λ,\lambda_2>c_{\lambda},\lambda_+\leq C_{\Lambda}\Big)
		\geq\frac12( 1- μ(1)), $

		$ \mathbb{P}^{{\rm GW}} \big(g^{x1\cdot\mathcal{T}}(x1,x1) \le C_g\big) \geq \frac12, $

		\item \label{item:lemma iii}$ \mathbb{E}^{{\rm GW}}\Big[\dfrac{c_λ\lambda_{x2,x21}}{2C_Λ(2C_{\Lambda}+\lambda_{x2,+})}P_{x21}^{\mathcal{T}}(\tilde{V}_{L-2}= {H}_{x2}=\infty)\,\big|\,x\in{\mathcal{T}},\,\pi((\lambda_{x,xi})_{i\in{\N}}) \ge 2
		\Big]= c_V,  $

		\item \label{item:lemma iv}$ \mathbf{Q}_{x}^{κ,L} \Big( \big |\big\{y\in{\partial\mathbf{T}_{V_L}\setminus\{x1,\mathbf{X}_{V_L}\}}:\,\boldsymbol{λ}_{y,y^-}\le C_Λ\big \}\big|< c_fL,\tilde{V}_L(\mathbf{X})=\infty\Big)	\leq ε, $

		\item \label{item:lemma v}$ \mathbf{Q}_{x}^{κ,L}\bigg(
		\dfrac{1}{L^{\frac32}} \displaystyle\sum_{y\in{\mathbf{W}}} 	\big(\boldsymbol{\lambda}_y\big)^{\frac32}    \geq
		B,\tilde{V}_L(\mathbf{X})=\infty
		\bigg)\leq ε. $
	\end{enumerate}
\end{lemma}

\begin{proof}
	\begin{enumerate}[i)]
		\item This  is immediate from the definition in \eqref{DEF-Poissonian}.
		\item First note that $\nu\big(\pi((\lambda_i)_{i\in{\N}}) \ge 2\big)=1-\mu(1)$ by definition \eqref{DEF-OffspringDistribution} of $\mu$  in combination with our assumption \eqref{REMARK-TInfiniteDescendants} in Subsection \ref{SECTION-Pruning}. Moreover, $\mathcal{T}$ is $ \Pgw \barAS $ transient due to Proposition~\ref{prop:GWtransient}. Therefore, the Green function $g^{x1\cdot\mathcal{T}}(x1,x1)$ associated to the tree $\mathcal{T}$ rooted at $x1$ is $ \Pgw\barAS$ finite, and its law does not depend on the choice of $x.$ Since probability measures are continuous from below, by definition of the conductances in \eqref{DEF-Conductances+} and above, one can find a small enough positive constant $c_{\lambda}$ as well as large enough finite constants $C_{\Lambda}$ and $ C_g,$ independent of $x,$ such that \ref{item:lemma ii} holds uniformly in $x \in \mathcal X.$

		\item  Note that for each $y\in{\mathcal{T}}\setminus\{\emptyset\},$ since the subtree $ \mathcal{T}_{y^-} $ is $ \as $ transient,  for almost all realizations of $\mathcal T,$ the probability $ P_{y}^{\mathcal{T}}(H_{y^-}=\infty) $ is strictly positive. Therefore, using the strong Markov property at time $V_{L-2}$ -- which is finite and larger than $H_{x2}$ with positive probability under $P_{x21}^\mathcal T,$ see its definition in \eqref{DEF-StoppingTimeVn} -- and using the previous with $y=X_{V_{L-2}},$ it follows from the definition of $\tilde{V}_{L-2}$ in \eqref{DEF-StoppingTimeHL} that the variable appearing in the $\mathbb{P}^{{\rm GW}}$-expectation of \ref{item:lemma iii} is a.s.\ positive, and we can conclude.

		\item We will use twice the weak law of large numbers for the i.i.d.\ sequence of weights $(\boldsymbol{\lambda}_i^{(k)})_{i\in{\N}},$ $k\geq2,$ from \eqref{eq:deflambdaik}. For this purpose, from the proof of \ref{item:lemma ii} we recall that $\nu( \pi(({\lambda}_{i})_{i\in{\N}})\geq 2) = 1-μ(1)>0.$ As a consequence, the sequence of random variables
		$ |\{k\in{\{2,\dots,L\}}:\,\pi((\boldsymbol{\lambda}_{i}^{(k)})_{i\in{\N}})\ge 2\}|/L,$ $L \in \N,$ converges to $ 1-μ(1)$ in probability as $L\to \infty$ by \eqref{eq:deflambdaik}.
		Fixing $ c_f\in (0, (1-μ(1))/2),$ we obtain for $L$ large enough that
		\begin{equation}
			\label{eq:firstpartproofiv}
			\mathbf{Q}_{x}^{κ,L} \Big( \big|\big\{k\in{\{2,\dots,L-1\}}:\, \pi(({\boldsymbol{\lambda}}_{i}^{(k)})_{i\in{\N}})\geq2\big\}\big| 	< 2L c_f
			\Big)	\leq \frac{ε}{2}.
		\end{equation}

		Similarly, fixing $C_\Lambda$ large enough so that
		\begin{equation*}
			ν\Big(\sum_{i}\lambda_i\leq C_{\Lambda}\Big)> 1-c_f,
		\end{equation*}
		we have by \eqref{eq:deflambdaik} that for $L$ large enough
		\begin{equation}
			\label{eq:secondpartproofiv}
			\mathbf{Q}_{x}^{κ,L}\Big(\big |\big\{k\in{\{2,\dots,L-1\}}:\,\sum_{i\in{\N}}\boldsymbol{\lambda}_{i}^{(k)}\le C_Λ\big \}\big|< (1-c_f)L\Big)\leq \frac{ε}{2}.
		\end{equation}
		Recalling the notation $\mathbf{W}$ from \eqref{eq:defW},  and that $\boldsymbol{\lambda}_{y,+}=\sum_{i\in{\N}}\boldsymbol{\lambda}_{y,yi},$ see \eqref{DEF-Conductances+}, our goal is now to prove that, under $\mathbf{Q}_{x}^{κ,L},$
		\begin{equation}
			\label{eq:propertyivissame}
			\begin{gathered}
				\text{ if }\big |\big\{y\in{\mathbf{W}\setminus\{x\}}:\,\boldsymbol{\boldsymbol{λ}}_{y,+}\le C_Λ\big \}\big|\geq (1-c_f)L\text{ and }\big|\big\{y\in{\mathbf{W}\setminus\{x\}}:\,\big|G_{y}^{\mathbf{T}_{V_L}}\big|\ge 2\big\}\big|	\ge 2L c_f, \\
				\text{then }\big |\big\{y\in{\partial\mathbf{T}_{V_L}\setminus\{x1,\mathbf{X}_{V_L}\}}:\,\boldsymbol{λ}_{y,y^-}\le C_Λ\big \}\big|\geq c_fL;
			\end{gathered}
		\end{equation}
		indeed, in view of \eqref{eq:lambdaonWareiid}, on the event $\tilde{V}_L(\mathbf{X})=\infty,$ which implies ${V}_L(\mathbf{X})<H_{x^-}(\mathbf{X}),$ we can take advantage of \eqref{eq:propertyivissame} in order to use \eqref{eq:firstpartproofiv} and  \eqref{eq:secondpartproofiv} to upper bound the probability of the event appearing in \ref{item:lemma iv} of Lemma~\ref{LEMMA-Goodwatershed}, and we can conclude.

		To prove \eqref{eq:propertyivissame}, let us define $A:=\{y\in{\mathbf{W}}\setminus \{x\}:\,|G_y^{\mathbf{T}_{V_L}}|\geq2\}$ the set of vertices in ${\mathbf{W}\setminus \{x\}}$ with at least two children in $\mathbf{T}_{V_L}.$ 
        Observe that 
       $|∂\mathbf{T}_{V_L}\setminus G_x^{\mathbf{T}_{V_L}}| \ge |A|+1$, which can easily be proved recursively on $|\mathbf{W}|$ starting at $|\mathbf{W}|=2$.
        In addition, for each $y\in{\partial\mathbf{T}_{V_L}\setminus G_x^{\mathbf{T}_{V_L}}}$ we have $y^-\in{\mathbf{W}\setminus\{x\}}$ and $\boldsymbol{\lambda}_{y,y^-}\leq \boldsymbol{\lambda}_{y^-,+},$ and so $\boldsymbol{\lambda}_{y,y^-}\geq C_{Λ}$ for at most $c_fL$ different $y\in{\partial\mathbf{T}_{V_L}\setminus G_x^{\mathbf{T}_{V_L}}}$ on the first event of the first line of \eqref{eq:propertyivissame}. Therefore, since the second event in the first line of \eqref{eq:propertyivissame} implies $|A|\geq 2Lc_f,$ we have at least $c_fL+1$ many vertices $y\in{\partial\mathbf{T}_{V_L}\setminus G_x^{\mathbf{T}_{V_L}}}$ with $\boldsymbol{\lambda}_{y,y^-}\leq C_{\Lambda},$ which finishes the proof of \eqref{eq:propertyivissame}.

		\item Here we can use the Marcinkiewicz-Zygmund law of large numbers, which states that, if  $ (Y_k)_{k\in\N} $ is a sequence of i.i.d.\ random variables with $\mathbb{E}[\ab{Y_1}^r] <∞$ for some $ 0<r<1 $, then
		\begin{equation*}
			\frac{1}{n^{1/r}}\sum_{k=1}^{n} Y_k \xrightarrow[n\to∞]{a.s.}0.
		\end{equation*}
		A proof of this classical result can be found in \cite[Section~17.4, p.254]{Loeve1977}.
		We can take $ Y_k:= (\sum_{i}\boldsymbol{λ}^{(k)}_i)^{\frac32}$ and $ r=\frac 23 $ since the expectation of $Y_k^{\frac 23}$ under $\mathbf{Q}_x^{κ,L}$ is then equal to  $\mathbb{E}^ν[\sum_{i}\lambda_i],$ which is finite by our assumption \eqref{DEF-Conductances+} (see also \eqref{eq:assfinitemoment}). By \eqref{eq:lambdaonWareiid}, this then entails that $L^{-3/2}\sum_{y\in \mathbf{W}\setminus\{x\}}Y_k $ converges $ \as $ to 0 as $ L\to ∞ $, and hence for all $ε\in{(0,1)}$ and $B>0$ there exists $L_0=L_0(B,ε)$ so that for all $L\geq L_0,$
		\begin{equation}
			\label{eq:proofofv)}
			\mathbf{Q}_{x}^{κ,L}
			\bigg(
			\frac{1}{L^{\frac32}}  \sum_{k=1}^{L-1} 	\big(\sum_{i\in{\N}}\boldsymbol{\lambda}_i^{(k)}\big)^{\frac32}    \geq
			\frac{B}{6}
			\bigg) \leq  ε.
		\end{equation}
		Our goal is now to prove that for $L\geq L_0(B,ε),$

		\begin{equation}
			\label{eq:propertyvissame}
			\text{ if }\frac{1}{L^{\frac32}}  \sum_{y\in \mathbf{W}} 	(\boldsymbol{λ}_{y,+})^{\frac32}    <
			\frac{B}{6},
			\text{ then }\frac{1}{L^{\frac32}}  \sum_{y\in \mathbf{W}} 	(\boldsymbol{λ}_y)^{\frac32}    <
			B;
		\end{equation}
		indeed, in view of  \eqref{eq:lambdaonWareiid}, on the event $\tilde{V}_L(\mathbf{X})=\infty,$ we can use \eqref{eq:propertyvissame} and then \eqref{eq:proofofv)} to upper bound the probability of the event appearing in \ref{item:lemma v} of Lemma~\ref{lem:probatobegood}, so that we can conclude. To prove \eqref{eq:propertyvissame}, we use the bounds $(\boldsymbol{λ}_y)^{\frac32}\leq \sqrt{8}((\boldsymbol{λ}_{y,+})^{\frac32}+(\boldsymbol{λ}_{y,y^-})^{\frac32})$ for all $y\in{\mathbf{W}},$ the bound $\boldsymbol{λ}_{y,y^-}\leq \boldsymbol{λ}_{y^-,+}$ for all $y\in{\mathbf{W}\setminus\{x\}},$ the inequality $\boldsymbol{\lambda}_{x,x^-}=κ\leq C_{\Lambda},$   the fact that $\{y^-:\,y\in{\mathbf{W}\setminus\{x\}}\}\subset\mathbf{W},$ and take $L_0(B,ε)$ much larger than $C_{\Lambda}/B^{2/3}.$

	\end{enumerate}
\end{proof}

Let us now show that the bounds obtained in Lemma~\ref{LEMMA-Goodwatershed} can be combined to lower bound the probability that a vertex $a\in{F}$ is good, see Definition~\ref{DEF-Goodwatershed}. Recall that $\mathbf{P}_{L,\widetilde{u}}^W$ is the probability measure underlying our tree of free points constructed in Section~\ref{Section-PatchingTogetherwatersheds}, see also below \eqref{DEF-Poissonian}.

\begin{lemma}
	\label{lem:probatobegood}
	Let $c_{\lambda},$ $C_{\Lambda},$ $C_g$ and $c_f$ be as in Lemma~\ref{LEMMA-Goodwatershed}. There exists $c_p>0$ such that for all $B>0,$ there exists $L_0(B)\in{\N}$ such that for all $a\in{\mathcal{X}},$ $L\geq L_0(B)$ and $\widetilde{u}>0,$ on the event $\{\lambda_{a,a^-}^F\leq C_{\Lambda}\}$ we have
	\begin{equation*}
		\mathbf{P}^W_{L,\widetilde{u}}\left(a \text{ is }(L, B, c_λ,C_Λ,C_g, c_f) \text{-good}\,\big|\,\lambda_{a,a^-}^F,a\in{F}\right)\geq c_p(1-e^{-\widetilde{u}}).
	\end{equation*}
\end{lemma}

\begin{proof}

	We will check the properties of Definition~\ref{DEF-Goodwatershed}.
	In the first part of the proof, we show that the event appearing in Lemma~\ref{LEMMA-Goodwatershed} \ref{item:lemma iii} implies that Definition~\ref{DEF-Goodwatershed} \ref{item:iii} is fulfilled with positive conditional probabilities under the appropriate conditions. More precisely, we have for all $a\in{F}$ that
	\begin{equation}
		\label{eq:propertyiiiissame}
		\begin{gathered}
			\text{if }\boldsymbol{\lambda}^a_{\widehat{a},+}\leq C_{\Lambda},\lambda_{a,a^-}^F\leq C_{\Lambda}\text{ and }\boldsymbol{\lambda}_{\widehat{a},\widehat{a}2}^a> c_{λ}\\
			\text{then } P_{\widehat{a}}^{\mathcal{T}^{\mathbf{W}}}\big(H_{\set{\widehat{a}^-, \widehat{a}1 }}=\tilde{V}_L=\infty\big)
			\geq\frac{c_{\lambda}\boldsymbol{\lambda}^a_{\widehat{a}2,\widehat{a}21}}{2C_{\Lambda}(2C_{\Lambda}+\boldsymbol{\lambda}^{a}_{\widehat{a}2,+})}
			P_{\widehat{a}21}^{\mathcal{T}^{\mathbf{W}}}(\tilde{V}_{L-2}= {H}_{\widehat{a}2}=\infty);
		\end{gathered}
	\end{equation}
	indeed, under the conditions from \eqref{eq:propertyiiiissame}, noting that $\boldsymbol{\lambda}_{\widehat{a},\widehat{a}^-}^a=\lambda_{a,a^-}^F$ by \eqref{eq:deflambdaF}, and thus $\boldsymbol{\lambda}_{\widehat{a}}^a\leq 2C_{\Lambda},$ we have that
	\begin{equation*}
		P_{\widehat{a}}^{\mathcal{T}^{\mathbf{W}}}(X_2=\widehat{a}21)= \frac{\boldsymbol{λ}^a_{\widehat{a}, \widehat{a}2}\boldsymbol{\lambda}^a_{\widehat{a}2,\widehat{a}21}}{\boldsymbol{λ}^a_{\widehat{a}}(\boldsymbol{\lambda}^a_{\widehat{a},\widehat{a}2}+\boldsymbol{\lambda}^a_{\widehat{a}2,+})}\ge \frac{c_{\lambda}\boldsymbol{\lambda}^a_{\widehat{a}2,\widehat{a}21}}{2C_{\Lambda}(2C_{\Lambda}+\boldsymbol{\lambda}^a_{\widehat{a}2,+})}.
	\end{equation*}
	Therefore, \eqref{eq:propertyiiiissame} follows easily by using the Markov property at time $2,$ noting that, under $P_{\widehat{a}}^{\mathcal{T}^{\mathbf{W}}}$ and on  the event $\{X_{2}=\widehat{a}21\},$  in view of \eqref{DEF-StoppingTimeVn} and \eqref{DEF-StoppingTimeHL}, we have $\tilde{V}_{L-2}((X_{k+2})_{k\geq0})=\tilde{V}_L((X_k)_{k\geq0}).$ Furthermore, if $\widehat{a}2$ is never visited after time $2,$ then $\widehat{a}1$ and $\widehat{a}^-$ are never visited by $X.$ Moreover, note that the random variable on the right-hand side of the inequality of the second line of \eqref{eq:propertyiiiissame} is independent of $\mathcal{T}^{\widehat{a}1},$  $\Gamma_{\widehat{a}},$ $(\boldsymbol{\lambda}_{\widehat{a},\widehat{a}i}^a)_{i\in{\N}}$ and $\lambda_{a,a^-}^F.$  Combining Proposition~\ref{PROP-SinglewatershedLikeGW}, \eqref{eq:defwatershedata}, Lemma~\ref{LEMMA-Goodwatershed} \ref{item:lemma iii} and \eqref{eq:propertyiiiissame}, we thus have on the intersection of the events $\{\boldsymbol{\lambda}_{\widehat{a},\widehat{a}2}^a> c_{\Lambda}\},$ $\{\boldsymbol{\lambda}^a_{\widehat{a},+}\leq C_{\Lambda}\}$  and $\{\lambda_{a,a^-}^F\leq C_{\Lambda}\},$ that
	\begin{equation}
		\label{eq:boundprobagood2}
		\mathbf{P}_{L,\widetilde{u}}^{\mathbf{W}}\left(H_{\set{\widehat{a}^-, \widehat{a}1 }}(\mathbf{X}^{{a}})=\tilde{V}_L(\mathbf{X}^{{a}})=\infty\,
		\Big|\,
		Γ_{\widehat{a}},\,
		(\boldsymbol{\lambda}_{\widehat{a},\widehat{a}i}^a)_{i\in{\N}},\,
		\mathcal{T}^{\widehat{a}1},\,
		\lambda_{a,a^-}^F,\,
		a\in{F}\right)
		\geq c_V.
	\end{equation}

	In this second part of the proof, we aim at combining the estimates from Lemma~\ref{LEMMA-Goodwatershed} in order to infer the general lower bound $c_p(1-e^{-\widetilde{u}})$ on the probability for $a$ to be good. Obtaining a lower bound on the intersection of the events \ref{item:i}, \ref{item:ii} and \ref{item:iii} in Definition~\ref{DEF-Goodwatershed} is easy by independence, Lemma~\ref{LEMMA-Goodwatershed} and \eqref{eq:boundprobagood2}.  More care is required for the other properties though.

	It is not difficult to combine Lemma~\ref{LEMMA-Goodwatershed} \ref{item:lemma iv} and \ref{item:lemma v}, since the complements of the events there happen with high probability, as we now explain.    On the event $\{\lambda_{a,a^-}^F\leq C_{\Lambda}\},$ using the estimates from Lemma~\ref{LEMMA-Goodwatershed} \ref{item:lemma iv}, \ref{item:lemma v} for $ε=\frac{1}{3}\frac{c_V(1-μ(1))}{2},$ and writing them in the form of Definition~\ref{DEF-Goodwatershed} -- see \eqref{eq:defwatershedata}, \eqref{DEF-FreePoints}, \eqref{eq:deflambdaF} and the definition of the tree of free points from \eqref{eq:deffreepointrecursively} and below -- we thus have for all $L\geq L_0(B),$ with $L_0(B)=L_0(B,ε)$ from Lemma~\ref{LEMMA-Goodwatershed} for this choice of $ε$ that
	\begin{equation}
		\label{eq:boundprobagood1}
		\begin{split}
			\mathbf{P}_{L,\widetilde{u}}^{\mathbf{W}}&\left(
			\begin{gathered}
				\Big\{\big|\big\{a'\in{G_{a}^F}:\,λ_{a,a'}^F\le C_Λ\big \}\big|\geq c_fL,\, {L^{-\frac32}}  \sum_{y\in \mathbf{W}^{a}} 	(\boldsymbol{λ}^a_y)^{\frac32}    < B\Big\}^c,
				\\H_{\{\widehat{a}^-,\widehat{a}1\}}(\mathbf{X}^a)=\tilde{V}_L(\mathbf{X}^a)=\infty\,
			\end{gathered}
			\middle|             \,
			Γ_{\widehat{a}}, \,
			\mathcal{T}^{\widehat{a}1},\,
			\lambda_{a,a^-}^F,\,
			a\in{F} \!                \right)
			\\&\leq  \frac{2}{3}\frac{c_V(1-μ(1))}{2}.
		\end{split}
	\end{equation}
	Here, we used that both, the event $H_{\{\widehat{a}^-,\widehat{a}1\}}(\mathbf{X}^a)=\tilde{V}_L(\mathbf{X}^a)=\infty$ and the events in Definition~\ref{DEF-Goodwatershed} \ref{item:iv} and \ref{item:v}, are $(\mathbf{T}^a,\boldsymbol{\lambda}^a,\mathbf{X}^a)$-measurable, and thus independent of $Γ_{\widehat{a}}$ and $\mathcal{T}^{\widehat{a}1},$ and that $\{\widehat{a}:\,a\in{G_a^F}\}=\partial\mathbf{T}_{V_L(\mathbf{X}^a)}^{{a}}\setminus\{\widehat{a}1,\mathbf{X}^a_{V_L(\mathbf{X}^a)}\}$ when $H_{\{\widehat{a}^-,\widehat{a}1\}}(\mathbf{X}^a)=\infty$ in view of \eqref{DEF-FreePoints}, \eqref{eq:deffreepointrecursively}.

	Now we can further combine \eqref{eq:boundprobagood2} with the equation in the first line of \ref{item:lemma ii} of Lemma~\ref{LEMMA-Goodwatershed} (recall that the number of children $\big| G_{\widehat{a}}^{\mathbf{T}^{a}_{1}} \big|$ of $\widehat{a}$ in $\mathbf{T}_1^a$ is equal to $\pi((\boldsymbol{\lambda}^a_{\widehat{a},\widehat{a}i})_{i\in{\N}})$). One can combine this with \eqref{eq:boundprobagood1} thanks to the dependence of the bound \eqref{eq:boundprobagood1} on $c_V(1-μ(1))/2,$ noting also that the event in the first line of Definition~\ref{DEF-Goodwatershed} \ref{item:ii} is independent of $Γ_{\widehat{a}}$ and $\mathcal{T}^{\widehat{a}1}$, to obtain that on the event $\{\lambda_{a,a^-}^F\leq C_{\Lambda}\},$ for all $L\geq L_0(B)$ we have
	\begin{equation}
		\label{eq:boundprobagood4}
		\begin{split}
			\mathbf{P}_{L,\widetilde{u}}^{\mathbf{W}}
			&\left(
			\begin{gathered}
				\big |\big\{a'\in{G_{a}^F}:\,λ_{a,a'}^F\le C_Λ\big \}\big|\ge c_fL,\,
				{L^{-\frac32}}  \sum_{y\in \mathbf{W}^{a}} 	(\boldsymbol{λ}^a_y)^{\frac32}<B,\,\\
				H_{\set{\widehat{a}^-,\widehat{a}1}}(\mathbf{X}^{{a}})=     \tilde{V}_L(\mathbf{X}^{{a}})=\infty,\,
				\big| G_{\widehat{a}}^{\mathbf{T}^{a}_{1}} \big| \ge 2,\,
				\boldsymbol{λ}^a_{\widehat{a},\widehat{a}1}> c_λ,\,
				\boldsymbol{\lambda}^a_{\widehat{a},+}\leq C_{\Lambda}
			\end{gathered}\,
			\middle| \,
			Γ_{\widehat{a}},\,
			\mathcal{T}^{\widehat{a}1},\,
			\lambda_{a,a^-}^F,\,
			a\in{F}
			\right)         \\
			&   \ge \frac{1}{3}\frac{c_V(1-μ(1))}{2}.
		\end{split}
	\end{equation}

	Finally, for the good events in \ref{item:i} and the second line of \ref{item:ii} in Definition~\ref{DEF-Goodwatershed}, conditionally on $a\in{F}$ and $\lambda_{a,a^-}^F,$ the random variables
	$\Gamma_{\widehat{a}}$ and
	$\mathcal{T}^{\widehat{a}1}$
	have respective laws  $\mathbb{P}_{\widetilde{u}}^{\Gamma}(\Gamma_{\widehat{a}}\in{\cdot})$ and $\mathbb{P}^{{\rm GW}}(\widehat{a}1\cdot\mathcal{T}\in{\cdot}),$ (see, respectively, below \eqref{DEF-Poissonian} and \eqref{eq:defrestofthetrees}), and are independent. Therefore, the two estimates provided by Lemma~\ref{LEMMA-Goodwatershed} \ref{item:lemma i}
	and the second line of \ref{item:lemma ii}, yield that for all $\widetilde{u}>0$ one has
	\begin{equation}
		\label{eq:boundprobagood3}
		\mathbf{P}_{L,\widetilde{u}}^{\mathbf{W}}\left(Γ_{\widehat{a}} \ge 1, g^{\mathcal{T}^{\widehat{a}1}}(\widehat{a}1,\widehat{a}1) \le C_g\,\Big|\,\lambda_{a,a^-}^F,a\in{F}\right)
		\geq\frac12(1-\exp(-\widetilde{u})).
	\end{equation}
	Combining \eqref{eq:boundprobagood4} and \eqref{eq:boundprobagood3}, we can readily conclude by taking $c_p=c_V(1-\mu(1))/12.$
\end{proof}

We now want to show that the set of good free points introduced in Definition~\ref{DEF-Goodwatershed} percolates with the help of Lemma~\ref{lem:probatobegood}. This set can be interpreted as a random subset in $\mathcal{X},$ endowed with the $\sigma$-algebra introduced at the end of Section~\ref{SECTION-GWTree}. Recall the definition $G_x^A$ of the number of children of $x$ in $A\subset\mathcal{X}$ from \eqref{DEF-GenerationOfTree}. In the following technical lemma, we say that a tree is $d$-ary if it contains $\emptyset$ and every vertex has exactly $d$ children.
While it seems like a standard result, we were not able to locate it in the literature and therefore provide a proof here.

\begin{lemma}\label{LEMMA-DependBernPercolation}
	There exists a function $d:[0,\infty)\rightarrow \N_0$ such that $d(t)\rightarrow\infty$ as $t\rightarrow\infty$ and the following holds. Under some probability measure $\mathbb{P},$ let $S\subset\mathcal{X}$ be a random set containing $\emptyset$ almost surely, such that for some $N\in{\N}$ and $p\in{[0,1]},$ for all $ x\in{\mathcal{X}}$
	\begin{equation}
		\label{eq:stochadom}
		\mathbb{P}\Big(|G_x^{{S}}|\geq N
		\given \mathcal{F}_x\Big)		\geq p\text{ on the event }\{x\in{S}\};
	\end{equation}
	here, $\mathcal{F}_x=\sigma(\ind_{\set{y\in{S}}},y\in \mathcal X \setminus ({x\cdot(\mathcal{X}\setminus\{\emptyset\}})))$ is the $\sigma$-algebra generated by the restriction of ${S}$ to vertices which are not descendants of  $x.$ Then, ${S}$ contains with positive probability, depending only on $p$ and $N,$ a $d(Np)$-ary tree.
\end{lemma}
\begin{proof}
	In this proof, we say that a random subset of $\mathcal{X}$ is a weightless Galton--Watson tree with offspring distribution $pδ_N+(1-p)δ_0$ if, after possible reordering of the labels, this set has the same law as the tree $\mathcal{T}$ seen as a subset of $\mathcal{X}$ (that is removing the weights), introduced in Section~\ref{SECTION-GWTree} when the offspring distribution $μ$ from \eqref{DEF-OffspringDistribution} is $pδ_N+(1-p)δ_0.$ Note that since we discard the weights here, the law of this tree is entirely determined by its offspring distribution.

	Let us first show that we can couple $S$ and a weightless Galton--Watson tree with offspring distribution $p\delta_N+(1-p)\delta_0,$ such that $S$ is included in this tree.   For this purpose, fix a sequence $x_0,x_1,\dots$ exhausting $\mathcal{X}$ and such that $\{x_0,\dots,x_{k-1}\}\subset{(x_k\cdot\mathcal{X})^c}$ for each $k\in{\N_0}.$ The result will follow once we have that, under some probability measure $\tilde{\mathbb{P}},$ there exist an
	i.i.d.\ family of Bernoulli random variables $\zeta_{x_k},$ $k \in \N_0$ with parameter $p,$ and random
	sets $\tilde{S}_k,$ $k\in{\N_0},$ with the following properties: $\tilde{S}_k$ is an increasing sequence of sets, each with the same law as $S_k:=\{x\in{S}:\,x\sim x_i\text{ for some }i\leq k\}$ under $\mathbb{P},$ and if $\zeta_{x_k}=1$ and $x_k\in{\tilde{S}_k},$ then $|G_{x_k}^{\tilde{S}_k}|\geq N$ (in order to facilitate reading, the construction of these random variables will take place in the last paragraph of the proof).
	Indeed, defining $\tilde{S}$ as the union of $\tilde{S}_k,$ $k\in{\N_0},$ one  obtains that  $\tilde{S}$ has the same law as $S$ under $\mathbb{P}.$ Furthermore, the tree $T$ obtained recursively by keeping exactly $N$ children in $\tilde{S}$ of $x\in{\tilde{S}}$ each time $\zeta_{x}=1,$ and keeping zero children otherwise, is then a Galton--Watson tree with offspring distribution $p\delta_N+(1-p)\delta_0,$ which is contained in $\tilde{S}.$

	In order to conclude, we still need to show that for each $\tilde{d}\in{\N_0},$ there exists $t=t(\tilde{d})\in{(0,\infty)}$ such that for each $p\in{[0,1]}$ and $N\in{\N}$ with $pN\geq t,$ a weightless Galton--Watson tree with offspring distribution $p\delta_N+(1-p)\delta_0$ contains with positive probability a $\tilde{d}$-ary tree, and then take $d(s):=\sup\{\tilde{d}\in{\N_0}:\,t(\tilde{d})\leq s\}$ for all $s>0,$ with the convention $\sup\emptyset=0.$ This can be easily proven by noting that, if $G_{\tilde{d}}$ is the function from \cite[Theorem~5.29]{LyonsPeres17}, then $G_{\tilde{d}}(0)>0$ and $G_{\tilde{d}}(1-p/2)<1-p/2$ if $pN\geq t$ for some $t$ large enough. We leave the details to the reader.

	It therefore remains to construct construct the sequences $\tilde{S}_k$ and $\zeta_{x_k},$ $k\in{\N_0}.$
	We have $x_{0}=\emptyset,$ and \eqref{eq:stochadom} applied to $x=\emptyset$ implies that one can indeed define a Bernoulli random variable $\zeta_{\emptyset}$ with parameter $p$ and $\tilde{S}_0$ such that $\tilde{S}_0$ has the same law as $\{x\in{S}:\,x\sim \emptyset\},$ and $\zeta_{\emptyset}=1$ implies $|G_{\emptyset}^{\tilde{S}_0}|\geq N.$  Assume now that $\zeta_{x_i},$ $i\leq k-1,$ and $\tilde{S}_{k-1}$ are constructed. Let $\tilde{S}_k$ be the union of $\tilde{S}_{k-1}$ and some children of $x_k,$ constructed so that, conditionally on $(\zeta_{x_i})_{i\leq k-1}$ and $\tilde{S}_{k-1},$ the law of $\tilde{S}_k$ is the same as law of $S_{k}$ conditionally on $S_{k-1}=\tilde{S}_{k-1}.$   Then
	\eqref{eq:stochadom} implies that, conditionally on $(\zeta_{x_i})_{i\leq k-1}$ and $\tilde{S}_{k-1},$ $\ind\big\{|G_{x_k}^{\tilde{S}_k}|\geq N\big\}$ stochastically dominates a Bernoulli random variable with parameter $p$ on the event $\{x_k\in{\tilde{S}_{k-1}}\}.$ Hence, up to extending the probability space $\tilde{\mathbb{P}},$ we can define a Bernoulli random variable ${\zeta}_{x_k}$ with parameter $p,$ independent of ${\zeta}_{x_i},$ $i\leq k-1,$ and $\tilde{S}_{k-1},$ and such that if $\zeta_{x_k}=1$ and $x_k\in{\tilde{S}_{k-1}}$ then $|G_{x_k}^{\tilde{S}_k}|\geq N.$ This concludes the induction, and the proof that $\tilde{S}$ contains a.s.\ a weightless Galton--Watson tree with offspring distribution $p\delta_N+(1-p)\delta_0.$
\end{proof}

We now prove that with positive probability, the tree of $(L, B, c_λ,C_Λ,C_g, c_f)$-good free points contains a $d$-ary tree for suitable choices of the parameters. To do so, observe that on the one hand, the probability for a free point to be good is bounded from below due to  Lemma~\ref{lem:probatobegood}. On the other hand, property \ref{item:iv} of Definition~\ref{DEF-Goodwatershed} will let us tune the parameter $L$ in such a way that a good free point has many children. We will then be able to use Lemma~\ref{LEMMA-DependBernPercolation} in order to conclude.

\begin{prop}\label{PROP-Prop1}
	Let $c_{\lambda},$ $C_{\Lambda},$ $C_g$ and $c_f$ be as in Lemma~\ref{LEMMA-Goodwatershed}, $c_p$ as in Lemma~\ref{lem:probatobegood}, and the function $d$ as in Lemma~\ref{LEMMA-DependBernPercolation}. For all $B>0,$ there exists $L_0(B)\in{\N}$ such that for all $L\geq L_0(B)$ and $\widetilde{u}>0,$ the set
	\begin{equation}
		\label{eq:defFg}
		F^g:=\{\emptyset\}\cup\big\{a\in F\setminus\{\emptyset\}
		\given a^- \text{ is }(L,B, c_λ,C_Λ,C_g, c_f) \text{-good and }\lambda_{a,a^-}^F\leq C_{\Lambda}
		\big \}
	\end{equation}
	contains with positive  $\mathbf{P}_{L,\widetilde{u}}^{\mathbf{W}}$ probability a $d(Lq(\widetilde{u}))$-ary tree, where $q(\widetilde{u})=c_fc_p(1-e^{-\widetilde{u}}).$
\end{prop}

\begin{proof}\
	Let $B>0.$ Fix $c_{\lambda},$ $C_{\Lambda},$ $C_g,$ $c_f,$ and $L_0(B)$  as in Lemma~\ref{lem:probatobegood}, and fix $L\geq L_0(B)$ and $\widetilde{u}>0.$ Throughout the proof we write ``good'' instead of ``$(L, B, c_λ,C_Λ,C_g, c_f)$-good'' to simplify notation, keeping the implicit dependence on the parameters in mind. Let us first extend the definition of the weights $\lambda^F$ from $\{\{a,a^-\}\, : \, a\in{F\setminus\{\emptyset\}}\}$ to $\{\{a,a^-\}\, : \, a\in\mathcal{X}\setminus\{\emptyset\}\}$  by letting $\lambda_{a,a^-}^F=0$ if $a\in{\mathcal{X}\setminus F}.$ For each $a\in{\mathcal{X}\setminus F},$ we also fix arbitrarily some $\widehat{a}\in{\mathcal{X}},$  so that $\widehat{a}\neq\widehat{a'}$ for all $a\neq a'\in{\mathcal{X}}.$  This way, we can also define $(\mathbf{T}^a,\boldsymbol{\lambda}^a,\mathbf{X}^a),$ $a \in \mathcal X\setminus F,$ as a family of independent watersheds with law $\mathbf{Q}^{λ_{a^-,a}^F,L}_{\widehat{a}},$ see \eqref{eq:defwatershedata}.  Note that for $a\notin{F}$ we never actually use the additional watershed $(\mathbf{T}^a,\boldsymbol{\lambda}^a,\mathbf{X}^a)$ nor the notation $\widehat{a},$  they are however necessary to define the following $\sigma$-algebra
	\begin{equation*}
		\mathcal{W}^{a} := σ\big(Γ_{\widehat{a}},\mathbf{X}^a,
		(\boldsymbol{λ}_{x,y}^a)_{x\sim y\in \mathbf{T}^{{a}}}	,
		(λ_{x,y}^{\widehat{a}1})_{x\sim y\in{ \mathcal{T}^{\widehat{a}1}}}\big)\text{ for all }a\in{\mathcal{X}},
	\end{equation*}
	where $\lambda^{\widehat{a}1}$ are the weights of the tree $\mathcal{T}^{\widehat{a}1}$ which was defined in \eqref{eq:defrestofthetrees}; also recall that $\mathbf{X}^a,$ $\boldsymbol{\lambda}^a$ and $\lambda^{\widehat{a}1}$ are random variables whose canonical $\sigma$-algebras on their respective state spaces have been defined at the end of Section~\ref{SECTION-GWTree}.
	By construction, $(\mathbf{T}^{a^-},\boldsymbol{λ}^{a^-},\mathbf{X}^{a^-}),$  $\mathcal{T}_{(\widehat{a^-})1}^{\mathbf{W}},$ the weight $\lambda_{a^-,a}^F=\boldsymbol{\lambda}_{\widehat{a},\widehat{a}^-}^{a^-},$ see  \eqref{eq:deflambdaF},  as well as the event $\{a\in{F}\}=\{\lambda_{a^-,a}^F>0\}$ are $\mathcal{W}^{a^-}$-measurable. Therefore, in view of Definition~\ref{DEF-Goodwatershed}
	\begin{equation}
		\label{eq:FgisWames}
		\{a \in{F^g}\}\in{\mathcal{W}^{a^-}}\text{ for all }a\in{\mathcal{X}},
	\end{equation}
	where we recall $F^g$ from \eqref{eq:defFg}, and
	with the convention $\mathcal{W}^{\emptyset^-}:=\sigma(\{\emptyset\})$ is the trivial $\sigma$-algebra.   By \eqref{eq:defwatershedata}, a watershed depends on the previous watersheds only through the weights $\lambda_{a,a^-}^F,$ that is $\mathcal{W}^a$ and $\mathcal{W}^{a'},$ $a'\notin{a\cdot \mathcal{X}},$ are independent conditionally on $\lambda_{a,a^-}^F$ for all $a\in{F\setminus\{\emptyset\}}.$
	Therefore, defining for each $a\in{\mathcal{X}}$ the $\sigma$-algebra
	\begin{equation}
		\label{eq:defFga}
		\mathcal{F}^g_a:=\sigma\big(\mathcal{W}^{(a')^-},a'\notin{a\cdot(\mathcal{X}\setminus\{\emptyset\})}\big)=\sigma(\mathcal{W}^{a'},a'\notin{a\cdot\mathcal{X}}),
	\end{equation}
	we have that for all $a\in{F},$
	\begin{equation}
		\label{eq:replaceFabylambdaa}
		\mathbf{P}_{L,\widetilde{u}}^{\mathbf{W}}(a\text{ is good}\,|\,\mathcal{F}^g_{a})=\mathbf{P}_{L,\widetilde{u}}^{\mathbf{W}}(\text{$a$ is good}\,|\,\lambda_{a,a^-}^F,a\in{F}),
	\end{equation}
	with the convention $\lambda_{\emptyset,\emptyset^-}^F=0.$
	Note that, in view of \eqref{eq:FgisWames}, the $\sigma$-algebra $\mathcal{F}^g_a$ contains the $\sigma$-algebra $\mathcal{F}_a$ from Lemma~\ref{LEMMA-DependBernPercolation} when $S=F^g.$
	By property \ref{item:iv} of Definition~\ref{DEF-Goodwatershed}, we moreover have $|G_a^{F^g}|=|\{a'\in{G_a^F}:\,\lambda_{a,a'}^F\leq C_{\Lambda}\}|\geq c_fL$ if $a\in{F}$ is good. Thus since $\{\lambda_{a,a^-}^F\leq C_{\Lambda}\}\subset\{a\in{F^g}\}\in{\mathcal{F}^g_{a}}$ by  \eqref{eq:FgisWames} and \eqref{eq:defFga}, we have that on the event $\{a\in{F^g}\},$
	\begin{equation*}
		\mathbf{P}_{L,\widetilde{u}}^{\mathbf{W}}(|G_a^{F^g}|\geq c_fL\,|\,\mathcal{F}^g_{a})\geq \mathbf{P}_{L,\widetilde{u}}^{\mathbf{W}}(a\text{ is good}\,|\,\mathcal{F}^g_{a}) \geq c_p(1-e^{-\widetilde{u}}),
	\end{equation*}
	where we used  Lemma~\ref{lem:probatobegood} and \eqref{eq:replaceFabylambdaa} in the last inequality. Using \eqref{eq:defFga} and Lemma~\ref{LEMMA-DependBernPercolation} for $S=F^g,$ we can conclude.

\end{proof}

With the help of Proposition~\ref{PROP-W_Gw_WInterl}, we now show that for a suitable choice of the parameters $u,\widetilde{u}>0,$ under $\mathbf{P}_{L,\widetilde{u}}^W,$ for each $(L, B, c_λ,C_Λ,C_g, c_f)$-good free point $a\in{F},$ one can include the watershed $\mathbf{W}^a$ in the random interlacements set $\mathcal{I}^u$ from Proposition~\ref{PROP-W_Gw_WInterl}.   For this purpose, we need to verify that all the assumptions of \eqref{EQ-ConditionsWatershedsRandInt} are verified for good free points.

\begin{prop}\label{PROP-GoodwatershedsInInterlacements}
	Let $ u,B,c_λ,c_Λ,C_g, c_f>0,$  $L\in{\N},$ $ a\in F $
	and
	\begin{equation}\label{PROP1-DEF-Uprime}
		\widetilde{u}= u c_e,\text{ where }c_e:=\frac{c_λ}{c_λC_g+1}.
	\end{equation}
	Then, under the extension of the probability space $\mathbf{P}_{L,\widetilde{u}}^W$ from Proposition~\ref{PROP-W_Gw_WInterl},
	\begin{equation}
		\label{eq:WainclusIu}
		\mathbf{W}^{a} \subset \mathcal{I}^u\text{ for all }(L, B, c_λ,c_Λ,C_g, c_f)\text{-good vertices }a\in{F}.
	\end{equation}
\end{prop}
\begin{proof}
	Fix some $(L, B, c_λ,c_Λ,C_g, c_f)$-good vertex $a\in{F}.$
	First note that by properties \ref{item:i} and \ref{item:iii} of Definition~\ref{DEF-Goodwatershed}, the first and second condition in \eqref{EQ-ConditionsWatershedsRandInt} are satisfied, and thus by Proposition~\ref{PROP-W_Gw_WInterl},
	\begin{equation}
		\label{eq:WainclusIuif}
		\mathbf{W}^a\subset\mathcal{I}^u \quad \text{ once we show } \quad {u}\geq\frac{\widetilde{u}}{{e}_{\{\widehat{a}\},\mathcal{T}^{\mathbf{W}}_{\widehat{a}}}({\widehat{a}})}.
	\end{equation}
	To bound the parameter ${e}_{\{\widehat{a}\},\mathcal{T}^{\mathbf{W}}_{\widehat{a}}}({\widehat{a}})$ from below we will use property \ref{item:ii} of Definition~\ref{DEF-Goodwatershed}. We  use the  analogy to electrical circuits, and note that by Rayleigh’s Monotonicity Principle \cite[(2.5) and Sections 2.3 and 2.4]{LyonsPeres17}, we have that $ g^{\mathcal{T}_{\widehat{a}}^{\mathbf{W}}}(\widehat{a},\widehat{a})\leq g^{\mathcal{T}_{\widehat{a},1}^{^{\mathbf{W}}}}(\widehat{a},\widehat{a}),$ where $\mathcal{T}_{\widehat{a},1}^{^{\mathbf{W}}}$ denotes the subtree of $\mathcal{T}_{\widehat{a}}^{\mathbf{W}}$ consisting only of $\widehat{a}$ and $\mathcal{T}_{\widehat{a}1}^{\mathbf{W}}$. Moreover, using a series transformation \cite[Subsection 2.3.I]{LyonsPeres17}, equations \eqref{LEMMA Goodwatershed-EQ-ConductanceBounded} and \eqref{Goodwatershed-EQ-GreenFunctionBounded} imply that $ g^{\mathcal{T}_{\widehat{a},1}^{\mathbf{W}}}(\widehat{a},\widehat{a}) \le C_g+\frac{1}{c_λ}$ since, on the event $H_{\widehat{a}1}(\mathbf{X}^{\widehat{a}})=\infty$ which is implied by property \ref{item:iii} of Definition~\ref{DEF-Goodwatershed}, $\mathcal{T}^{\widehat{a}1}$ is the subtree $\mathcal{T}_{\widehat{a}1}^{\mathbf{W}}$ of $\mathcal{T}^{\mathbf{W}}$ below $\widehat{a}1$ as explained in the second paragraph below Definition~\ref{DEF-Goodwatershed}.
	Thus, the equilibrium measure at $ \widehat{a}  $ for $\mathcal{T}_{\widehat{a}}^{\mathbf{W}}$ is bounded from below by
	\begin{equation}\label{PROP1-EQ-ConstantCE}
		{e}_{\{\widehat{a}\},\mathcal{T}^{\mathbf{W}}_{\widehat{a}}}({\widehat{a}})
		=\frac{1}{g^{\mathcal{T}_{\widehat{a}}^{\mathbf{W}}}(\widehat{a},\widehat{a})} \ge \frac{c_λ}{c_λC_g+1}=:c_e.
	\end{equation}
	We can conclude by combining \eqref{PROP1-DEF-Uprime}, \eqref{eq:WainclusIuif} and \eqref{PROP1-EQ-ConstantCE}.

\end{proof}

If $q(\widetilde{u})L$ is large enough, combining Propositions~\ref{PROP-Prop1} and \ref{PROP-GoodwatershedsInInterlacements} provides us with an infinite tree of good free points $a$ satisfying $\mathbf{W}^a\subset\mathcal{I}^u.$  Taking advantage of  property \ref{item:v} from Definition~\ref{DEF-Goodwatershed}, we are now ready to prove percolation for the set on the left-hand side of \eqref{eq:Iu+Auincluded}. For each $p\in{(0,1)},$ under some probability $\mathbb{P}^{\rm E}_p,$ let $(\mathcal{E}_x)_{x\in{\mathcal{X}}}$ be an independent family of exponential random variables with parameter one, and $(\mathcal{B}_x)_{x\in{\mathcal{X}}}$ the independent family of Bernoulli random variables defined above \eqref{eq:defBp}. Recall that $φ$ is a Gaussian free field on $T$ under $\Pgff{T}{},$ see Section~\ref{sec:GFF}, that $\mathcal{I}^u$ is a random interlacements set on $T$ under $\Pri_T,$ see Section~\ref{sec:RI}, that $\mathcal{T}$ is a Galton--Watson tree under $\mathbb{P}^{{\rm GW}},$ see Section~\ref{SECTION-GWTree}, and let $B_p$ be as in \eqref{eq:defBp} and $A_u$ as in \eqref{PROP1-DEF-Au}.

\begin{prop}\label{PROP-Percolation-Au}
	There exists $ u_0>0 $ such that for each $ u\in (0,u_0],$ there exists  $p\in(0,1)$ so that the set $A_u\cap B_p \cap \mathcal{I}^u$ contains $\mathbb{E}^{{\rm GW}}[\Pri_{\mathcal{T}}\otimes\Pgff{\mathcal{T}}{}\otimes\mathbb{P}_{p}^{\rm E}(\cdot)]$-a.s.\ an unbounded cluster.
\end{prop}
\begin{proof}

	Under $\mathbf{E}_{L,\widetilde{u}}^{\mathbf{W}}
	[\Pgff{\mathcal T^{\mathbf{W}}}{}
	\otimes  \mathbb{P}_{p}^{\rm E}
	(\cdot)],$ for some $L\in{\N}$ and $\widetilde{u}>0,$ consider the event
	\begin{equation}
		\label{eq:defAuW}
		A_{u}^{\mathbf{W}}
		:=\left\lbrace	x\in \mathcal{T}^{\mathbf{W}}\colon
		\mathcal{E}_x>4u\lambda_x^{\mathbf{W}}\text{ or }\ab{φ_x}>2\sqrt{2u}	\right\rbrace
		\cap{\left\lbrace	x\in \mathcal{T}^{\mathbf{W}}\colon
			\mathcal{B}_x=1	\right\rbrace}.
	\end{equation}
	For $a\in{F},$ we now evaluate  the probability, conditioned on the value of $ φ_{\widehat{a}^-},$ that $ \mathbf{W}^{a}\subset A_{u}^{\mathbf{W}}$ (recall \eqref{DEF-watershed}). For $ \mathcal{E} $ and $\mathcal{B},$ simple estimates for exponential and Bernoulli variables will be sufficient, while for the Gaussian free field we take advantage of the Markov property \eqref{LEMMA-MarkovProperty} applied to the set $U_a:=\mathcal{T}^{\mathbf{W}}_{{\widehat{a}} }. $  For each $y\in{U_a},$ one can decompose the field as $ φ_y= ψ^{U_a}_y + β^{U_a}_y ;$ here, $ ψ^{U_a}_y $ is a centered Gaussian field, independent of $ β_y^{U_a}$ and $φ_{\widehat{a}^-},$ and with variance $ g^{\mathcal{T}^{\mathbf{W}}}_{U_a}(y,y),$ which  by \eqref{DEF-GreenFunction} satisfies
	\begin{equation*}
		g^{\mathcal{T}^{\mathbf{W}}}_{U_a}(y,y)\ge \frac{1}{λ_y^{\mathbf{W}}}\text{ for all }y\in{U_a}.
	\end{equation*}
	Thus, for all $y\in{U_a}$ we have -- using the symmetry and unimodality of the distribution of $ψ^{U_a}_y$ to obtain the first inequality -- that
	\begin{equation}\label{PROP1-EQ-InequalityPhiU}
		\begin{split}
			\Pgff{\mathcal{T}^{\mathbf{W}}}{}\big(		\ab{φ_{y}} \leq  2\sqrt{2u} 		\given φ_{{\widehat{a}} ^-}  \big)
			&=		\Pgff{\mathcal{T}^{\mathbf{W}}}{}\big(		\ab{ψ^{U_a}_y+ β^{U_a}_y} \leq 2\sqrt{2u} 		\given φ_{{\widehat{a}} ^-}  \big)		\\
			&\le	\Pgff{\mathcal{T}^{\mathbf{W}}}{}\big(		\ab{ψ^{U_a}_y} \leq 2\sqrt{2u} 		 \big)			\le
			\frac{4\sqrt{2u}}{\sqrt{2π /λ_y^{\mathbf{W}}}}.
		\end{split}
	\end{equation}
	Therefore, for all $a\in{F},$
	\begin{equation}
		\label{eq:boundprobaAu}
		\begin{split}
			&\quad \Pgff{\mathcal T^{\mathbf{W}}}{}\otimes\mathbb{P}_{p}^{\rm E}
			\big(
			\mathbf{W}^{a} \subset A_{u}^{\mathbf{W}}\given φ_{{\widehat{a}} ^-} \big)\\
			& \leftstackrel{\eqref{eq:defAuW}}{=}  \prod_{y\in \mathbf{W}^{a} }   \mathbb{P}^{\rm E}_p(\mathcal{B}_y=1)                                     \Big(    1
			- \Pgff{\mathcal T^{\mathbf{W}}}{}\otimes\mathbb{P}_{p}^{\rm E}\Big(\bigcup_{y\in \mathbf{W}^{a}}
			\{\ab{φ_{y}} \leq 2\sqrt{2u}\} 		\cap 		\{\mathcal{E}_{y} \le 4u\lambda_y^{\mathbf{W}} \}		\Given φ_{{\widehat{a}} ^-} \Big)\Big) \\
			& \ge              p^L\bigg(                                   1
			-	\sum_{y\in \mathbf{W}^{a} } 	\Pgff{\mathcal T^{\mathbf{W}}}{}\big(\ab{φ_{y}} \leq 2\sqrt{2u}	\given φ_{{\widehat{a}} ^-}\big) \mathbb{P}_{p}^{\rm E}\big(\mathcal{E}_{y} \le 4u\lambda_y^{\mathbf{W}}\big)     \bigg)     \\
			& \leftstackrel{\eqref{PROP1-EQ-InequalityPhiU}}{\ge}
			p^L\bigg( 1
			-	\sum_{y\in \mathbf{W}^{a}} 	\frac{4\sqrt{2uλ_y^{\mathbf{W}}}}{\sqrt{2π}} \Big(1-e^{-4uλ_y^{\mathbf{W}}}\Big)     \bigg)                                              \\
			&\ge p^L\bigg( 1 -	\frac{16}{\sqrt{π}} u^{\frac32} \sum_{y\in \mathbf{W}^{a}} 	(λ_y^{\mathbf{W}})^{\frac32} \bigg),
		\end{split}
	\end{equation}
	taking advantage of the inequality $1-e^{-x} \le x$ for $x >0$ in order to obtain the last inequality.

	We now fix the parameters and start with choosing $c_λ,C_Λ,C_g, c_f,c_p>0$ as well as
	$L_0(B),$ with $B$ to be fixed later on, as the parameters from Proposition~\ref{PROP-Prop1},
	and $c_e$ as the parameter from \eqref{PROP1-DEF-Uprime}.
	Finally, for $u>0$ define

	\begin{equation}
		\label{eq:choiceLbaru}
		\widetilde{u}(u):= uc_e,
		\,L(u,B):=\left\lceil\frac{c_e}{3(1-e^{-uc_e})}\Big(\frac{\sqrt{π}}{32B}\Big)^{\frac 23}
		\right\rceil\vee L_0(B)
		\text{ and }p(u,B)=2^{-\frac1{L(u,B)}}.
	\end{equation}
	Using the bound $1-e^{-x}\geq x/2$ for $x>0$ small enough, we can now find $u_0=u_0(c_e,B)>0$ such that
	\begin{equation}
		\label{eq:Lusmallenough}
		L(u,B)\leq \frac{1}{u}\Big(\frac{\sqrt{π}}{32B}\Big)^{\frac 23}\text{ for all }u\in{(0,u_0]}.
	\end{equation}
	Then for all $u\in{(0,u_0)},$ under $\mathbf{P}_{L(u,B),\widetilde{u}(u)}^{{W}},$ for each $ (L(u,B), B, c_λ,C_Λ,C_g, c_f)$-good vertex $a\in{F},$ we can continue the chain of inequalities in \eqref{eq:boundprobaAu} to obtain
	\begin{equation}
		\label{eq:nextboundonprobaphiE}
		\begin{split}
			\Pgff{\mathcal T^{\mathbf{W}}}{}\otimes\mathbb{P}_{p(u,B)}^{\rm E}
			\big(
			\mathbf{W}^{a} \subset A_{u}^{\mathbf{W}}\given φ_{{\widehat{a}} ^-} \big)
			&\stackrel{\eqref{eq:boundprobaAu}}{\geq}
			p(u,B)^{L(u,B)} \Big(1 -	\frac{16}{\sqrt{π}} u^{\frac{3}{2}} \sum_{y\in \mathbf{W}^{a}} 	(λ_y^{\mathbf{W}})^{\frac32}     \Big)   \\
			&\stackrel{\eqref{eq:boundonconductances}}{\geq}
			p(u,B)^{L(u,B)}
			\Big(1-\frac{16}{\sqrt{π}}B{(uL(u,B))^{\frac32}}\Big)   \\
			&\stackrel{\eqref{eq:choiceLbaru},\eqref{eq:Lusmallenough}}{\geq}
			\frac12 \cdot \frac12 =
			\frac{1}{4}.
		\end{split}
	\end{equation}
	With our choice of parameters, see  in  particular \eqref{eq:choiceLbaru}, we can use Proposition~\ref{PROP-Prop1}  to show that the set $F^g$ from \eqref{eq:defFg} contains with positive probability a $d(c_dB^{-2/3})$-ary tree that we denote by $F^{g0},$ where $d(c_dB^{-2/3})$ will be large  (cf.\ \eqref{eq:limitd}), and $c_d:=c_ec_pc_f(\sqrt{\pi}/32)^{2/3}/3.$ Conditionally on the realization of the Galton--Watson tree $\mathcal{T}^{\mathbf{W}},$ and on the event that $F^{g0}$ exists,  we write
	\begin{equation}
		\label{eq:defFg1}
		\begin{split}
			&{F}^{g1}:=\{\emptyset\}\cup\left\{a\in{F^{g0}\setminus\{\emptyset\}}:\,\mathbf{W}^{a^-} \subset A_{u}^{\mathbf{W}}	\right\}\text{ and }\\
			&{\mathcal{F}}_a^{g1}:=\sigma\big(\ind_{\{\mathbf{W}^{(a')^-}\subset A_{u}^{\mathbf{W}}\}},a'\in{(F\setminus F_a)\cup\{a\}}\big)
		\end{split}
	\end{equation}
	for all $a\in{F},$ where $F_a$, the subtree below $a,$ was defined in the paragraph below \eqref{DEF-GenerationOfTree}, and where we use the convention $\mathbf{W}^{\emptyset^-}=\emptyset.$
	Taking advantage of the Markov property, see \eqref{LEMMA-MarkovProperty} and below, under $\Pgff{\mathcal{T}^{\mathbf{W}}}{}$ and conditionally on $φ_{\widehat{a}^-},$ the field $φ_{|\mathbf{W}^{a}}$ is independent of $φ_{\emptyset}$ and $φ_{|\mathbf{W}^{(a')^-}}$ for all $a'\in{(F\setminus F_a)\cup\{a\}}.$
	Thus,  for all $u\in{(0,u_0)}$ and  $a\in{\mathcal{X}},$  on the event that $F^{g0}$ exists and $a\in{F^{g1}}$ (which implies in particular that $a$ is good), we have that
	\begin{equation}
		\label{eq:boundonprobaFg1}
		\begin{split}
			\Pgff{\mathcal T^{\mathbf{W}}}{}\otimes\mathbb{P}_{p(u,B)}^{\rm E}\Big(\big|G_a^{F^{g1}}\big|\geq d(c_dB^{-2/3})
			\given \mathcal{F}_a^{g1},φ_{\emptyset}\Big)
			&=\Pgff{\mathcal T^{\mathbf{W}}}{}\otimes\mathbb{P}_{p(u,B)}^{\rm E}{(\mathbf{W}^{a} \subset A_{u}^{\mathbf{W}}\,|\,φ_{\widehat{a}^-})}
			\stackrel{\eqref{eq:nextboundonprobaphiE}}{\ge}
			\frac14.
		\end{split}
	\end{equation}
	Therefore, conditionally on the realization of the Galton--Watson tree $\mathcal{T}^{\mathbf{W}}$ and on the event that $F^{g0}$ exists, by Lemma~\ref{LEMMA-DependBernPercolation}, the set $F^{g1}$ contains with positive  $\Pgff{\mathcal T^{\mathbf{W}}}{}\otimes\mathbb{P}_{p(u,B)}^{\rm E}(\, \cdot\,|\,φ_{\emptyset})$-probability (not depending on $φ_{\emptyset}$) a $d\big(d(c_dB^{-2/3})/4\big)$-ary tree. Moreover, since
	\begin{equation}
		\label{eq:limitd}
		d\big(d(c_dB^{-2/3})/4\big) \rightarrow\infty\text{ as }B\rightarrow0,
	\end{equation}
	taking $B$ small enough we get that, under $\mathbf{E}_{L(u,B),\widetilde{u}(u)}^{{W}}[\Pgff{\mathcal T^{\mathbf{W}}}{}\otimes\mathbb{P}_{p(u,B)}^{\rm E}(\,\cdot\,|\,φ_{\emptyset})],$ the set $F^{g1}$ contains an infinite subtree with positive probability that we denote by $\delta,$ and which does not depend on $φ_{\emptyset}.$

	Write $p(u)=p(u,B)$ and $L(u)=L(u,B)$ for this choice of $B.$ For each $a\in{F^{g1}},$ we have $\mathbf{W}^{a^-}\subset A_{u}^{\mathbf{W}}\cap\mathcal{I}^u$ by  \eqref{eq:defFg}, \eqref{eq:WainclusIu} and \eqref{eq:defFg1}. Since $\widehat{a}\in{\mathbf{W}^a}$ and $\widehat{a}^-\in{\mathbf{W}^{a^-}}$ by construction, and so $\mathbf{W}^a$ and $\mathbf{W}^{a^-}$ are adjacent in $\mathcal{T}^{\mathbf{W}}$ (i.e.\ $\min_{x \in \mathbf{W}^{a^-}, \, y \in \mathbf{W}^{a}} d_{\mathcal{T}^{\mathbf{W}}}(x,y)=1$) the infinite connected tree in $F^{g1}$ yields an infinite connected subset $\bigcup_{a\in{F^{g1}}}\mathbf{W}^{a}$ in $\mathcal{T}^{\mathbf{W}}$ which is included in $A_{u}^{\mathbf{W}}\cap\mathcal{I}^u.$ Since $(\mathcal{T}^{\mathbf{W}},A_{u}^{\mathbf{W}},\mathcal{I}^u)$ under $\mathbf{E}_{L(u),\widetilde{u}(u)}^{\mathbf{W}}[\Pgff{\mathcal T^{\mathbf{W}}}{}\otimes\mathbb{P}_{p(u)}^{\rm E}(\cdot)]$ has the same law as $(\mathcal{T},A_{u}\cap B_{p(u)},\mathcal{I}^u)$ under $\mathbb{E}^{{\rm GW}}[\Pri_{\mathcal{T}}\otimes\Pgff{\mathcal{T}}{}\otimes\mathbb{P}_{p(u)}^{\rm E}(\cdot)]$ by \eqref{PROP1-DEF-Au}, \eqref{eq:TWisGW} and \eqref{eq:defAuW}, we proved that the root is included in an unbounded connected component of $A_u\cap B_p \cap \mathcal{I}^u$ with positive probability.
	\medskip

	In order to conclude, we still need to prove that percolation occurs almost surely.
	The strategy will be to construct a Galton--Watson tree $\mathcal{T}^Z$ such that there are conditionally independent copies of the tree $F^{g1}$ from \eqref{eq:defFg1} whose associated watersheds can all be embedded into $\mathcal{T}^Z.$ Since each of these copies of $F^{g1}$ is infinite with probability at least $\delta,$ at least one of them will be infinite a.s., and we can conclude. We now explain how to do this construction in detail.  Under some probability measure $\mathbb{P}^Z_u,$ let  $(Z_k)_{k\in\N}$ be an i.i.d.\ sequence of subtrees in $\mathcal{X},$ with the same law as the subtree
	\begin{equation*}
		\mathcal{T}_-^{\mathbf{W}}\cup\bigcup_{a\in{F}:\,\tilde{V}_L(\mathbf{X}^a)=H_{\widehat{a}1}(\mathbf{X}^a)=\infty}\mathcal{T}^{\widehat{a}1}
	\end{equation*}
	of $\mathcal{T}^{\mathbf{W}}$ under $\mathbf{P}_{L(u),\widetilde{u}(u)}^{{W}},$ where $\mathcal{T}_-^{\mathbf{W}}$ is defined in \eqref{DEF-WUnion} and $\mathcal{T}^x$ in \eqref{eq:defrestofthetrees}. Since $\mathcal{T}_-^{\mathbf{W}}$ is constructed by the use of watersheds, in a slight abuse of language we will also call watersheds the respective subsets of $Z_k$ corresponding to watersheds in $\mathcal{T}_-^{\mathbf{W}},$ if no confusion is to arise from this. Let us now define recursively a sequence of trees $\mathcal{T}^{Z}_k,$ $k\in{\N},$ with $\partial\mathcal{T}^{Z}_k\neq\varnothing,$ a.s.\ as follows: first take $\mathcal{T}^Z_1=Z_1.$ Note that $\partial Z_1\neq\varnothing$ a.s.\ since 
    either $\tilde{V}_L(\mathbf{X}^\emptyset)=\infty$, and then $\partial Z_1$ contains any point of $\partial (\mathbf{T}^\emptyset\setminus \mathbf{T}^\emptyset_{V_L(X^{\emptyset})}),$ which is a.s.\ non-empty; 
    or otherwise if $\tilde{V}_L(\mathbf{X}^\emptyset)<\infty$ then $\widehat{\emptyset1} \in{\partial Z_1}$ (which does not always corresponds to $\widehat{\emptyset}1$) since we did not add the tree $\mathcal{T}^{\widehat{\emptyset}1}$ in the definition of $Z_1$ and $\widehat{\emptyset1}\in{\partial \mathcal{T}_-^{\mathbf{W}}}$ by \eqref{eq:deffreepointrecursively}.

	To define $\mathcal{T}_k^Z$ recursively, assume that $\mathcal{T}^Z_{k-1}$ is defined with $\partial\mathcal{T}^Z_{k-1}\neq\varnothing.$ Let $x_k$ be the first vertex in $\partial \mathcal{T}^Z_{k-1}$ (in lexicographic order in Ulam-Harris notation). We then define $\mathcal{T}_k^Z$ as the union of $\mathcal{T}_{k-1}^Z$ and $x_k\cdot Z_k,$ which also verifies $\partial \mathcal{T}^Z_k\neq\varnothing.$

	Let $\mathcal{T}_-^Z$ be the union of $\mathcal{T}_k^Z,$ $k\in{\N},$ and $\mathcal{T}^Z$ be the union of $\mathcal{T}_-^Z$ and some additional independent Galton--Watson trees below each $x\in{\partial \mathcal{T}_-^Z},$ each with the same law as $x\cdot\mathcal{T}$ under $\mathbb{P}^{\rm GW}.$ Then, by construction, $\mathcal{T}^Z$ has the same law as the usual Galton--Watson tree $\mathcal{T}$ under $\mathbb{P}^{\rm GW}.$ Define $F_k^{g0}$ and $\mathbf{W}^a_k,$ $a\in{F_k^{g0}},$ similarly as above \eqref{eq:defFg1} and in \eqref{DEF-watershed}, but corresponding to $Z_k,$ which are i.i.d.\ copies of $F^{g0}$ and $\mathbf{W}^a,$ $a\in{F^{g0}},$ in $k\in{\N}.$ Moreover, under $P^Z_u:=\mathbb{E}^Z_u[\Pgff{\mathcal{T}^Z}{}\otimes\mathbb{P}_{p(u)}^{\rm E}(\cdot)],$ define $A_{u}^Z$ similarly as in \eqref{eq:defAuW}, but with  $\mathcal{T}^{\mathbf{W}}$ replaced by $\mathcal{T}^Z,$ and for each $k\in{\N},$ take $F_k^{g1}=\{a\in{F_k^{g0}}:\,x_k\cdot\mathbf{W}^{a^-}_k\subset A_{u}^{Z}\},$ similarly as in \eqref{eq:defFg1}. Then by Markov's property for the Gaussian free field, conditionally on $φ_{x_k},$ $F_k^{g1}$ is independent of $F_i^{g1},$ $i<k,$ and thus for each $u\in{(0,u_0)}$ we have
	\begin{equation}
		\label{eq:PZFkg1infty}
		{P}^{Z}_u\big(|F_k^{g1}|=\infty\,|\,F_i^{g1}, i<k\big)={E}_u^{Z}\big[{P}_u^{Z}(|F_k^{g1}|=\infty\,|\,φ_{x_k})\,|\,F_{i}^{g1},i<k\big]\geq\delta;
	\end{equation}
	here, the last inequality follows from the fact that, for each $a\in{\R},$ the law of $F_k^{g1}$ conditionally on $φ_{x_k}=a$ under $P^Z_u$ is the same as the law of $F^{g1}$ conditionally on $φ_{\emptyset}=a$ under $\mathbf{E}_{L(u),\widetilde{u}(u)}^{\mathbf{W}}[\Pgff{\mathcal T^{\mathbf{W}}}{}\otimes\mathbb{P}_{p(u)}^{\rm E}(\cdot)],$ and $\delta$ is the constant introduced below \eqref{eq:boundonprobaFg1}. Using the tower property recursively on $k\in{\N},$ one can easily show that \eqref{eq:PZFkg1infty} implies that there exists $P^Z_u$-a.s.\ $k_0\in{\N}$ such that $|F_{k_0}^{g1}|=\infty.$  Note moreover that one can use Proposition~\ref{PROP-W_Gw_WInterl} similarly as in the proof of Proposition~\ref{PROP-GoodwatershedsInInterlacements}, to obtain an interlacements $\mathcal{I}^{u}$ on $\mathcal{T}^Z$ with $x_k\cdot\mathbf{W}^a_k\subset\mathcal{I}^{u}$ for each $a\in{F_k^{g0}}$ and $k\in{\N}.$ To this effect, note in particular that \eqref{PROP1-EQ-ConstantCE} still holds on $\mathcal{T}^Z$ since for each $k\in{\N}$ and $a\in{F_k^{g0}},$ the subtree $\mathcal{T}^Z_{x_k\cdot\widehat{a}1}$ of $\mathcal{T}^Z$ below $x_k\cdot\widehat{a}1$ is the copy $\mathcal{T}^{\widehat{a}1}_k$ of $\mathcal{T}^{\widehat{a}1}$ associated to $Z_k,$ translated by $x_k.$ Therefore, for each $u\in{(0,u_0)},$ the set $F_{k_0}^{g1}$ is $P_u^Z$-a.s.\ infinite and its associated watersheds $\mathbf{W}^a_{k_0},$ $a\in{F_{k_0}^{g1}},$ are included in $\mathcal{I}^{u}\cap A_{u}^Z,$ and we can conclude.

\end{proof}

In order to deduce Theorem~\ref{THM-h*>0} from Proposition~\ref{PROP-Percolation-Au}, we are going to use the isomorphism \eqref{EQ-IsomorphismTheorem} between the Gaussian free field and random interlacements. We first show that condition \eqref{eq:capRWinfty} -- which entails the validity of the isomorphism \eqref{EQ-IsomorphismTheorem} by Proposition~\ref{THM-Isomorphism} -- holds $\Pgw\barAS\ $ for the Galton--Watson tree $\mathcal{T}$.

\begin{prop}\label{prop:capRWinfty}
	$\Pgw$-almost surely we have that for all $x\in{\mathcal{T}}$,
	\begin{equation*}
		\text{$P_x^{\mathcal{T}}(\, \cdot\,|\,H_{x^-}=\infty)$-almost surely,} \quad \capac_{\mathcal{T}}( \{X_i,i \in\N \})=∞.
	\end{equation*}
\end{prop}
\begin{proof}
	Let $x\in{\mathcal{X}}$ and $L\in{\N}.$ Under some probability $\tilde{\mathbf{Q}}_x^L,$ we now define a tree $\tilde{\mathbf{T}},$ with weights denoted by $\tilde{\boldsymbol{\lambda}}_{y,z},$ $y,z\in{\tilde{\mathbf{T}}},$ $y\sim z,$ as some extension of the tree $\mathbf{T}_{V_L}$ starting at $x$ from Section~\ref{Section-Singlewatershed}, by completing its remaining ends so that $\tilde{\mathbf{T}}$ is a Galton--Watson tree conditioned on $x\in{\tilde{\mathbf{T}}}$. More precisely, first define $\tilde{\mathbf{T}}\setminus\tilde{\mathbf{T}}_x,$ that is the part of the tree $\tilde{\mathbf{T}}$ which is not below $x,$ with the same law as $\mathcal{T}\setminus\mathcal{T}_x$ under $\Pgw(\,\cdot\,|\,x\in{\mathcal{T}}),$ endowed with the corresponding weights. Then, attach to $x$ a copy of the tree $\mathbf{T}_{{V}_L}$ with the same law as under $\mathbf{Q}_x^{\tilde{\boldsymbol{\lambda}}_{x^-,x},L},$ as defined in Section~\ref{Section-Singlewatershed}. With a slight abuse of notation, we see $\mathbf{T}_{V_L}$ as a subset of $\tilde{\mathbf{T}}.$ Finally for each remaining point $y\in{\partial\mathbf{T}_{{V}_L}},$ attach to $y$ an independent copy of $y\cdot\mathcal{T}.$  Let $\tilde{\mathbf{X}}$ be a process with the same law as $(\mathbf{X}_{k\wedge V_L})_{k\in{\N_0}}$ under $\mathbf{Q}_x^{\tilde{\mathbf{\lambda}}_{x^-,x},L},$ it follows easily from Proposition~\ref{PROP-SinglewatershedLikeGW} that $(\tilde{\mathbf{T}},\tilde{\mathbf{X}})$ under $\tilde{\mathbf{Q}}_x^L$ has the same law as $(\mathcal{T},(X_{k\wedge{V}_L})_{k\in{\N_0}})$ under  $\mathbb{E}^{{\rm GW}}[P^{\mathcal{T}}_x(\cdot)\,|\,x\in{\mathcal{T}}].$

	Similarly as in the proof of Lemma~\ref{LEMMA-Goodwatershed} \ref{item:lemma iv}, one can show that there exist positive constants $c_{\lambda}$ and $c_f$ so that,
	for each $ε>0,$ if $L$ is large enough, then

	$$ \tilde{\mathbf{Q}}_{x}^{L}\Big( \big |\big\{
	y\in{\partial{\mathbf{T}}_{V_L}}:
	\,\tilde{\boldsymbol{λ}}_{y,y^-}\geq c_λ\big \}\big| < c_fL,\,
	V_L(\tilde{\mathbf{X}})< H_{x^-}(\tilde{\mathbf{X}}) \Big)\leε. $$
	Indeed, this follows easily from \eqref{eq:lambdaonWareiid} and a reasoning similar to the one in \eqref{eq:firstpartproofiv}, \eqref{eq:secondpartproofiv} and \eqref{eq:propertyivissame}, replacing $\{\sum_{i\in{\N}}\boldsymbol{\lambda}_{i}^{(k)}\le C_Λ\}$ by $\{\exists\, i\in{\N}:\,\boldsymbol{\lambda}_i^{(k)}\ge c_{\Lambda}\}.$

	Since, conditionally on $\mathbf{T}_{V_L},$ $g^{\tilde{\mathbf{T}}_y}(y,y),$ $y\in{\partial\mathbf{T}_{V_L}},$ are i.i.d.\ with the same law as $g^{\mathcal{T}}(\emptyset,\emptyset),$ by the law of large number and the bound on the Green function from Lemma~\ref{LEMMA-Goodwatershed} \ref{item:lemma ii} we deduce that for $L$ large enough

	$$ \tilde{\mathbf{Q}}_{x}^{L}\Big( \big |\big\{
	y\in{∂{\mathbf{T}}_{V_L}}:
	\,\tilde{\boldsymbol{λ}}_{y,y^-}\ge c_λ,\,
	\,g^{\tilde{\mathbf{T}}_y}(y,y)\le C_g\
	\big \}\big| < \frac{c_f}{4}L ,V_L(\tilde{\mathbf{X}})< H_{x^-}(\tilde{\mathbf{X}}) \Big)\le 2ε. $$

	Note that the event $\{\tilde{\boldsymbol{λ}}_{y,y^-}\ge c_λ,
	\,g^{\tilde{\mathbf{T}}_y}(y,y) \le C_g \}$ implies by a similar reasoning  to above \eqref{PROP1-EQ-ConstantCE} that $g^{\tilde{\mathbf{T}}_{y^-}}(y^-,y^-) \le C_g+\frac{1}{c_λ}.$ Let $\tilde{\mathbf{W}}=\{\tilde{\mathbf{X}}_0,\dots,\tilde{\mathbf{X}}_{V_{L}}\}.$ Recalling the definition of the equilibrium measure from \eqref{DEF-EquilibriumMeasure}, we moreover have that  $e_{\tilde{\mathbf{W}},\tilde{\mathbf{T}}}(z)= e_{\{z\},\tilde{\mathbf{T}}_z}(z)=(g^{\tilde{\mathbf{T}}_{z}}(z,z))^{-1}$ for each $z\in{\partial\tilde{\mathbf{W}}\setminus\{x\}}.$ Since $y^-\in{\partial\tilde{\mathbf{W}}}$ for each $y\in{\partial\mathbf{T}_{V_L}}$ by construction, we deduce that for $L$ large enough
	\begin{equation*}
		\tilde{\mathbf{Q}}_{x}^{L}\Big( \mathrm{cap}_{\tilde{\mathbf{T}}}(\tilde{\mathbf{W}}) < \frac{c_f}{4(C_g+1/c_{\lambda})}L ,V_L(\tilde{\mathbf{X}})< H_{x^-}(\tilde{\mathbf{X}}) \Big)\le 2ε.
	\end{equation*}
	Since $\tilde{\mathbf{W}}$ has the same law under $\tilde{\mathbf{Q}}_x^L(\cdot,\,V_L(\tilde{\mathbf{X}})< H_{x^-}(\tilde{\mathbf{X}}))$  as the first $L$ points visited by $X$ under  $\mathbb{E}^{{\rm GW}}[P^{\mathcal{T}}_x(\cdot,\,V_L(X)<H_{x^-}(X))\,|\,x\in{\mathcal{T}}],$ letting first $L\to∞$ and then $ε\to 0,$ and noting that  $\{V_L(X)<H_{x^-}(X)\}$ decreases to $\{H_{x^-}(X)=\infty\},$ we readily obtain \eqref{prop:capRWinfty}.
\end{proof}

We can now deduce Theorem~\ref{THM-h*>0} from Proposition~\ref{PROP-Percolation-Au} using the isomorphism from Proposition~\ref{THM-Isomorphism} combined with Proposition~\ref{prop:capRWinfty}.

\begin{proof}[Proof of Theorem~\ref{THM-h*>0}]
	Consider the probability space $\mathbb{Q}_{\mathcal{T}}^u$ from Proposition~\ref{THM-Isomorphism}. Abbreviating $\mathcal{E}_x:=\mathcal{E}_{x}^{(1)},$ we have $\ell_{x,u}\geq\lambda_x^{-1}\mathcal{E}_x$ for all $x\in{\mathcal{I}^u}$ by \eqref{eq:deflocaltimesRI}. In view of Proposition~\ref{prop:capRWinfty}, we can apply the isomorphism \eqref{EQ-IsomorphismTheorem}, and we get $\mathbb{Q}_{\mathcal{T}}^u$-a.s.\ for all $ x\in \mathcal{I}^u\cap A_u $
	\begin{align*}
		\gamma_x & = -\sqrt{2u} + \sqrt{2\ell_{x,u}+φ^2_x} \geq -\sqrt{2u} + \sqrt{2\lambda_x^{-1}\mathcal{E}_x+φ^2_x}
		\stackrel{\eqref{PROP1-DEF-Au}}{\ge} -\sqrt{2u} + 2\sqrt{2u} = \sqrt{2u}.
	\end{align*}
	This yields \eqref{eq:Iu+Auincluded} by defining $\widehat{E}^{\geq\sqrt{2u}}= \{x\in{\mathcal{T}}:\gamma_x \ge \sqrt{2u}\}.$ By Proposition~\ref{PROP-Percolation-Au}, for all $ u\in(0,u_0)$ there is $\mathbb{Q}_{\mathcal{T}}^u$-a.s.\ an unbounded component for $A_u\cap\mathcal{I}^u,$ and so also for the level set $\widehat{E}^{\geq\sqrt{2u}}.$ This readily implies $ h_*>0$ since $\widehat{E}^{\geq\sqrt{2u}}$ has the same law as $E^{\geq\sqrt{2u}}.$
\end{proof}

\begin{remark}
	Rather surprisingly, our proof does not work anymore if one tries to replace the inclusion \eqref{eq:Iu+Auincluded} by any of the simpler inclusions $\mathcal{I}^u\cap \left\lbrace	x\colon
	\mathcal{E}_x>4u\lambda_x	\right\rbrace\subset \widehat{E}^{\geq \sqrt{2u}}$ or  $\mathcal{I}^u\cap \left\lbrace	x\colon
	\ab{φ_x}>2\sqrt{2u}	\right\rbrace\subset \widehat{E}^{\geq \sqrt{2u}}.$ In other words, we need to use both the local times of random interlacements and the Gaussian free field $φ$ in the isomorphism \eqref{EQ-IsomorphismTheorem}, and not just one of the two. Indeed, in view of Proposition~\ref{PROP-Prop1}, one needs to take $L$ at least equal to $C/u$ for some large constant $C<\infty$ in order for $F^g$ to percolate. For instance for constant conductances and small enough $u,$ the probability that $\mathbf{W}^a\subset\left\lbrace	x\colon
	\mathcal{E}_x>4u\lambda_x	\right\rbrace$ is at least $1-CuL,$ and the probability that $\mathbf{W}^a\subset\left\lbrace	x\colon
	\ab{φ_x}>2\sqrt{2u}\lambda_x	\right\rbrace$ is of order $1-C\sqrt{u}L$ in view of \eqref{PROP1-EQ-InequalityPhiU}, for some constant $C<\infty.$ These bounds are not interesting for the previous choice of $L=C/u.$ However combining them gives that the probability that $\mathbf{W}^a\subset A_u$ is of order $1-Cu^{3/2}L,$ see \eqref{eq:boundprobaAu}, which goes to one for the previous choice of $L$ when $u\rightarrow0.$
\end{remark}

\begin{proof}[Proof of Theorem~\ref{THM-quenchednoise}]
	The statement for random interlacements follows trivially from Proposition~\ref{PROP-Percolation-Au} for $u\leq u_0$ by the inclusion $\mathcal{I}^u\cap A_u \cap B_p \subseteq \mathcal{I}^u \cap B_p$. Using the monotonicity in $u$ of interlacements we obtain the statement for all $u>0.$ The statement for the Gaussian free field also follows from Propositions~\ref{PROP-Percolation-Au}, \ref{THM-Isomorphism} and \ref{prop:capRWinfty} similarly as in the proof of Theorem~\ref{THM-h*>0}.
\end{proof}

\begin{remark}
	An interesting open question is whether Theorem~\ref{THM-quenchednoise} is true in the whole supercritical phase of the Gaussian free field, that is for each $h<h_*,$ does there exist $p\in{(0,1)}$ such that $E^{\geq h}\cap B_p$ percolates, or is  transient even?
\end{remark}

\section{Transience of the level sets}\label{SEC-transience}

In this section we prove Theorem~\ref{THM-Transience}, that is that both, the interlacements set and the level sets of the Gaussian free field above small positive levels, are transient -- even when intersected with a small Bernoulli noise. More precisely, we prove that the random walk on the tree of very good watersheds is transient, see Proposition~\ref{prop:transientgoodwat}, and use arguments similar to the proof of Theorem~\ref{THM-h*>0} to conclude. The notion of very goodness we use here is a refinement of the one introduced in Definition~\ref{DEF-Goodwatershed}, see \eqref{item:iv'} below, and is adapted in order to  ensure that the random walk on the tree of very good watersheds can be compared to a random walk on a Galton--Watson with a constant drift, see \eqref{eq:drift}. We then follow the strategy of the proof of \cite[Theorem~1]{Collevecchio} in order to deduce transience. In addition to the usual assumption \eqref{DEF-Conductances+}, we assume throughout this section that, conditionally on the non-weighted tree $\mathcal{T},$ the family $(\lambda_{x,y})_{x\sim y\in{\mathcal{T}}}$ is i.i.d.\ and has compact support. In terms of the construction of the Galton--Watson tree in Section~\ref{SECTION-GWTree}, this is equivalent to assuming that, under $\nu$ and conditionally on $\pi((\lambda_j)_{j\in{\N}}),$ the family $(λ_{i})_{1 \le i\leq \pi((\lambda_j)_{i\in{\N}})}$ is  i.i.d., that the law of $\lambda_1$ does not depend on $\pi((\lambda_j)_{j\in{\N}}),$ and  that there exist $0<\overline{c}_λ<\overline{C}_Λ<∞$ such that $ν$-$\as$
\begin{equation}\label{eq-BoundedConductancesIn6}
	\overline{c}_λ<λ_{i}<\overline{C}_Λ\text{ for all }1 \le i\leq \pi((\lambda_j)_{j\in{\N}}).
\end{equation}
We use the independence of the conductances when referring to \cite{Gantert2012} in the proof of Lemma~\ref{lem:probatobeTRANgood}, and the assumption \eqref{eq-BoundedConductancesIn6} in \eqref{eq:drift}. Note that \eqref{eq:assfinitemoment} and \eqref{eq-BoundedConductancesIn6} imply that the mean offspring distribution $m$ is finite.

Let us now define a notion of goodness which is stronger than the one introduced in Definition~\ref{DEF-Goodwatershed}: in this section, we say that a point $a\in F$ is $(L, B ,C_g, c_f, c_L)$-\emph{very good} if it verifies the conditions \ref{item:i} to \ref{item:iii} with $c_λ=\overline{c}_λ$ and $C_Λ=\overline{C}_Λ$ (which simplifies these conditions in view of \eqref{eq-BoundedConductancesIn6}), and \ref{item:v} of Definition~\ref{DEF-Goodwatershed}, as well as

\begin{enumerate}[$iv)'$]
	\item \label{item:iv'} the set of children of the vertex $a$ in the tree of free points $F$ satisfies
	\begin{equation*}
		\big |\big\{a'\in{G_{a}^F}\colon
		d_{\mathcal{T}^{\mathbf{W}}}(\widehat{a},\widehat{a'}) \ge c_LL \big \}\big|\geq  \dfrac{c_f L}{2},
	\end{equation*}
\end{enumerate}
where we recall that $d_{\mathcal{T}^{\mathbf{W}}}$ denotes the graph distance within ${\mathcal{T}^{\mathbf{W}}}.$
Note that the inequality $λ_{a,a'}^F\le C_Λ=\overline{C}_Λ$ is trivially satisfied under \eqref{eq-BoundedConductancesIn6} by taking $C_Λ=\overline{C}_Λ,$ and thus \ref{item:iv'} is stronger than \ref{item:iv} in Definition~\ref{DEF-Goodwatershed} (up to changing the constant $c_f$).
We now follow a strategy inspired by that of Section~\ref{SECTION-Percolation} in order to show that the tree of very good free points contains a $d$-ary tree. We first evaluate the probability for a point to verify the property \ref{item:iv'}, analogously to Lemma~\ref{LEMMA-Goodwatershed} \ref{item:lemma iv}. Recall the construction of the trees $\mathbf{T}_k,$ $k\in{\N_0},$ under the probability measure $\mathbf{Q}_x^{κ,L}$ from Section~\ref{Section-Singlewatershed}, as well as the stopping time $V_L(\mathbf{X})$ and $\tilde{V}_L(\mathbf{X})$ from \eqref{DEF-StoppingTimeVn} and \eqref{DEF-StoppingTimeHL}. In what follows we abbreviate $V_L=V_L(\mathbf{X})$ to simplify notation.

\begin{lemma}\label{lem:probatobeTRANgood}
	Let $c_f$ be as in Lemma~\ref{LEMMA-Goodwatershed}. There exists $c_L>0$ such that for all $ε>0,$ there exists $L_0=L_0(ε)\in{\N}$ such that for all $x\in{\mathcal{X}},$ $L\geq L_0$ and $κ\leq \overline{C}_\Lambda,$
	\begin{equation*}
		\mathbf{Q}_{x}^{κ,L} \Big( \big |\big\{y\in{\partial\mathbf{T}_{V_L}\setminus\{x1,\mathbf{X}_{V_L}\}}:\,d_{\mathbf{T}_{V_L}}(x,y)\ge c_LL\big \}\big|< c_fL/2,\tilde{V}_L(\mathbf{X})=\infty\Big)	\leq ε.
	\end{equation*}
\end{lemma}

\begin{proof}
	It is known, see \cite[Theorem~17.13]{LyonsPeres17}, that the speed of a random walk on a Galton--Watson tree $\mathcal{T}$ with unit conductances is $\Pgw$-\as\, strictly positive and deterministic; i.e., the limit $v:=\lim_{k\to \infty}\frac{d_{\mathcal{T}}(\emptyset,X_k)}{k}>0$ exists and is a constant. This result was generalized in \cite{Gantert2012} to Galton--Watson trees with finite mean for the offspring distribution and i.i.d.\ conductances verifying \eqref{DEF-Conductances+}.  In view of Proposition~\ref{PROP-SinglewatershedLikeGW}, the process $\mathbf{X}$ under $\mathbf{Q}_x^{κ,L}( \  \cdot \  , \tilde{V}_L(\mathbf{X})=\infty)$ has the same law as a random walk $X$ on $\mathcal{T}$ under $P_x^{\mathcal{T}}( \  \cdot \ , \tilde{V}_L(X)=\infty\,| \, \lambda_{x,x^-}=κ).$ Therefore, for all $ε>0$ we can find a $k_0=k_0(ε)$ such that for all $k>k_0,$ $L\in{\N},$ $x\in{\mathcal{X}}$ and $κ\leq \overline{C}_Λ,$ we have
	\begin{equation}\label{eq-RWDrift}
		\mathbf{Q}_x^{κ,L} \Big(\exists \, n\geq k:\,d_{\mathbf{T}_{\tilde{V}_L}}(\mathbf{X}_{n},x)\le v k/2,\tilde{V}_L(\mathbf{X})=\infty \Big)\le ε/3.
	\end{equation}
	In order to find enough vertices in $\mathfrak{F}_a$ at distance at least $c_L$ from $x$, we note that $\ab{\mathbf{T}_k}\leq \ab{\mathbf{T}_{V_k}}= \sum_{x\in \{\mathbf{X}_1,\dots,\mathbf{X}_{V_k}\}}|\{x\}\cup G_x^{\mathbf{T}_{V_k}}|$, and that $\{G_x^{\mathbf{T}_{V_k}}:\,x\in{\{X_1,\dots,X_{V_k}}\}\}$ is an i.i.d.\ family of cardinality $k$ if $\tilde{V}_L=\infty,$ $k\leq L,$ similarly as in \eqref{eq:lambdaonWareiid}. Since $m<\infty,$ by the weak law of large number we can find $C_P>0$ such that for all $ε>0,$ there exists $k_0\in{\N}$ such that for all $k>k_0,$ $L\geq k,$ $x\in{\mathcal{X}}$ and $κ>0$
	\begin{equation}\label{eq-freePointsAbove}
		\mathbf{Q}^{κ,L}_x (\ab{\mathbf{T}_k}\ge C_P k,\tilde{V}_L(\mathbf{X})=\infty ) \le ε/3.
	\end{equation}
	Applying \eqref{eq-RWDrift} and \eqref{eq-freePointsAbove} with $k=\frac{c_f}{2C_P}L,$ for $L$ large enough so that $k\ge k_0,$ we obtain that with probability at most $2ε/3,$ on the event $\tilde{V}_L(\mathbf{X})=\infty,$ there are more than $c_fL/2$ points in ${\mathbf{T}_{{V}_L}}$ at distance less than $c_LL$  from $x,$ where $c_L:=\frac{v c_f}{4C_P}.$ We can then conclude by combining this with Lemma~\ref{LEMMA-Goodwatershed} \ref{item:lemma iv} for $ε/3.$

\end{proof}

Recall the definition of $A^{\mathbf{W}}_{u}$ in \eqref{eq:defAuW}.
We can now prove analogously to the proof of Proposition~\ref{PROP-Percolation-Au} that $(L,B,C_g, c_f, c_L)$-\emph{very good} points, whose associated watershed is included in $A_{u}^{\mathbf{W}},$ contain a supercritical Galton--Watson tree.

\begin{prop}
	\label{prop:percoFg1'}
	Let $c_λ=\overline{c}_λ$, $C_g$ and $c_f$ be as in  Lemma~\ref{LEMMA-Goodwatershed},  $c_e$ as in \eqref{PROP1-DEF-Uprime}, and $c_L$ as in Lemma~\ref{lem:probatobeTRANgood}. For each $d\in{\N},$ there exist $B>0$ and $u_0>0,$ such that, for each $u\in{(0,u_0)},$ there exist $L\in{\N}$ and $p\in{(0,1)},$ so that under $\mathbf{E}_{L,\widetilde{u}}^{{W}}[\Pgff{\mathcal T^{\mathbf{W}}}{}\otimes\mathbb{P}_{p}^{\rm E}(\, \cdot\,|\,φ_{\emptyset})],$ with $\widetilde{u}=uc_e,$ the tree
	\begin{equation*}
		\begin{split}
			F^{g1'}\!:=\!\{\emptyset\}\cup\Big\{a \in F\setminus\{\emptyset\} :\, a^- \text{ is $(L, B,C_g, c_f, c_L)$-\emph{very good}},\\
			d_{\mathcal{T}^{\mathbf{W}}}(\widehat{a},\widehat{a^-})\geq c_LL\text{ and }\mathbf{W}^{a^-}\subseteq A^{\mathbf{W}}_{u} \Big\}
		\end{split}
	\end{equation*}
	contains with positive probability, not depending on $φ_{\emptyset},$ a $d$-ary tree.
\end{prop}
\begin{proof}
	Using Lemma~\ref{lem:probatobeTRANgood} in place of Lemma~\ref{LEMMA-Goodwatershed} \ref{item:lemma iv}, and adding the condition $d_{\mathcal{T}^{\mathbf{W}}}(\widehat{a},\widehat{a^-})\geq c_LL$ in the definition \eqref{eq:defFg} -- which is possible in view of the condition \ref{item:iv'} --  one can easily prove similarly as below \eqref{eq:boundonprobaFg1} that for each $B>0$ there exists $u_0=u_0(B),$ such that for all $u\in{(0,u_0)},$ there exists $L=L(u,B)$ and $p=p(u,B)$ as in \eqref{eq:choiceLbaru},  so that $F^{g1'}$ contains a $d\big(d(c_dB^{-2/3})/4\big)$-ary tree, and we can conclude in view of \eqref{eq:limitd}.

\end{proof}

We prove now transience using the argument of \cite[Theorem~1]{Collevecchio}.
\begin{prop}
	\label{prop:transientgoodwat}
	There exists $B>0,$ $u>0,$ $L\in{\N}$ and $p\in{(0,1)},$ such that under $\mathbf{E}_{L,uc_e}^{{W}}[\Pgff{\mathcal T^{\mathbf{W}}}{}\otimes\mathbb{P}_{p}^{\rm E}(\,\cdot\,|\,φ_{\emptyset})],$ the connected component of $\emptyset$ in the tree with vertex set
	\begin{equation*}
		\mathcal{T}^{g1'}:=\bigcup_{a\in F^{g1'}} \mathbf{W}^a
	\end{equation*}
	is transient with positive probability, not depending on $φ_{\emptyset}$.
\end{prop}
\begin{proof}
	Consider a random walk $X$ on $\mathcal{T}^{g1'}$  starting in $\emptyset.$ We proceed by contradiction, and assume that $\mathcal{T}^{g1'}$ is recurrent, that is, the walk $X$ comes back to the root almost surely. We introduce the following color scheme: $\emptyset$ is white, and a vertex $ai\in F^{g1'}$ is white if ${a}$ is white and $\!\widehat{\,ai\,}\!$ is visited by $X$ in the interval $[H_{\widehat{a}}, \inf\{k\geq H_{\widehat{a}}:\,X_k={\widehat{a^-}}\}]$. We want to show that there is an infinite number of white vertices with positive probability; indeed, since then there would in particular be an infinite connected component of white vertices, this would constitute a contradiction as the watershed associated to each white vertex in the connected component of $\emptyset$ is visited by $X$ in the interval $[H_{(\mathbf{W}^{\emptyset})^c},\inf\{k\geq H_{(\mathbf{W}^{\emptyset})^c}:\,X_k={\emptyset}\}]$ by definition.

	For a fixed vertex ${ai}\in{F^{g1'}}$, we evaluate the probability, starting from $\widehat{a},$ to visit $\!\widehat{\,ai\,}\!$ before returning to $\widehat{a^-}$. Because of recurrence, for the computation of this probability, we can restrict ourselves to the only path connecting $\widehat{a^-}$ to $\!\widehat{\,ai\,}\!$ and we compute its effective conductance $\mathcal{C}$ (see \cite[(2.4)]{LyonsPeres17}). Both the distances between $\widehat{a^-}$ and $\widehat{a}$, and the one between $\widehat{a}$ and $\!\widehat{\,ai\,}\!$ are at least $c_LL$ by definition of $F^{g1'},$ and at most $L$ by definition of watersheds, see in particular \eqref{DEF-StoppingTimeVn} and \eqref{DEF-FreePoints}. Therefore, using the series law (see \cite[Subsection 2.3.I]{LyonsPeres17}) we obtain that the probability of a random walk starting from $\widehat{a},$ to visit $\!\widehat{\,ai\,}\!$ before returning to $\widehat{a^-},$ is equal to
	\begin{equation}
		\label{eq:drift}
		\frac{\mathcal{C}(\widehat{a}\leftrightarrow \widehat{\,ai\,})}
		{\mathcal{C}(\widehat{a^-}\leftrightarrow \widehat{a})+
			\mathcal{C}(\widehat{a}\leftrightarrow \widehat{\,ai\,})}
		=
		\frac{
			\Big(\sum_{x\in (\widehat{a},\widehat{\,ai\,}]} \frac{1}{λ_{x^-, x}} \Big)^{-1}}
		{   \Big(\sum_{x\in (\widehat{a^-},\widehat{a}]} \frac{1}{λ_{x^-, x}} \Big)^{-1}
			+   \Big(\sum_{x\in (\widehat{a},\widehat{\,ai\,}]} \frac{1}{λ_{x^-, x}} \Big)^{-1}}
		\stackrel{\eqref{eq-BoundedConductancesIn6}}{\ge} \frac{\overline{c}_λ}{\overline{C}_Λ}\frac{c_L}{2},
	\end{equation}
	where $(x,y]$ denotes the unique path connecting $x$ to $y,$ minus $x.$ For each $d\in{\N},$ it follows from Proposition~\ref{prop:percoFg1'} that for an appropriate choice of $B,u,L$ and $p,$ the tree of white vertices contains with positive probability a weightless Galton--Watson tree with mean offspring distribution larger than $d\frac{\overline{c}_λ}{\overline{C}_Λ}\frac{c_L}{2}.$ Taking $d=\lceil 4 \frac{\overline{C}_Λ}{\overline{c}_λ c_L}\rceil,$ this tree of white vertices is infinite with positive probability, which concludes the proof.
\end{proof}

\begin{proof}[Proof of Theorem~\ref{THM-Transience}]
	Similarly to the proofs of Theorems~\ref{THM-h*>0} and \ref{THM-quenchednoise} at the end of Section~\ref{SECTION-Percolation}, one can use the isomorphism \eqref{EQ-IsomorphismTheorem}, which holds by Proposition~\ref{prop:capRWinfty} similarly as in the proof of Theorem~\ref{THM-h*>0}, as well as Proposition~\ref{PROP-GoodwatershedsInInterlacements} to show that the component of $\emptyset$ in the tree $\mathcal{T}^{g1'}$ from Proposition~\ref{prop:transientgoodwat} can be included in $\mathcal{I}^u \cap B_p$ or $\widehat{E}^{\ge \sqrt{2u}} \cap B_p,$ proving the transience of those sets with positive probability by Rayleigh’s Monotonicity Principle (see \cite[Section 2.4]{LyonsPeres17}). To show that transience occurs almost surely for some component, one can proceed similarly to the end of the proof of Theorem~\ref{PROP-Percolation-Au} by considering the Galton--Watson tree $\mathcal{T}^Z$ on which there are infinitely many conditionally independent copies of $\mathcal{T}^{g1'},$ and thus one of these copies is transient a.s.
\end{proof}

\begin{appendix}
\section*{Appendix: The critical parameter \texorpdfstring{$h_*$}{h*} is constant}
\setcounter{secnumdepth}{0}
\renewcommand*{\thedeff}{A.\arabic{deff}}
\setcounter{deff}{0}
\setcounter{equation}{0}
\renewcommand{\theequation}{A.\arabic{equation}}

\label{SECTION-uhDeterministic}

In this section we prove that $h_*(\mathcal T)$ does not depend on the realization of the Galton--Watson tree~$\mathcal T.$
\begin{thm} \label{PROP-h*constant}
	$\mathcal{T}\mapsto h_*(\mathcal{T})$ is  constant $ \Pgw$-almost surely.
\end{thm}

This result is known
in the case of deterministic unit conductances \cite{AbacherliSznitman2018}. We provide here a proof for the generalized case of random conductances. It proof is based on the 0-1 law for inherited properties of \cite[Proposition~5.6]{LyonsPeres17}, which we shortly recall here. For this purpose, we start with  the following definition.
\begin{deff}
	A property $\mathcal{P} $ (of trees) is called \emph{inherited} if the following holds true: When a tree $ T $ with root $x$ has property $ \mathcal{P} $, then all the subtrees $ T_y,$ $ y\in{G_x^T},$ also satisfy property $ \mathcal{P} $.
\end{deff}

Let us now recall the 0-1 law from {\cite[Proposition~5.6]{LyonsPeres17}}, whose proof can easily be adapted in our context of Galton--Watson trees with random conductances verifying \eqref{REMARK-TInfiniteDescendants}.

\begin{thm}[[Proposition~5.6 of \cite{LyonsPeres17}]\label{THM-01Law}
	If $ \mathcal{P} $ is an inherited property, then
	\begin{equation*}
		\Pgw( \mathcal{T} \text{ has } \mathcal{P})\in \set{0,1}.
	\end{equation*}
\end{thm}

Let us now take advantage of the previous theorem in order to prove that $ h_* $ is constant. For this purpose, we define for each $h\in{\R}$ the property $ \mathcal{P}^h $ by saying  that  a tree $ {{T}} $ rooted at $x$ satisfies $ \mathcal{P}^h $ if $T_y$ is transient for all $y\in{T}$ and
\begin{align*}
	&\Pgff{{T}}{} \big(\big|E^{\ge h}_x\big|=∞\big)=0,
\end{align*}
where  for $y \in  T$ we denote by $ E^{\ge h}_{y} $ the connected component of $ y$ in $ \{z\in  T \,  : \,  φ_z\ge h\} $. We now need to prove that the property $\mathcal{P}^h$ is inherited, which has been done in the setting of unit conductances in \cite[Lemma~5.1]{AbacherliSznitman2018}.  For the reader's convenience we now present a proof in our setting inspired by \cite{Tassy2010}.

\begin{lemma}
	\label{ph:inherited}
	For each $h\in{\R},$ the property $ \mathcal{P}^h $ is inherited.
\end{lemma}
\begin{proof}

	Assume that $T$ is a tree rooted at $x$ verifying $\mathcal{P}^h.$ For any $y \in  T$ with $y\in{G_x^T}$ we have
	\begin{align*}
		\Pgff{ T}{}\big(\big| E^{\ge h}_x\big|=∞\big)
		&\ge \Pgff{ T}{}\big(\big| E^{\ge h}_y\cap {T}_y\big|=∞, φ_x\ge h\big)
		\ge  \Pgff{ T}{}\big(\big| E^{\ge h}_y\cap {T}_y\big|=∞\big)
		\Pgff{ T}{}( φ_x\ge h),     \end{align*}
	where the second inequality is a consequence of the finite dimensional FKG inequality for Gaussian fields, see \cite{Pitt1982}, and a classical limiting procedure.
	Since the second factor on the right-hand side is non-zero, $\Pgff{{T}}{}\big(\big|E^{\ge h}_x\big|=∞\big)=0$ implies for each $y\in{G_x^T}$
	\begin{equation*}
		\Pgff{ T}{}\big(\big| E^{\ge h}_y\cap {T}_y\big|=∞\big)=0.
	\end{equation*}

	What is left to do is to show that the previous equation holds also for the Gaussian free field on the subtree ${T}_y $. By disintegration, we observe that for $\lambda$-almost all $ b\in\R $ we have
	\begin{equation*}
		\Pgff{ T}{\big| E^{\ge h}_y\cap {{T}}_y\big|=∞\given φ_y=b}=0.
	\end{equation*}
	From the Markov property applied to the set $ K=\{y\} $, it follows that the restriction of the Gaussian free field under $ \Pgff{{T}}{\, \cdot\,|\,φ_y=b} $ to $ {{T}}_y $ has the same law as the Gaussian free field under $ \Pgff{{T}_y}{\, \cdot\,|\,φ_y=b} .$ Hence we obtain that for each $y\in{G_x^T}$ and $\lambda$-almost all $ b\in\R $ we have
	\begin{equation*}
		\Pgff{{T}_y}{\big| E^{\ge h}_y\big|=∞\given φ_y=b}=0.
	\end{equation*}
	Integrating again we obtain $ \Pgff{T_y}{\big| E^{\ge h}_y\big|=∞}=0,$ proving that $\mathcal{P}^h$ is inherited.
\end{proof}

With the previous 0-1 law and the inherited property $ \mathcal{P}^h $, we can prove Theorem~\ref{PROP-h*constant}.

\begin{proof}[Theorem~\ref{PROP-h*constant}]
	Since the property $ \mathcal{P}^h $ is inherited by Lemma~\ref{ph:inherited}, it follows from Theorem~\ref{THM-01Law} that $ \Pgw (\mathcal{T} \text{ has } \mathcal{P}_h)\in \set{0,1}$ for each $h\in{\R}.$  Moreover by Proposition~\ref{prop:GWtransient} and since $\mathcal{T}_x$ has the same law as $x\cdot\mathcal{T}$ under $\Pgw,$ see \eqref{REMARK-TInfiniteDescendants}, $\mathcal{T}_x$ is transient for all $x\in{\mathcal{T}}$ $\Pgw{}$-a.s. Hence for every $s\in{\mathbb{Q}} $, there exists an event $ A_s  $ with $ \Pgw (A_s)=1$ such that
	$ \mathcal{T}\mapsto\ind_{
		\{\Pgff{\mathcal{T}} {}(|E^{\ge s}_\emptyset|=∞)=0\}} $ is constant on $ A_s.  $ Thus on the event $ A:=\bigcap_{s\in \mathbb{Q}} A_s$, all the functions $ \ind_{\{ \Pgff{\mathcal{T}}{}(|E^{\ge s}_\emptyset|=∞)=0 \}},$ $s\in{\mathbb{Q}},$ are constant. Now, since the function $ h\mapsto \Pgff{\mathcal{T}}{}(|E^{\ge h}_\emptyset|=∞)$ is decreasing, the function
	\begin{equation*}
		\mathcal{T}\mapsto  \inf_{s\in\mathbb{Q}} \set{\Pgff{\mathcal{T}}{}(|E^{\ge s}_\emptyset|=∞)=0} =\inf_{h\in\mathbb{R}} \set{\Pgff{\mathcal{T}}{}(|E^{\ge h}_\emptyset|=∞)=0}
	\end{equation*}
	is well defined and constant on $ A ,$ and we can conclude by \eqref{DEF-h_*} and FKG inequality.
\end{proof}

Using an inherited property $ \mathcal{P}^u $ similar to before but for the vacant set $\mathcal{V}^u,$
one can also prove in our setting the constancy of the critical parameter $u_*$ for random interlacements.
\end{appendix}

\printbibliography

\end{document}